\documentclass[reqno]{amsart}
\usepackage{amsmath}
\usepackage{amsthm}
\usepackage{amssymb}
\usepackage{stmaryrd}
\usepackage{afterpage}
\usepackage{rotating}    
\usepackage{adjustbox}   
\usepackage{booktabs}
\usepackage[utf8]{inputenc}
\usepackage[T1]{fontenc}
\usepackage{lmodern}
\usepackage{eucal}
\pdfoutput=1 
\usepackage{microtype}
\usepackage{url}
\usepackage{enumitem}
\usepackage{tikz}
\usetikzlibrary{arrows.meta}
\usepackage{tikz-cd}
\usepackage{bbm}
\usepackage{longtable,booktabs}
\usepackage[all]{xy}
\usepackage{leftindex}
\usepackage{float}

\usepackage[letterpaper, left = 1.5in, right = 1.5in, top = 1.5in, bottom = 1.5in]{geometry}

\usepackage{pdflscape}

\usepackage{hyperref}

\usepackage{xcolor}
\hypersetup{
    colorlinks,
    linkcolor={red!50!black},
    citecolor={blue!50!black},
    urlcolor={blue!80!black}
}
\definecolor{chartgray}{gray}{0.4}
\definecolor{darkcyan}{rgb}{0, 0.7, 0.7}
\definecolor{darkblue}{rgb}{0, 0.65, 0}
\definecolor{truemagenta}{rgb}{1, 0, 1}

\usepackage{graphicx}

\usepackage{lipsum}     
\usepackage{booktabs} 
\usepackage{rotating}
\usepackage{tabularx}

\usepackage[colorinlistoftodos]{todonotes}

\usepackage{aliascnt}
\usepackage[capitalise,nameinlink]{cleveref}
\crefname{section}{Section}{Sections}
\crefname{subsection}{Subsection}{Subsections}

\theoremstyle{definition}
\newtheorem{theorem}{Theorem}[section]

\newaliascnt{prop}{theorem}
\newtheorem{prop}[prop]{Proposition}
\aliascntresetthe{prop}

\crefname{prop}{Proposition}{Propositions}
\Crefname{prop}{Proposition}{Propositions}

\newaliascnt{lemma}{theorem}
\newtheorem{lemma}[lemma]{Lemma}
\aliascntresetthe{lemma}

\crefname{lemma}{Lemma}{Lemmas}
\Crefname{lemma}{Lemma}{Lemmas}

\newaliascnt{cor}{theorem}
\newtheorem{cor}[cor]{Corollary}
\aliascntresetthe{cor}

\crefname{cor}{Corollary}{Corollaries}
\Crefname{cor}{Corollary}{Corollaries}

\newaliascnt{ex}{theorem}
\newtheorem{ex}[ex]{Example}
\aliascntresetthe{ex}

\crefname{ex}{Example}{Examples}
\Crefname{ex}{Example}{Examples}

\newtheorem{rmk}[theorem]{Remark}

\newtheorem{conjecture}[theorem]{Conjecture}

\newtheorem*{rmk*}{Remark}
\newtheorem*{ex*}{Example}
\newtheorem*{theorem*}{Theorem}
\newtheorem*{defn*}{Definition}
\newtheorem*{question*}{Question}
\newtheorem{thmx}{Theorem}

\crefformat{questionmanual}{#2Question#3}
\Crefformat{questionmanual}{#2Question#3}

\newcommand{\bfu}{\mathbf{u}}
\newcommand{\bfv}{\mathbf{v}}

\newcommand{\bbZ}{\mathbb{Z}}

\newcommand{\bbF}{\mathbb{F}}

\newcommand{\bbR}{\mathbb{R}}
\newcommand{\bbS}{\mathbb{S}}

\newcommand{\bbC}{\mathbb{C}}

\newcommand{\calA}{\mathcal{A}}

\newcommand{\Indet}{\mathrm{Indet}}

\newcommand{\Fib}{\operatorname{Fib}}

\newcommand{\Ext}{\operatorname{Ext}}
\newcommand{\Cof}{\operatorname{Cof}}

\newcommand{\im}{\operatorname{im}}

\newcommand\xqed[1]{%
  \leavevmode\unskip\penalty9999 \hbox{}\nobreak\hfill
  \quad\hbox{#1}}
\newcommand\tqed{\xqed{$\triangleleft$}}

\makeatletter
\DeclareRobustCommand{\tvdots}{%
  \vbox{\baselineskip4\p@\lineskiplimit\bbZ@\kern0\p@\hbox{.}\hbox{.}\hbox{.}}}
\makeatother

\makeatletter
\@namedef{subjclassname@2020}{\textup{2020} Mathematics Subject Classification}
\makeatother

\makeatletter
\newcommand{\raisemath}[1]{\mathpalette{\raisem@th{#1}}}
\newcommand{\raisem@th}[3]{\raisebox{#1}{$#2#3$}}
\makeatother

\begin{document}

\title[A Hurewicz theorem for equivariant homology by vector fields on spheres]{A Hurewicz theorem for $RO(C_2)$-graded equivariant homology governed by vector fields on spheres}

\author[M.~Guo]{Manyi Guo}
\address{Department of Mathematics, University of Washington, Seattle, WA 98195, USA}
\email{manyiguo@uw.edu}

\author[G.~Li]{Guchuan Li}
\address{School of Mathematical Sciences, Peking University, Beijing 100871, China}
\email{liguchuan@math.pku.edu.cn}

\author[Y.~Lu]{Yunze Lu}
\address{Shanghai Center for Mathematical Sciences, Fudan University, Shanghai, China}
\email{luyunze@fudan.edu.cn}

\author[S.~Ma]{Sihao Ma}
\address{Department of Mathematics, UCLA, Los Angeles, CA 90095-1555, USA}
\email{masihao@math.ucla.edu}

\author[Y.~Wu]{Yuchen Wu}
\address{Department of Mathematics, University of California San Diego, La Jolla, CA 92093, USA}
\email{yuw181@ucsd.edu}

\author[Z.~Xu]{Zhouli Xu}
\address{Department of Mathematics, UCLA, Los Angeles, CA 90095-1555, USA}
\email{xuzhouli@ucla.edu}

\author[A. J.~Yang]{Albert Jinghui Yang}
\address{Department of Mathematics, University of Pennsylvania, Philadelphia, PA 19104, USA}
\email{yangjh@sas.upenn.edu}

\author[S.~Zhang]{Shangjie Zhang}
\address{Department of Mathematics, University of California San Diego, La Jolla, CA 92093, USA}
\email{shz046@ucsd.edu}

\begin{abstract}
  We determine the $RO(C_2)$-graded Hurewicz images of the $C_2$-equivariant Eilenberg--MacLane spectra $H\underline{\mathbb F_2}$, $H\underline{\mathbb Z}$ and $H\underline{A}$, where $\underline{\mathbb F_2}$ and $\underline{\mathbb Z}$ denote the constant Mackey functors with values in $\mathbb F_2$ and $\mathbb Z$, respectively, and $\underline A$ denotes the Burnside Mackey functor. 
  
  Surprisingly, the answer is closely tied to the problem of vector fields on spheres \cite{adams1962vector}: the element $\frac{\theta}{\rho^k\tau^n}$ in the negative cone of the homotopy groups of $H\underline{\mathbb F_2}$ lies in the Hurewicz image if and only if $S^n$ admits $k$ linearly independent vector fields.

  Moreover, using the Generalized Leibniz Rule and the Generalized Mahowald Trick introduced by \cite{linwangxu2025lastkervaire}, we show that there are nonzero Adams differentials of arbitrary length supported by filtration-$0$ elements in the genuine $C_2$-equivariant Adams spectral sequence.
\end{abstract}

\maketitle

\tableofcontents

\section{Introduction}

In equivariant homotopy theory, Mackey functors $\underline{M}$ play a fundamental role analogous to that of abelian groups in classical algebraic topology. For a finite group $G$, they provide the natural coefficient systems for equivariant homology, and the associated Eilenberg--MacLane spectra $H\underline{M}$ represent homology theories graded not merely by integers, but by the \emph{real representation ring} $RO(G)$. This representation grading of $H\underline{M}$ is one of the key features of equivariant homotopy theory and is essential in the proof of the Gap Theorem and Periodicity Theorem of Hill--Hopkins--Ravenel's solution to the Kervaire invariant one problem \cite{hillhopkinsravenel2016nonexistence}. 

Given a multiplicative homology theory $E$, a natural and fundamental problem is to determine the image of the Hurewicz map
\[\iota_E \colon \pi_* S^0 \longrightarrow \pi_* E,\] 
or equivalently, which maps between spheres induce nontrivial maps on $E$-homology. In the $G$-equivariant setting, one similarly studies the $RO(G)$-graded Hurewicz map
\[\iota_{E'} \colon \pi^{G}_{\star} S^0 \longrightarrow \pi^{G}_{\star} E',\]
which encodes the $G$-equivariant maps between \emph{representation spheres} detected by $E'$. One of the objectives of this paper is to answer the following natural question:

\begin{question*}\phantomsection
\label[questionmanual]{mainquestion}
    Let $C_2$ be the cyclic group of order 2. What are the Hurewicz images of the $C_2$-equivariant Eilenberg--MacLane spectra $H\underline{M}$, where $\underline{M} = \underline{\mathbb F_2}, \  \underline{\mathbb Z}, \  \underline{A}$?
    \tqed
\end{question*}

Here $\underline{\mathbb F_2}$ and $\underline{\mathbb Z}$ are the constant Mackey functors with values in $\mathbb F_2$ and $\mathbb Z$, respectively, and $\underline A$ denotes the Burnside Mackey functor.

We begin with the case $\underline{\mathbb F_2}$. Recall from \cite{hukriz2001real} that the $RO(C_2)$-graded homotopy groups of $H\underline{\mathbb F_2}$ are given by
 \[\pi_\star^{C_2}H\underline{\bbF_2} =\bbF_2[\rho, \tau]\oplus NC, \qquad NC = \bigoplus_{k, n\geq 0}\bbF_2\left\{\frac{\theta}{\rho^k\tau^n}\right\}.\]
The summand $NC$ is commonly referred to as the \emph{negative cone}. 

To our great surprise, the equivariant \cref{mainquestion} is governed by the classical problem in differential topology that concerns the maximal number of linearly independent vector fields on a sphere. Building on lower bounds given by Clifford algebras \cite{eckmann1942gruppen}, Adams completely resolved this problem using stable homotopy theory \cite{adams1962vector}. Our first main result of this paper is that this classical geometric result appears naturally, and in an exact form, in the computation of the $RO(C_2)$-graded Hurewicz image of $H\underline{\bbF_2}$. 

\begin{thmx}\label{thmA}
    The element $\frac{\theta}{\rho^k\tau^n}$ lies in the Hurewicz image of $H\underline{\bbF_2}$ if and only if $S^n$ admits $k$ linearly independent vector fields.
    \tqed
\end{thmx}

A complete description of the Hurewicz image of $H\underline{\bbF_2}$ is given in \cref{mainthm} and \cref{fig:hurewiczim}.

We provide two proofs for \cref{thmA}. The first, given in \cref{sec:geoconstr}, is a direct geometric construction of explicit unstable $C_2$-equivariant maps between representation spheres. These maps are detected by $H\underline{\mathbb F_2}$-homology and cannot be further desuspended.

The second proof is more closely tied to the computational theme of the paper. It uses the recent result of the fourth author \cite{ma2026borel}, which establishes a systematic correspondence among the genuine $C_2$-equivariant Adams spectral sequence, the Borel $C_2$-equivariant Adams spectral sequence, and the classical Adams spectral sequences of stunted real projective spectra. From this perspective, \cref{thmA} characterizes the permanent cycles of filtration 0 in the negative cone of the \emph{genuine $C_2$-equivariant Adams spectral sequence}.

Determining the filtration-$0$ permanent cycles naturally leads to a finer analysis of the remaining classes that do not survive on the $0$-line. Our second main theorem proves the existence of genuine $C_2$-equivariant Adams differentials of arbitrary length supported by such classes.

\begin{thmx}\label{thmB}
    For \emph{any} $r\geq 2$, there exist nontrivial $d_r$-differentials supported by filtration-$0$ elements in the genuine $C_2$-equivariant Adams spectral sequence.
    \tqed
\end{thmx}

As a general structural result about the equivariant Adams spectral sequence, \cref{thmB} is surprising for at least four reasons.
\begin{enumerate}
    \item Classically, the $0$-line of the Adams spectral sequence of the sphere consists only of the unit, which is a permanent cycle and detects the identity map. In contrast, in the genuine $C_2$-equivariant Adams spectral sequence of the sphere, filtration-$0$ classes already support nontrivial differentials. This phenomenon is not observed in simpler $C_2$-spectra, such as $ko_{C_2}$ \cite{guillouhillisaksenravenel2020cohomology}.

    \item Evidence suggests that classical Adams differentials supported by elements of a fixed filtration (at least within every $Sq^0$-family) tend to have uniformly bounded length \cite{adams1960nonexistence,lishiwangxu2019hurewicz, lili2026adams}. The analogous behavior breaks down $C_2$-equivariantly: \cref{thmB} gives genuine $C_2$-equivariant Adams differentials of arbitrary length already at the lowest filtration.

    \item Arbitrarily long Adams differentials also occur in classical settings, but the proven cases are at odd primes with sources of increasing Adams filtrations \cite{andrews2015v1}. In \cref{cor:arblongdiffofass}, we show the analogous statement in the 2-primary classical Adams spectral sequence. By contrast, \cref{thmB} is different in nature: the differentials of increasing length have their sources all lie in filtration $0$.

    \item The most extensive existing computation of the genuine $C_2$-equivariant Adams spectral sequence due to Guillou--Isaksen \cite{guillouisaksen2024c2} covers \emph{stems} between $0$ and $25$ and \emph{coweights} between $-9$ and $7$. Within this range, every filtration-$0$ element is a permanent cycle, and the longest differentials that appear are $d_4$-differentials occurring in higher filtrations. This stands in sharp contrast to the phenomenon established in \cref{thmB}.
\end{enumerate}

\cref{thmB} follows as a corollary of \cref{arbitrarylongdiff}, which specifies the explicit sources of certain genuine $C_2$-equivariant Adams $d_r$-differentials for each $r\geq 4$. 
A key step in the proof of \cref{arbitrarylongdiff} uses the techniques of the Generalized Leibniz Rule and the Generalized Mahowald Trick, developed by Lin--Wang and the sixth author in their recent solution to the last Kervaire invariant one problem in dimension 126 \cite{linwangxu2025lastkervaire}.

We start by giving more context for the \cref{mainquestion} and stating the complete description of the Hurewicz images of $H\underline{\bbF_2}, H\underline{\bbZ}$, and $H\underline{A}$ in \cref{subsec:1.1}. In \cref{subsec:1.2}, we turn to the remaining filtration-$0$ classes in the genuine $C_2$-equivariant Adams spectral sequence and state the results from which \cref{thmB} follows. 

\subsection{The equivariant Hurewicz images}\label{subsec:1.1}

For a ring spectrum $E$, the Hurewicz map
\[
\iota_E \colon \pi_*S^0 \longrightarrow \pi_*E
\]
encodes the information in $\pi_*S^0$ that is \emph{detected} by $E$-homology, thereby providing a first approximation to the stable stems \cite{mahowald1982imageofj, douglasfrancishenriqueshill2014topological,hillhopkinsravenel2016nonexistence}. Consequently, determining the Hurewicz images of ring spectra has long been a central theme in homotopy theory. Recent progress in this direction includes computations for topological modular forms \cite{behrenshillhopkinsmahowald2020detecting,  brunerrognes2021adams, behrensmahowaldquigley2023hurewicz, bhattacharya2024new}, as well as for the real bordism theory and real Johnson--Wilson theories \cite{lishiwangxu2019hurewicz, hahnshi2020real}.

This paper studies the Hurewicz maps in the $C_2$-equivariant setting. Recall that the real representation ring $RO(C_2)$ has a basis consisting of the trivial representation and the sign representation $\sigma$. We use the bi-grading $(s, w)$ to indicate the $RO(C_2)$-degree $(s-w)+w\sigma$. We refer to $s$ as the \textit{stem}, $w$ as the \textit{weight}, and $s-w$ as the \textit{coweight}. 

Naturally graded by $RO(C_2)$, $C_2$-equivariant homology theories exhibit substantially richer structure than their classical counterparts; yet they remain computationally tractable \cite{greenlees1988stable, greenlees1988borelhomology, hukriz2001real, szymik2007equivariant}. Given a $C_2$-Mackey functor $\underline{M}$, there is an associated genuine $C_2$-equivariant Eilenberg--MacLane spectrum $H\underline{M}$. If $\underline{M}=\underline{\bbF_2}$, $H\underline{\bbF_2}$ is the base for the genuine $C_2$-equivariant Adams spectral sequence, which has been studied extensively by \cite{belmontguillouisaksen2021c2, guillouisaksen2020bredon, guillouhillisaksenravenel2020cohomology} and has been the primary tool in recent progress on computing the $RO(C_2)$-graded equivariant stable stems in a substantial range \cite{guillouisaksen2024c2}. If $\underline{M}=\underline{\bbZ}$, the constant Mackey functor with value $\bbZ$, $H\underline{\bbZ}$ is a key input to the slice spectral sequence and, in particular, to Hill--Hopkins--Ravenel's solution to the Kervaire invariant problem \cite{hillhopkinsravenel2016nonexistence}. If $\underline{M}=\underline{A}$, the Burnside Mackey functor, $H\underline{A}$ is the monoidal unit in the category of Green functors, which are equivariant analogs of rings. However, despite a wide range of applications in homotopy theory, it remains unclear which elements in the equivariant stable stems are detected by these Eilenberg--MacLane spectra.

The initial goal of this paper is to explicitly compute the Hurewicz image of $H\underline{\bbF_2}$. Specifically, recall that 
\[\pi_{*, *}^{C_2}H\underline{\bbF_2} =\bbF_2[\rho, \tau]\oplus NC, \qquad NC = \bigoplus_{n, k\geq 0}\bbF_2\left\{\frac{\theta}{\rho^k\tau^n}\right\}\]
where $|\rho| = (-1,-1)$, $|\tau| = (0,-1)$, and $|\theta| = (0,2)$. The first summand equals the coefficient ring of the $\bbR$-motivic Eilenberg–MacLane spectrum \cite{Voevodsky2003motivic}
\[{H\bbF_2^\bbR}_{*, *}=\bbF_2[\rho, \tau].\]
The generator $\frac{\theta}{\rho^k\tau^n}$ in the \emph{negative cone} has degree $(k, k+n+2)$. The $\bbF_2[\rho, \tau]$-module structure of $NC$ can be read off from the notation, together with the relations 
\[\rho \cdot \theta = \tau \cdot \theta =\theta \cdot \theta =0.\] 
\cref{HF2figure} depicts the additive structure of $H\underline{\bbF_2}_{*, *}$, where each black dot represents a copy of $\bbF_2$. 

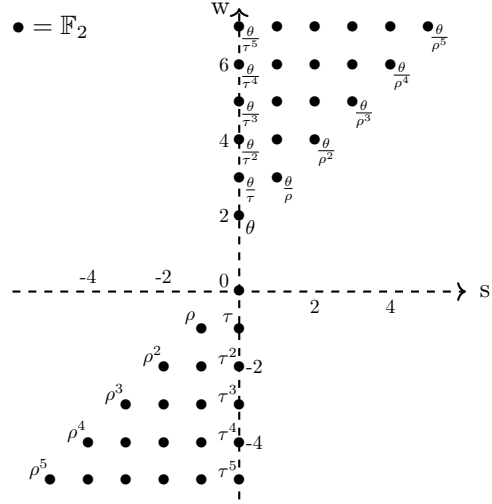
\begin{figure}[H]
        \centering
        \begin{tikzpicture}[scale=0.5]
\draw[thick,dashed, ->] (-6,0) -- (6,0);
\draw[thick,dashed, ->] (0,-5.5) -- (0, 7.5);

\node at (-5, 7){$\bullet= \bbF_2$};

\node at (6.5,0){s};
\node at (-0.5, 7.5){w};

\node at (-0.4, 0.3)[scale=0.75]{0};
\node at (-0.4, 2)[scale=0.75]{2};
\node at (-0.4, 4)[scale=0.75]{4};
\node at (-0.4, 6)[scale=0.75]{6};
\node at (0.4, -2)[scale=0.75]{-2};
\node at (0.4, -4)[scale=0.75]{-4};

\node at (2, -0.4)[scale=0.75]{2};
\node at (4, -0.4)[scale=0.75]{4};
\node at (-2, 0.4)[scale=0.75]{-2};
\node at (-4, 0.4)[scale=0.75]{-4};

\node at (0,2){$\bullet$};
\node at (0.3, 1.7)[scale=0.75]{$\theta$};
\node at (1,3){$\bullet$};
\node at (2,4){$\bullet$};
\node at (3,5){$\bullet$};
\node at (4,6){$\bullet$};
\node at (5,7){$\bullet$};

\node at (1.3, 2.7)[scale=0.75]{$\frac{\theta}{\rho}$};
\node at (2.3, 3.7)[scale=0.75]{$\frac{\theta}{\rho^2}$};
\node at (3.3, 4.7)[scale=0.75]{$\frac{\theta}{\rho^3}$};
\node at (4.3, 5.7)[scale=0.75]{$\frac{\theta}{\rho^4}$};
\node at (5.3, 6.7)[scale=0.75]{$\frac{\theta}{\rho^5}$};

\node at (0,3){$\bullet$};
\node at (1,4){$\bullet$};
\node at (2,5){$\bullet$};
\node at (3,6){$\bullet$};
\node at (4,7){$\bullet$};

\node at (0,4){$\bullet$};
\node at (1,5){$\bullet$};
\node at (2,6){$\bullet$};
\node at (3,7){$\bullet$};

\node at (0,5){$\bullet$};
\node at (1,6){$\bullet$};
\node at (2,7){$\bullet$};

\node at (0,6){$\bullet$};
\node at (1,7){$\bullet$};

\node at (0,7){$\bullet$};

\node at (0.3, 2.7)[scale=0.75]{$\frac{\theta}{\tau}$};
\node at (0.3, 3.7)[scale=0.75]{$\frac{\theta}{\tau^2}$};
\node at (0.3, 4.7)[scale=0.75]{$\frac{\theta}{\tau^3}$};
\node at (0.3, 5.7)[scale=0.75]{$\frac{\theta}{\tau^4}$};
\node at (0.3, 6.7)[scale=0.75]{$\frac{\theta}{\tau^5}$};

\node at (0,0){$\bullet$};
\node at (-1,-1){$\bullet$};
\node at (-2,-2){$\bullet$};
\node at (-3,-3){$\bullet$};
\node at (-4,-4){$\bullet$};
\node at (-5,-5){$\bullet$};

\node at (-1.3, -0.7)[scale=0.75]{$\rho$};
\node at (-2.3, -1.7)[scale=0.75]{$\rho^2$};
\node at (-3.3, -2.7)[scale=0.75]{$\rho^3$};
\node at (-4.3, -3.7)[scale=0.75]{$\rho^4$};
\node at (-5.3, -4.7)[scale=0.75]{$\rho^5$};

\node at (0,-1){$\bullet$};
\node at (-1,-2){$\bullet$};
\node at (-2,-3){$\bullet$};
\node at (-3,-4){$\bullet$};
\node at (-4,-5){$\bullet$};

\node at (0,-2){$\bullet$};
\node at (-1,-3){$\bullet$};
\node at (-2,-4){$\bullet$};
\node at (-3,-5){$\bullet$};

\node at (0,-3){$\bullet$};
\node at (-1,-4){$\bullet$};
\node at (-2,-5){$\bullet$};

\node at (0,-4){$\bullet$};
\node at (-1,-5){$\bullet$};

\node at (0,-5){$\bullet$};

\node at (-0.3, -0.7)[scale=0.75]{$\tau$};
\node at (-0.3, -1.7)[scale=0.75]{$\tau^2$};
\node at (-0.3, -2.7)[scale=0.75]{$\tau^3$};
\node at (-0.3, -3.7)[scale=0.75]{$\tau^4$};
\node at (-0.3, -4.7)[scale=0.75]{$\tau^5$};

\end{tikzpicture}
        \caption{$\pi_{s, w}^{C_2}H\underline{\bbF_2}$ }
        \label{HF2figure}
    \end{figure}
    
\begin{rmk}
The notation $\theta$ is also used in \cite{lishiwangxu2019hurewicz, behrensshah2020c2, may2020structure}. In other references, the same element is sometimes denoted by $\frac{\gamma}{\tau}$; see, for example, \cite{guillouhillisaksenravenel2020cohomology, guillouisaksen2024c2, belmontguillouisaksen2021c2}.
\end{rmk}

Furthermore, for any integer $n\geq 0$, write $n=(2a+1)2^{c+4d}$ with $a, c, d \in \bbZ$ and $0\leq c\leq 3$. Let 
\[\psi(n):=2^c+8d\] denote the $n$-th Radon--Hurwitz number \cite{eckmann1942gruppen}; in particular, Adams proved that the maximal number of continuous linearly independent vector fields on $S^{n-1}$ is $\psi(n)-1$ in \cite{adams1962vector}. We list the first few values of $\psi(n)$ for reference. 
\[\begin{tabular}{@{}c|c|c|c|c|c|c|c|c|c|c|c|c|c|c|c|c@{}}

$n=$ &  2 & 4 &6 &8 &10 &12 &14 &16 &18 &20 &22 &24 &26 &28 &30 &32                                                                      \\\hline 
$\psi(n)$  & 2 &4&2&8&2&4&2&9&   2 &4&2&8&2&4&2&10   
\end{tabular}\]

\begin{figure}
    \centering
    \begin{tikzpicture}[scale=0.195, h0/.style={draw={gray}},]
    \draw[h0, ->] (0,-2) -- (65, -2);
    \node at (66, -2){$s$};
    \draw[h0, ->] (0,-2) -- (0, 65);
    \node at (0, 66){$w$};

    \foreach \z in {0}
		\draw (\z, \z) node[scale=0.5]{\color{blue}{$\bullet$}};
    \foreach \z in {1, ..., 63}
		\draw (\z, \z) node[scale=0.4]{$\color{gray}{\circ}$};
        
    \foreach \z in {0, 1}
		\draw (\z, \z+1) node[scale=0.5]{\color{blue}{$\bullet$}};
    \foreach \z in {2, ..., 62}
		\draw (\z, \z+1) node[scale=0.4]{$\color{gray}{\circ}$};
        
    \foreach \z in {0}
		\draw (\z, \z+2) node[scale=0.5]{\color{blue}{$\bullet$}};
    \foreach \z in {1, ..., 61}
		\draw (\z, \z+2) node[scale=0.4]{$\color{gray}{\circ}$};

    \foreach \z in {0, 1, 2, 3}
		\draw (\z, \z+3) node[scale=0.5]{\color{blue}{$\bullet$}};
    \foreach \z in {4, ..., 60}
		\draw (\z, \z+3) node[scale=0.4]{$\color{gray}{\circ}$};

    \foreach \z in {0}
		\draw (\z, \z+4) node[scale=0.5]{\color{blue}{$\bullet$}};
    \foreach \z in {1, ..., 59}
		\draw (\z, \z+4) node[scale=0.4]{$\color{gray}{\circ}$};

    \foreach \z in {0, 1}
		\draw (\z, \z+5) node[scale=0.5]{\color{blue}{$\bullet$}};
    \foreach \z in {2, ..., 58}
		\draw (\z, \z+5) node[scale=0.4]{$\color{gray}{\circ}$};
        
    \foreach \z in {0}
		\draw (\z, \z+6) node[scale=0.5]{\color{blue}{$\bullet$}};
    \foreach \z in {1, ..., 57}
		\draw (\z, \z+6) node[scale=0.4]{$\color{gray}{\circ}$};

    \foreach \z in {0, ..., 7}
		\draw (\z, \z+7) node[scale=0.5]{\color{blue}{$\bullet$}};
    \foreach \z in {8, ..., 56}
		\draw (\z, \z+7) node[scale=0.4]{$\color{gray}{\circ}$};

    \foreach \z in {0}
		\draw (\z, \z+8) node[scale=0.5]{\color{blue}{$\bullet$}};
    \foreach \z in {1, ..., 55}
		\draw (\z, \z+8) node[scale=0.4]{$\color{gray}{\circ}$};
        
    \foreach \z in {0, 1}
		\draw (\z, \z+9) node[scale=0.5]{\color{blue}{$\bullet$}};
    \foreach \z in {2, ..., 54}
		\draw (\z, \z+9) node[scale=0.4]{$\color{gray}{\circ}$};
        
    \foreach \z in {0}
		\draw (\z, \z+10) node[scale=0.5]{\color{blue}{$\bullet$}};
    \foreach \z in {1, ..., 53}
		\draw (\z, \z+10) node[scale=0.4]{$\color{gray}{\circ}$};

    \foreach \z in {0, 1, 2, 3}
		\draw (\z, \z+11) node[scale=0.5]{\color{blue}{$\bullet$}};
    \foreach \z in {4, ..., 52}
		\draw (\z, \z+11) node[scale=0.4]{$\color{gray}{\circ}$};

    \foreach \z in {0}
		\draw (\z, \z+12) node[scale=0.5]{\color{blue}{$\bullet$}};
    \foreach \z in {1, ...,51}
		\draw (\z, \z+12) node[scale=0.4]{$\color{gray}{\circ}$};

    \foreach \z in {0, 1}
		\draw (\z, \z+13) node[scale=0.5]{\color{blue}{$\bullet$}};
    \foreach \z in {2, ..., 50}
		\draw (\z, \z+13) node[scale=0.4]{$\color{gray}{\circ}$};
        
    \foreach \z in {0}
		\draw (\z, \z+14) node[scale=0.5]{\color{blue}{$\bullet$}};
    \foreach \z in {1, ..., 49}
		\draw (\z, \z+14) node[scale=0.4]{$\color{gray}{\circ}$};

    \foreach \z in {0, ..., 8}
		\draw (\z, \z+15) node[scale=0.5]{\color{blue}{$\bullet$}};
    \foreach \z in {9, ..., 15}
		\draw (\z, \z+15) node[scale=0.5]{\color{red}{$\bullet$}};
    \foreach \z in {16, ..., 48}
		\draw (\z, \z+15) node[scale=0.4]{$\color{gray}{\circ}$};

    \foreach \z in {0}
		\draw (\z, \z+16) node[scale=0.5]{\color{blue}{$\bullet$}};
    \foreach \z in {1, ..., 47}
		\draw (\z, \z+16) node[scale=0.4]{$\color{gray}{\circ}$};
        
    \foreach \z in {0, 1}
		\draw (\z, \z+17) node[scale=0.5]{\color{blue}{$\bullet$}};
    \foreach \z in {2, ..., 46}
		\draw (\z, \z+17) node[scale=0.4]{$\color{gray}{\circ}$};
        
    \foreach \z in {0}
		\draw (\z, \z+18) node[scale=0.5]{\color{blue}{$\bullet$}};
    \foreach \z in {1, ..., 45}
		\draw (\z, \z+18) node[scale=0.4]{$\color{gray}{\circ}$};

    \foreach \z in {0, 1, 2, 3}
		\draw (\z, \z+19) node[scale=0.5]{\color{blue}{$\bullet$}};
    \foreach \z in {4, ..., 44}
		\draw (\z, \z+19) node[scale=0.4]{$\color{gray}{\circ}$};

    \foreach \z in {0}
		\draw (\z, \z+20) node[scale=0.5]{\color{blue}{$\bullet$}};
    \foreach \z in {1, ..., 43}
		\draw (\z, \z+20) node[scale=0.4]{$\color{gray}{\circ}$};

    \foreach \z in {0, 1}
		\draw (\z, \z+21) node[scale=0.5]{\color{blue}{$\bullet$}};
    \foreach \z in {2, ..., 42}
		\draw (\z, \z+21) node[scale=0.4]{$\color{gray}{\circ}$};
        
    \foreach \z in {0}
		\draw (\z, \z+22) node[scale=0.5]{\color{blue}{$\bullet$}};
    \foreach \z in {1, ..., 41}
		\draw (\z, \z+22) node[scale=0.4]{$\color{gray}{\circ}$};

    \foreach \z in {0, ..., 7}
		\draw (\z, \z+23) node[scale=0.5]{\color{blue}{$\bullet$}};
    \foreach \z in {8, ..., 40}
		\draw (\z, \z+23) node[scale=0.4]{$\color{gray}{\circ}$};

    \foreach \z in {0}
		\draw (\z, \z+24) node[scale=0.5]{\color{blue}{$\bullet$}};
    \foreach \z in {1, ..., 39}
		\draw (\z, \z+24) node[scale=0.4]{$\color{gray}{\circ}$};
        
    \foreach \z in {0, 1}
		\draw (\z, \z+25) node[scale=0.5]{\color{blue}{$\bullet$}};
    \foreach \z in {2, ..., 38}
		\draw (\z, \z+25) node[scale=0.4]{$\color{gray}{\circ}$};
        
    \foreach \z in {0}
		\draw (\z, \z+26) node[scale=0.5]{\color{blue}{$\bullet$}};
    \foreach \z in {1, ..., 37}
		\draw (\z, \z+26) node[scale=0.4]{$\color{gray}{\circ}$};

    \foreach \z in {0, 1, 2, 3}
		\draw (\z, \z+27) node[scale=0.5]{\color{blue}{$\bullet$}};
    \foreach \z in {4, ..., 36}
		\draw (\z, \z+27) node[scale=0.4]{$\color{gray}{\circ}$};

    \foreach \z in {0}
		\draw (\z, \z+28) node[scale=0.5]{\color{blue}{$\bullet$}};
    \foreach \z in {1, ..., 35}
		\draw (\z, \z+28) node[scale=0.4]{$\color{gray}{\circ}$};

    \foreach \z in {0, 1}
		\draw (\z, \z+29) node[scale=0.5]{\color{blue}{$\bullet$}};
    \foreach \z in {2, ..., 34}
		\draw (\z, \z+29) node[scale=0.4]{$\color{gray}{\circ}$};
        
    \foreach \z in {0}
		\draw (\z, \z+30) node[scale=0.5]{\color{blue}{$\bullet$}};
    \foreach \z in {1, ..., 33}
		\draw (\z, \z+30) node[scale=0.4]{$\color{gray}{\circ}$};

    \foreach \z in {0, ...,9}
		\draw (\z, \z+31) node[scale=0.5]{\color{blue}{$\bullet$}};
    \foreach \z in {10, ..., 31}
		\draw (\z, \z+31) node[scale=0.5]{\color{red}{$\bullet$}};
    \foreach \z in {32}
		\draw (\z, \z+31) node[scale=0.4]{$\color{gray}{\circ}$};
    
    \foreach \z in {0}
		\draw (\z, \z+32) node[scale=0.5]{\color{blue}{$\bullet$}};
    \foreach \z in {1, ..., 31}
		\draw (\z, \z+32) node[scale=0.4]{$\color{gray}{\circ}$};
        
    \foreach \z in {0, 1}
		\draw (\z, \z+33) node[scale=0.5]{\color{blue}{$\bullet$}};
    \foreach \z in {2, ..., 30}
		\draw (\z, \z+33) node[scale=0.4]{$\color{gray}{\circ}$};
        
    \foreach \z in {0}
		\draw (\z, \z+34) node[scale=0.5]{\color{blue}{$\bullet$}};
    \foreach \z in {1, ..., 29}
		\draw (\z, \z+34) node[scale=0.4]{$\color{gray}{\circ}$};

    \foreach \z in {0, 1, 2, 3}
		\draw (\z, \z+35) node[scale=0.5]{\color{blue}{$\bullet$}};
    \foreach \z in {4, ..., 28}
		\draw (\z, \z+35) node[scale=0.4]{$\color{gray}{\circ}$};

    \foreach \z in {0}
		\draw (\z, \z+36) node[scale=0.5]{\color{blue}{$\bullet$}};
    \foreach \z in {1, ..., 27}
		\draw (\z, \z+36) node[scale=0.4]{$\color{gray}{\circ}$};

    \foreach \z in {0, 1}
		\draw (\z, \z+37) node[scale=0.5]{\color{blue}{$\bullet$}};
    \foreach \z in {2, ..., 26}
		\draw (\z, \z+37) node[scale=0.4]{$\color{gray}{\circ}$};
        
    \foreach \z in {0}
		\draw (\z, \z+38) node[scale=0.5]{\color{blue}{$\bullet$}};
    \foreach \z in {1, ..., 25}
		\draw (\z, \z+38) node[scale=0.4]{$\color{gray}{\circ}$};

    \foreach \z in {0, ..., 7}
		\draw (\z, \z+39) node[scale=0.5]{\color{blue}{$\bullet$}};
    \foreach \z in {8, ..., 24}
		\draw (\z, \z+39) node[scale=0.4]{$\color{gray}{\circ}$};

    \foreach \z in {0}
		\draw (\z, \z+40) node[scale=0.5]{\color{blue}{$\bullet$}};
    \foreach \z in {1, ..., 23}
		\draw (\z, \z+40) node[scale=0.4]{$\color{gray}{\circ}$};
        
    \foreach \z in {0, 1}
		\draw (\z, \z+41) node[scale=0.5]{\color{blue}{$\bullet$}};
    \foreach \z in {2, ..., 22}
		\draw (\z, \z+41) node[scale=0.4]{$\color{gray}{\circ}$};
        
    \foreach \z in {0}
		\draw (\z, \z+42) node[scale=0.5]{\color{blue}{$\bullet$}};
    \foreach \z in {1, ..., 21}
		\draw (\z, \z+42) node[scale=0.4]{$\color{gray}{\circ}$};

    \foreach \z in {0, 1, 2, 3}
		\draw (\z, \z+43) node[scale=0.5]{\color{blue}{$\bullet$}};
    \foreach \z in {4, ..., 20}
		\draw (\z, \z+43) node[scale=0.4]{$\color{gray}{\circ}$};

    \foreach \z in {0}
		\draw (\z, \z+44) node[scale=0.5]{\color{blue}{$\bullet$}};
    \foreach \z in {1, ..., 19}
		\draw (\z, \z+44) node[scale=0.4]{$\color{gray}{\circ}$};

    \foreach \z in {0, 1}
		\draw (\z, \z+45) node[scale=0.5]{\color{blue}{$\bullet$}};
    \foreach \z in {2, ..., 18}
		\draw (\z, \z+45) node[scale=0.4]{$\color{gray}{\circ}$};
        
    \foreach \z in {0}
		\draw (\z, \z+46) node[scale=0.5]{\color{blue}{$\bullet$}};
    \foreach \z in {1, ..., 17}
		\draw (\z, \z+46) node[scale=0.4]{$\color{gray}{\circ}$};

    \foreach \z in {0, ..., 8}
		\draw (\z, \z+47) node[scale=0.5]{\color{blue}{$\bullet$}};
    \foreach \z in {9, ..., 15}
		\draw (\z, \z+47) node[scale=0.5]{\color{red}{$\bullet$}};
    \foreach \z in {16}
		\draw (\z, \z+47) node[scale=0.4]{$\color{gray}{\circ}$};

    \foreach \z in {0}
		\draw (\z, \z+48) node[scale=0.5]{\color{blue}{$\bullet$}};
    \foreach \z in {1, ..., 15}
		\draw (\z, \z+48) node[scale=0.4]{$\color{gray}{\circ}$};
        
    \foreach \z in {0, 1}
		\draw (\z, \z+49) node[scale=0.5]{\color{blue}{$\bullet$}};
    \foreach \z in {2, ..., 14}
		\draw (\z, \z+49) node[scale=0.4]{$\color{gray}{\circ}$};
        
    \foreach \z in {0}
		\draw (\z, \z+50) node[scale=0.5]{\color{blue}{$\bullet$}};
    \foreach \z in {1, ..., 13}
		\draw (\z, \z+50) node[scale=0.4]{$\color{gray}{\circ}$};

    \foreach \z in {0, 1, 2, 3}
		\draw (\z, \z+51) node[scale=0.5]{\color{blue}{$\bullet$}};
    \foreach \z in {4, ..., 12}
		\draw (\z, \z+51) node[scale=0.4]{$\color{gray}{\circ}$};

    \foreach \z in {0}
		\draw (\z, \z+52) node[scale=0.5]{\color{blue}{$\bullet$}};
    \foreach \z in {1, ..., 11}
		\draw (\z, \z+52) node[scale=0.4]{$\color{gray}{\circ}$};

    \foreach \z in {0, 1}
		\draw (\z, \z+53) node[scale=0.5]{\color{blue}{$\bullet$}};
    \foreach \z in {2, ..., 10}
		\draw (\z, \z+53) node[scale=0.4]{$\color{gray}{\circ}$};
        
    \foreach \z in {0}
		\draw (\z, \z+54) node[scale=0.5]{\color{blue}{$\bullet$}};
    \foreach \z in {1, ..., 9}
		\draw (\z, \z+54) node[scale=0.4]{$\color{gray}{\circ}$};

    \foreach \z in {0, ..., 7}
		\draw (\z, \z+55) node[scale=0.5]{\color{blue}{$\bullet$}};
    \foreach \z in {8}
		\draw (\z, \z+55) node[scale=0.4]{$\color{gray}{\circ}$};

    \foreach \z in {0}
		\draw (\z, \z+56) node[scale=0.5]{\color{blue}{$\bullet$}};
    \foreach \z in {1, ..., 7}
		\draw (\z, \z+56) node[scale=0.4]{$\color{gray}{\circ}$};
        
    \foreach \z in {0, 1}
		\draw (\z, \z+57) node[scale=0.5]{\color{blue}{$\bullet$}};
    \foreach \z in {2, ..., 6}
		\draw (\z, \z+57) node[scale=0.4]{$\color{gray}{\circ}$};
        
    \foreach \z in {0}
		\draw (\z, \z+58) node[scale=0.5]{\color{blue}{$\bullet$}};
    \foreach \z in {1, ..., 5}
		\draw (\z, \z+58) node[scale=0.4]{$\color{gray}{\circ}$};

    \foreach \z in {0, 1, 2, 3}
		\draw (\z, \z+59) node[scale=0.5]{\color{blue}{$\bullet$}};
    \foreach \z in {4}
		\draw (\z, \z+59) node[scale=0.4]{$\color{gray}{\circ}$};

    \foreach \z in {0}
		\draw (\z, \z+60) node[scale=0.5]{\color{blue}{$\bullet$}};
    \foreach \z in {1, ..., 3}
		\draw (\z, \z+60) node[scale=0.4]{$\color{gray}{\circ}$};

    \foreach \z in {0,1}
		\draw (\z, \z+61) node[scale=0.5]{\color{blue}{$\bullet$}};
    \foreach \z in {2}
		\draw (\z, \z+61) node[scale=0.4]{$\color{gray}{\circ}$};
        
    \foreach \z in {0}
		\draw (\z, \z+62) node[scale=0.5]{\color{blue}{$\bullet$}};
    \foreach \z in {1}
		\draw (\z, \z+62) node[scale=0.4]{$\color{gray}{\circ}$};

    \foreach \z in {0}
		\draw (\z, \z+63) node[scale=0.5]{\color{blue}{$\bullet$}};

    \draw (-0.8, 0) -- (-2.2, 0);
    \node at (-3.5,0){$\theta$};
    
    \draw (-0.8, 7) -- (-2.2, 7);
    \node at (-3.5,7){$\frac{\theta}{\tau^7}$};

    \draw (-0.8, 15) -- (-2.2, 15);
    \node at (-3.5, 15){$\frac{\theta}{\tau^{15}}$};

    \draw (-0.8, 23) -- (-2.2, 23);
    \node at (-3.5,23){$\frac{\theta}{\tau^{23}}$};

    \draw (-0.8, 31) -- (-2.2, 31);
    \node at (-3.5,31){$\frac{\theta}{\tau^{31}}$};

    \draw (-0.8, 39) -- (-2.2, 39);
    \node at (-3.5,39){$\frac{\theta}{\tau^{39}}$};

    \draw (-0.8, 47) -- (-2.2, 47);
    \node at (-3.5,47){$\frac{\theta}{\tau^{47}}$};

    \draw (-0.8, 55) -- (-2.2, 55);
    \node at (-3.5,55){$\frac{\theta}{\tau^{55}}$};

    \node at (55, 7)[scale=0.82]{$\color{blue}{\bullet}$ --- Hurewicz image};

    \node at (53.5, 9)[scale=0.82]{$\color{red}{\bullet}\ \color{black}{\&}\ \color{blue}{\bullet}$ --- $\Ext_{C_2}^{*, 0, *}$};

\end{tikzpicture}
    \caption{The Hurewicz image of $H\underline{\bbF_2}$ in the negative cone}
    \label{fig:hurewiczim}
\end{figure}
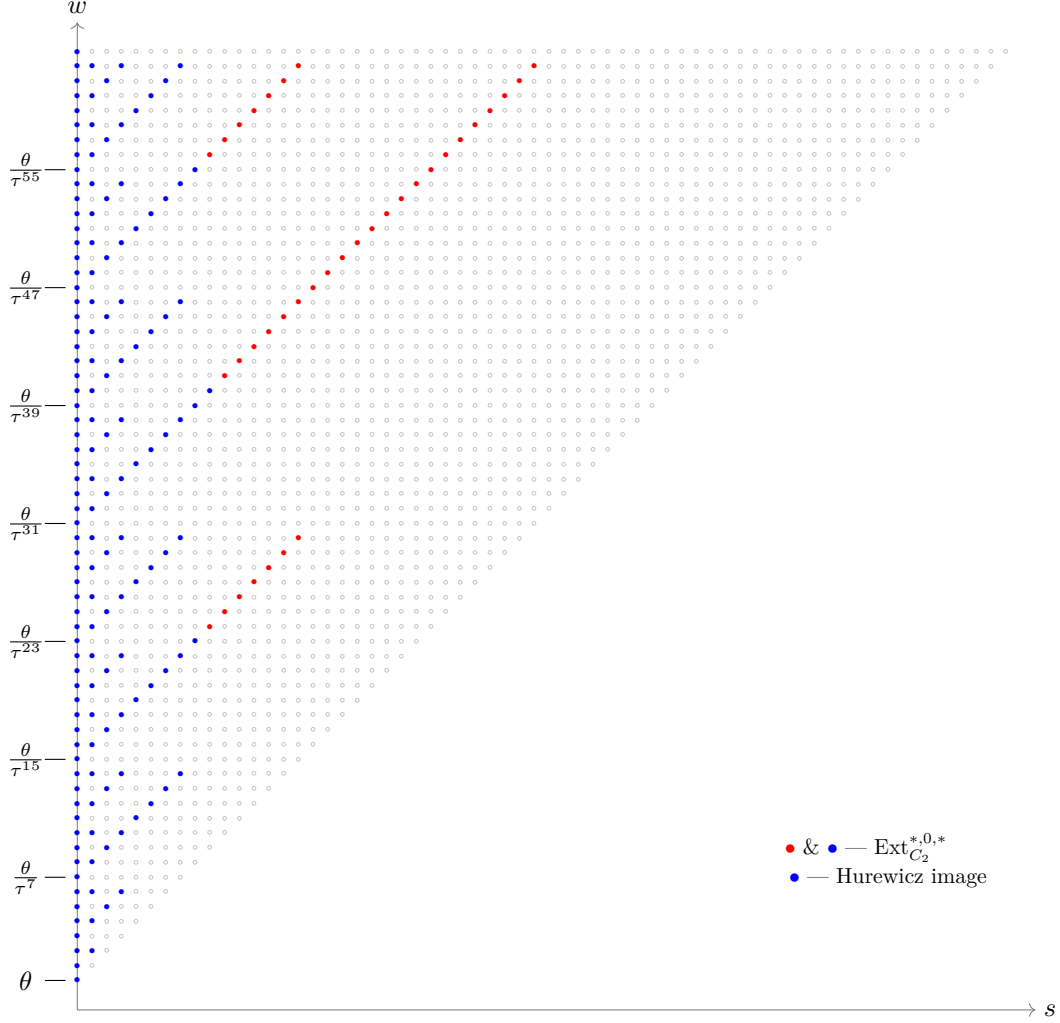

The Hurewicz image of $H\underline{\bbF_2}$ can be described as follows.

\begin{theorem}[\cref{homotopyHurewiczim}]\label{mainthm}
    The $RO(C_2)$-graded Hurewicz image of $H\underline{\bbF_2}$ consists of
\[\im (\iota_{H\underline{\bbF_2}})_{s, w}=\begin{cases}
        \bbF_2\{\rho^{-s}\} \quad &\mathrm{if}\ s=w\leq 0,\\
        \bbF_2\{\frac{\theta}{\rho^s \tau^{w-s-2} }\} \quad &\mathrm{if}\ 0\leq s< \psi(w-s-1) \ \mathrm{and}\ s-w\leq -2,\\
        0\quad &\mathrm{otherwise. \rlap{\hspace{15em} \tqed}}
        \end{cases} \]

\end{theorem}

For $s=w\leq 0$, $\rho$ is the image of $a_\sigma$, the Euler class of the $C_2$-sign representation $\sigma$. \cref{thmA} follows immediately from the case $s-w\leq -2$ (i.e. in the negative cone). Note that the bidegree of $\frac{\theta}{\rho^s\tau^{w-s-2}}$ is exactly $(s,w)$. A diagram of the Hurewicz image of $H\underline{\bbF_2}$ in the negative cone is shown in \cref{fig:hurewiczim}.

We prove \cref{mainthm} by first computing $\Ext_{C_2}^{*, 0, *}$, the $0$-line of the $E_2$-page of the genuine $C_2$-equivariant Adams spectral sequence, via the \emph{$\rho$-Bockstein spectral sequence}. The elements at degree $(s, 0, w)$ of $\Ext_{C_2}$ with $s-w\leq -2$ are represented by the solid (red and blue) dots in \cref{fig:hurewiczim}. The Hurewicz image of $H\underline{\bbF_2}$ consists precisely of surviving permanent cycles on the $0$-line.  Using the correspondence between the genuine Adams spectral sequence, the Borel Adams spectral sequence, and the classical Adams spectral sequence of stunted real projective spaces \cite{ma2026borel}, we then determine which $0$-line elements support Adams differentials. It turns out that these differentials are governed by the attaching maps of the corresponding stunted real projective spaces, and the necessary information follows as a corollary of Adams' work on the vector field problem \cite{adams1962vector}. This explains the appearance of the Radon--Hurwitz number $\psi(n)$ in \cref{mainthm}. The strategy of the proof of \cref{mainthm} is summarized in \cref{fig:roadmap}.

\begin{figure}[H]
\centering
\begin{tikzpicture}[
    node distance=1.6cm and 1.4cm,
    box/.style={
        draw,
        rounded corners,
        align=center,
        inner sep=4pt,
        text width=2cm
    },
    ar/.style={
        -{Stealth[length=2.2mm]},
        thick
    },
    arr/.style={
    {Stealth[length=2.2mm]}-{Stealth[length=2.2mm]},
    thick
    }  
]
\small
\node[box] (genuine) {Genuine $C_2$-ASS};

\node[box, right=of genuine] (borel) {Borel $C_2$-ASS};

\node[box, right=of borel] (rp) {Classical ASS of $\bbR P^\infty_n$};

\node[box, right=of rp] (geom) {vector fields on spheres};

\draw[arr] (genuine) --  (borel);

\draw[arr] (borel) --  (rp);

\draw[ar] (geom) -- node[above, align=center, scale=0.7] {differentials} (rp) ;

\end{tikzpicture}

\caption{Roadmap for the proof of \cref{mainthm}}
\label{fig:roadmap}
\end{figure}
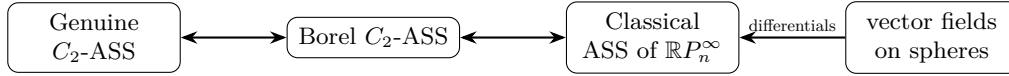

In \cref{fig:hurewiczim}, red dots are the classes that support Adams differentials, while blue dots are the elements in the Hurewicz image. Hollow circles are classes that do not survive to $\Ext_{C_2}$.

As a direct application of \cref{mainthm} and the results of Belmont, the sixth author and the eighth author \cite{belmontxuzhang2024reduced}, we also determine the Hurewicz image of $H\underline{\bbZ}$ and $H\underline{A}$.

\begin{theorem}[\cref{HurewiczofHZHA}]\label{mainHurewiczofHZHA}
    ~\begin{enumerate}
    \item The $RO(C_2)$-graded Hurewicz image of $H\underline{\bbZ}$ consists of
    \[\im (\iota_{H\underline{\bbZ}} )_{s, w}=\begin{cases}
        \bbZ \quad &\mathrm{if}\ s=0\ \mathrm{and}\ w=2k\ \mathrm{for}\ k \in \bbZ,\\
        \bbF_2 \quad &\mathrm{if}\ s=w<0,\\
        \bbF_2 \quad &\mathrm{if}\ 0\leq s< \psi(w-s-1) \ \mathrm{and}\ s-w=2k-1\ \mathrm{for\ } k\leq -1,\\
        0\quad &\mathrm{otherwise.}
    \end{cases}\]
    \item The $RO(C_2)$-graded Hurewicz image of $H\underline{A}$ consists of
    \[\im (\iota_{H\underline{A}} )_{s, w}=\begin{cases}
        A(C_2)\quad &\mathrm{if}\ s=w=0,\\
        \bbZ \quad &\mathrm{if}\ s=0\ \mathrm{and}\ w=-2k\ \mathrm{for}\ k\neq 0,\\
        \bbZ \quad &\mathrm{if}\ s=w\neq 0,\\
        \bbF_2 \quad &\mathrm{if}\ 0\leq s< \psi(w-s-1) \ \mathrm{and}\ s-w=2k-1\ \mathrm{for\ } k\leq -1,\\
        0\quad &\mathrm{otherwise. \rlap{\hspace{20.4em} \tqed}}
    \end{cases}\]
\end{enumerate}

\end{theorem}

The generators of these Hurewicz images are specified in \cref{sec:HZHA}.

\subsection{Genuine $C_2$-Adams differentials on the $0$-line}\label{subsec:1.2}
After determining all the filtration-$0$ permanent cycles in the genuine $C_2$-equivariant Adams spectral sequence, we are naturally led to study which Adams differentials are supported by the non-survival classes on the $0$-line.

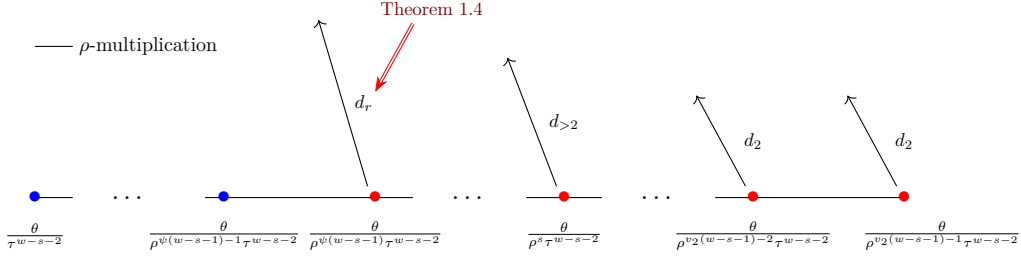
\begin{figure}
    \centering
    \begin{tikzpicture}[scale=0.5]
    \usetikzlibrary{arrows.meta}
    
    \node at (16, 0){$\color{red}{\bullet}$};
    \node at (20, 0){$\color{red}{\bullet}$};
    \draw[->] (19.8, 0.3) -- (18.5, 2.7);
    \node at (20, 1.5)[scale=0.7]{$d_2$};
    \draw[->] (15.8, 0.3) -- (14.5, 2.7);
    \node at (16, 1.5)[scale=0.7]{$d_2$};
    \draw (19.9,0) -- (16.1, 0);
    \node at (13.5, 0){$\dots$};
    \draw (15.9, 0) -- (15, 0);
    \node at (11, 0){$\color{red}{\bullet}$};
    \draw (11.1, 0) -- (12, 0);
    \draw[->] (10.8, 0.3) -- (9.5, 3.7);
    \node at (11, 2)[scale=0.7]{$d_{>2}$};
    \draw (10.9, 0) -- (10, 0);
    \node at (8.5, 0){$\dots$};
    \draw (6.1, 0) -- (7, 0);
    \node at (6, 0){$\color{red}{\bullet}$};
    \node at (2, 0){$\color{blue}{\bullet}$};
    \draw (5.9,0) -- (2.1, 0);
    \draw[->] (5.8, 0.3) -- (4.5, 4.7);
    \node at (5.7, 2.5)[scale=0.7]{$d_{r}$};

    \node at (21, -1)[scale=0.75]{$\frac{\theta}{\rho^{v_2(w-s-1)-1}\tau^{w-s-2}}$};
    \node at (16, -1)[scale=0.75]{$\frac{\theta}{\rho^{v_2(w-s-1)-2}\tau^{w-s-2}}$};
    \node at (11, -1)[scale=0.75]{$\frac{\theta}{\rho^{s}\tau^{w-s-2}}$};
    \node at (6, -1)[scale=0.75]{$\frac{\theta}{\rho^{\psi(w-s-1)}\tau^{w-s-2}}$};
    \node at (2, -1)[scale=0.75]{$\frac{\theta}{\rho^{\psi(w-s-1)-1}\tau^{w-s-2}}$};
    \node at (-3, -1)[scale=0.75]{$\frac{\theta}{\tau^{w-s-2}}$};

    \draw (-2.9, 0) -- (-2, 0);
    \draw (1.9, 0) -- (1.5, 0);
    \node at (-0.5, 0){$\dots$};
    \node at (-3, 0){$\color{blue}{\bullet}$};

    \draw (-3, 4) -- (-2, 4);
    \node at (0,4) [scale=0.75]{$\rho$-multiplication};

    \draw[red, double, -{Stealth}] (7, 4.6) --(6, 2.8);
    \node at (7.5, 5)[scale=0.7]{\cref{arbitrarylongdiff}};
    \end{tikzpicture}
    \caption{The genuine $C_2$-equivariant Adams spectral sequence at coweight $s-w$ and filtration 0}
    \label{fig:fil0diagram}
\end{figure}

The discussion in \cref{sec:genc2diff} indicates that, for a fixed coweight $s-w\leq -2$, when $s$ is the largest, the initial differentials supported by the filtration-$0$ elements are always $d_2$-differentials; as $s$ becomes smaller, longer differentials begin to appear. An illustrative diagram for the general pattern of the genuine $C_2$-equivariant Adams spectral sequence at coweight $s-w \leq -2$ and filtration 0 is shown in \cref{fig:fil0diagram}. In accordance with \cref{fig:hurewiczim}, the red dots indicate classes that support nontrivial Adams differentials, while the blue dots represent permanent cycles. The horizontal lines indicate $\rho$-extensions from right to left. In particular, \cref{thmA} ensures that the classes 
\[\frac{\theta}{\rho^{s}\tau^{w-s-2}}, \qquad 0\leq s\leq \psi(w-s-1)-1\] 
are all permanent cycles. For the next class \[\frac{\theta}{\rho^{\psi(w-s-1)}\tau^{w-s-2}},\] which is the first element at coweight $s-w$ that is not a permanent cycle, we prove that it supports an Adams differential of length depending on the \textit{2-adic valuation} of $w-s-1$.  

\begin{theorem}[\cref{longestdiff}]\label{arbitrarylongdiff}
    Let $v_2(\_)$ denote the \textit{2-adic valuation}. If $v_2(w-s-1)\geq 5$, the element $\frac{\theta}{\rho^{\psi(w-s-1)}\tau^{w-s-2}}$ survives to the $E_r$-page for $r=v_2(w-s-1)-1$, and it supports a nontrivial $d_r$-differential.
    \tqed
\end{theorem}

If $v_2(w-s-1)\leq 3$, then $\psi(w-s-1)=2^{v_2(w-s-1)}$, and \cref{thmA} implies that all filtration-$0$ elements of this coweight are permanent cycles. If $v_2(w-s-1)=4$, then $\psi(w-s-1)=9$, and we will show in \cref{ex:d2} that $\frac{\theta}{\rho^s\tau^{w-s-2}}$ supports a $d_2$-differential for each $9\leq s \leq 15$. Together with \cref{arbitrarylongdiff} and an example of a $d_3$-differential in \cref{ex:d3}, this implies \cref{thmB}.

Based on low-dimensional calculations, we propose the following conjecture for some of the major intermediate differentials. 

\begin{conjecture}
    Let $M$ denote the algebraic Mahowald invariant for the 2-primary Adams spectral sequence \cite{mahowaldravenel1993root}. For $j \gg k> 0$, 
    \[d_{k+1}(\frac{\theta}{\rho^{2^{j-k}+s(k)}\tau^{2^j-1}})=\frac{\theta}{\tau^{2^j+2^{j-k-1}+c(k)-1}}M(h_0^k)h_{j-k},\]
    where $s(k)$ and $c(k)$ are the stem and coweight of $M(h_0^k)$, with values given in the table below.
    \begin{center}\begin{tabular}{>{$}l<{$}>{$}l<{$}>{$}l<{$}}
        \hline
        k & s(k) & c(k) \\
        \hline
        4n+1 & 8n+1 & 4n \\
        4n+2 & 8n+2 & 4n \\
        4n+3 & 8n+3 & 4n \\
        4n+4 & 8n+7 & 4n+3 \rlap{\hspace{13em} \tqed} \\
        \hline
    \end{tabular}   \end{center} 
\end{conjecture}

We prove \cref{arbitrarylongdiff} by analyzing differentials in the classical Adams spectral sequences of various stunted projective spectra. One of the key steps uses the techniques of the Generalized Leibniz Rule and the Generalized Mahowald Trick, which provide a framework for propagating Adams differentials and hidden extensions via synthetic homotopy. For concrete statements and a general background on synthetic homotopy theory, see, for example, \cite[Section~3, Section~6]{linwangxu2025lastkervaire}.

Another crucial lemma in the proof of \cref{arbitrarylongdiff}, which may be of independent interest, identifies the differential supported by the class one $h_0$-step below the class detecting the image-of-$J$ generator at stem $8k+7$ in the classical Adams spectral sequence of $S^0$. The order of the image-of-$J$ \cite{adams1966on} makes it immediate to locate its generator in the $h_0$-tower in $\Ext(S^0)$, and therefore the element admitting an $h_0$-extension to the generator must support an Adams differential. Evidence from both the computation below the 90-stem \cite{isaksenwangxu2023stable} and the Adams spectral sequence of image-of-$J$ \cite{bruner2022adams} indicates that the target should be $P^k h_2^2$, while a concrete proof does not seem to appear in the literature. We explicitly confirm this family of differentials.

\begin{theorem}[\cref{diffofimj}]\label{imjdiffintro}
    For $k\ge1$, let $r_k:= v_2(k+1)+2$. Let $a_{8k+7}$ denote the element in $\Ext(S^0)$ that detects the generator of $\pi_{8k+7}J$. Then there exists an element $a_{8k+7}'$ satisfying $a_{8k+7}'\cdot h_0=a_{8k+7}$ such that it survives to the $E_{r_k}$-page of the Adams spectral sequence and supports a nontrivial $d_{r_k}$-differential
    \[d_{r_k}(a_{8k+7}')=h_2\cdot P^kh_2. \tag*{$\triangleleft$}\]
\end{theorem}

As an immediate consequence, we obtain the following.

\begin{cor}\label{cor:arblongdiffofass}
    In the 2-primary classical Adams spectral sequence of $S^0$, there exist nontrivial $d_r$-differentials for any $r\ge 2$.
\tqed
\end{cor}

The analogous odd primary result was first proved by Andrews \cite{andrews2015v1}, by consulting the Miller square \cite{miller1981relations} and studying the $v_1$-periodic Adams spectral sequence. The proof of the 2-primary statement is more subtle, as the Cartan--Eilenberg spectral sequence does not collapse in this setting.

\subsection{Notation}
\begin{itemize}
    \item $\iota_E$: the Hurewicz map for a ring spectrum $E$;
    \item $\sigma$: the $C_2$-sign representation;
    \item $\pi_{*, *}^{C_2}(\_)$: the $RO(C_2)$-graded $C_2$-equivariant stable homotopy groups;
    \item $a_\sigma$: the Euler class of $\sigma$ in $\pi^{C_2}_{-1,-1}S^0$;
    \item $\underline{\bbF_2}:$ the constant Mackey functor over $\bbF_2$;
    \item $\underline{\bbZ}:$ the constant Mackey functor over $\bbZ$;
    \item $\underline{A}:$ the $C_2$-equivariant Burnside Mackey functor;
    \item $H \underline{\bbF_2}$: the genuine $C_2$-equivariant Eilenberg--MacLane spectrum valued in $\underline{\bbF_2}$;
    \item $H \underline{\bbZ}$: the genuine $C_2$-equivariant Eilenberg--MacLane spectrum valued in $\underline{\bbZ}$;
    \item $H \underline{A}$: the genuine $C_2$-equivariant Eilenberg--MacLane spectrum valued in $\underline{A}$;
    \item $H\bbF_2^\bbR$: the $\bbR$-motivic Eilenberg--MacLane spectrum with $\bbF_2$-coefficient;
    \item $NC$: the negative cone of $H \underline{\bbF_2}_{*, *}$;
    \item $\Ext^{*, *, *}_\bbC$: the $E_2$-page of the 2-primary $\bbC$-motivic Adams spectral sequence;
    \item $\Ext^{*, *, *}_\bbR$: the $E_2$-page of the 2-primary $\bbR$-motivic Adams spectral sequence;
    \item $\Ext_{C_2}^{*, *, *}$: the $E_2$-page of the 2-primary genuine $C_2$-equivariant Adams spectral sequence;
    \item $\Ext_{h}^{*, *, *}$: the $E_2$-page of the 2-primary Borel $C_2$-equivariant Adams spectral sequence;
    \item $\Ext^{*, *}$: the $E_2$-page of the classical 2-primary Adams spectral sequence;
    \item $\leftindex^{\bbR}E_r^{*, *, *}$: the $E_r$-page of the 2-primary $\bbR$-motivic Adams spectral sequence;
    \item $\leftindex^{C_2}E_r^{*, *, *}$: the $E_r$-page of the 2-primary genuine $C_2$-equivariant Adams spectral sequence;
    \item $\leftindex^{h}E_r^{*, *, *}$: the $E_r$-page of the 2-primary Borel $C_2$-equivariant Adams spectral sequence;
    \item $E_r^{*, *}$: the $E_r$-page of the classical 2-primary Adams spectral sequence;
    \item $v_2(\_)$: the 2-adic valuation function;
    \item $\psi(n)$: the $n$-th Radon--Hurwitz number;
    \item $J$: the image-of-$J$ spectrum.
\end{itemize}

\subsection{Organization}
In \cref{sec:prelim} we review some basics of $C_2$-equivariant Adams spectral sequences. We prove \cref{mainthm} in \cref{sec:hurewiczim} by determining the permanent cycles on the $0$-line of the genuine $C_2$-Adams spectral sequence. In \cref{sec:genc2diff}, we further explore the patterns of the genuine $C_2$-equivariant Adams differentials on the $0$-line elements and prove \cref{thmB} and \cref{arbitrarylongdiff}. The proof relies on \cref{imjdiffintro}, which we will discuss in \cref{sec:imjdiff}. The geometric proof of \cref{thmA} is given in \cref{sec:geoconstr}. In \cref{sec:HZHA}, we prove \cref{mainHurewiczofHZHA} as an application of \cref{mainthm}.

\subsection*{Acknowledgment} The authors thank the American Institute of Mathematics for their hospitality during the workshop, where the collaboration on this project started. The authors thank the other co-organizers of the AIM workshop on Computations in Stable Homotopy Theory, Eva Belmont, Hana Jia Kong, XiaoLin Danny Shi, for organizing an excellent meeting and for creating a stimulating and supportive environment that made this collaboration possible. The authors thank Bert Guillou and Dan Isaksen for helpful discussions at the AIM workshop and for comments on earlier drafts. The authors thank William Balderrama for helpful conversations. The sixth author is partially supported by the AMS Centennial Research Fellowship.

\section{Preliminaries on \texorpdfstring{$C_2$-equivariant}{C2 equivariant} homotopy}\label{sec:prelim}

We review the basics of the $\bbR$-motivic, genuine $C_2$-equivariant, and Borel $C_2$-equivariant Adams spectral sequences; references for this material may be found in \cite{duggerisaksen2010motivic,duggerisaksen2017low, guillouhillisaksenravenel2020cohomology, ma2026borel}.

Let $H\bbF_2^\bbR$ and $H\bbF_2^h$ denote the $\bbR$-motivic and Borel $C_2$-equivariant Eilenberg--MacLane spectra with $\bbF_2$-coefficients, respectively; let $\mathcal{A}^{\bbR}$, $\mathcal{A}^{C_2}$, and $\mathcal{A}^{h}$ denote the $\bbR$-motivic, genuine $C_2$-equivariant, and Borel $C_2$-equivariant Steenrod algebras, respectively. Then, the $\bbR$-motivic, genuine $C_2$-equivariant, and Borel $C_2$-equivariant Adams spectral sequences of the spheres are
\begin{align*}
\leftindex^{\bbR}E_2^{s, f, w}&=\Ext_{\bbR}^{s,f,w}:=\Ext_{\calA^{\bbR}}^{s,f,w}({H\bbF_2}^\bbR_{*, *}, {H\bbF_2}^\bbR_{*, *})\implies (\pi_{s,w}^\bbR)^\wedge_2\\
\leftindex^{C_2}E_2^{s, f, w}&= \Ext_{C_2}^{s,f,w}:=\Ext_{\calA^{C_2}}^{s,f,w}({H\underline{\bbF_2}}_{*, *}, {H\underline{\bbF_2}}_{*, *}) \implies (\pi_{s,w}^{C_2})^\wedge_2\\
\leftindex^{h}E_2^{s,f,w} &= \Ext_{h}^{s,f,w} := \Ext_{\calA^h}^{s,f,w}({H\bbF_2}^h_{*, *}, {H\bbF_2}^h_{*, *})\implies (\pi_{s,w}^{C_2})^\wedge_2
\end{align*}
where $f$ denotes the Adams filtration, and $(s,w)$ in the equivariant case aligns with the chosen $RO(C_2)$-grading, while $(s, w)$ in the $\bbR$-motivic case is the pair of stem and motivic weight. Similarly, we use the grading $\Ext^{s, f}$ and $\Lambda^{s, f}$ for the classical $\Ext$-group and the Lambda algebra, where $s$ is the stem and $f$ is the Adams filtration. Our grading convention follows that of \cite{duggerisaksen2010motivic, isaksenwangxu2023stable}.

\begin{prop}[\cite{guillouhillisaksenravenel2020cohomology, belmontxuzhang2024reduced}]\label{rhobockstein}
~\begin{enumerate}
    \item The splitting $H\underline{\bbF_2}_{*, *}  \cong {H\bbF_2}^\bbR_{*, *} \oplus NC$ induces a splitting
    \[\Ext_{C_2}=\Ext_\bbR\oplus \Ext_{NC}\]
    where $\Ext_{NC} = \Ext_{\mathcal{A}^{C_2}}(H\underline{\bbF_2}_{*, *} , NC)$.
    \item There is a $\rho$-Bockstein spectral sequence converging to $\Ext_{C_2}$ such that under the splitting in part (1), the spectral sequence decomposes as
    \[E_1^+=\Ext_\bbC[\rho]\implies \Ext_\bbR\]
    \[E_1^-\implies \Ext_{NC}\] 
    where $\Ext_\bbC$ is the $E_2$-page of the $\bbC$-motivic Adams spectral sequence. A Bockstein differential $d_r$ takes a class $x$ of degree $(s,f,w)$ to a class $d_r(x)$ of degree $(s-1, f+1, w)$.
    \item $E_1^-$ is generated over $\bbF_2$ by elements of the following two types:
    \begin{itemize}
        \item Elements of the form $\frac{\theta}{\rho^a \tau^b}x$ where $a, b \geq 0$, and $x$ is an element of $\Ext_\bbC$ that is $\tau$-free and not divisible by $\tau$. If $x$ has degree $(s, f, w)$ in $\Ext_\bbC$, then $\frac{\theta}{\rho^a \tau^b}x$ has degree $(s+a,f, w+a+b+2)$.
        \item Elements of the form $\frac{Q}{\rho^a \tau^b}x$ where $0\leq a, 0\leq b\leq k$, and $x$ is an element of $\Ext_\bbC$ that is $\tau$-torsion and divisible by $\tau^k$ but not $\tau^{k+1}$. If $x$ has degree $(s, f, w)$ in $\Ext_\bbC$, then $\frac{Q}{\rho^a \tau^b}x$ has degree $(s+a+1,f-1, w+a+b+1)$. \tqed
\end{itemize}
\end{enumerate}
\end{prop}

The relationship between the genuine and the Borel $C_2$-equivariant Adams spectral sequences is as follows.

\begin{prop}[{\cite[Theorem~6.2]{ma2026borel}}]\label{genuinevsborel}
    \[\Ext_{h}^{s,f,w}\cong H\left(\Ext_{\bbR}^{s,f,w}\oplus\Ext_{NC}^{s,f-1,w},\delta\right),\]
    where $\delta$ is the composition map
    \[\Ext_{NC}^{s,f-1,w}\hookrightarrow\Ext_{C_2}^{s,f-1,w}\xrightarrow{d_2}\Ext_{C_2}^{s-1,f+1,w}\to\Ext_{\bbR}^{s-1,f+1,w}. \tag*{$\triangleleft$}\]
\end{prop}

In particular, there is a natural map
\[\varphi:\Ext_{h}^{s,f,w}\to\Ext_{NC}^{s,f-1,w}.\]
When $w\ge s+2$, $\varphi$ is an isomorphism since both $\Ext_{\bbR}^{s,f,w}$ and $\Ext_{\bbR}^{s-1,f+1,w}$ are trivial. Furthermore, since $\Ext_{\bbR}^{s,f-1,w}$ is trivial as well, we have
\[\Ext_{h}^{s,f,w}\cong\Ext_{C_2}^{s,f-1,w}.\]

Let $j$ be the Borel completion map $H\underline{\bbF_2}\to H\bbF_2^h$. It induces a map
\[j_*:\leftindex^{C_2}E_1^{s,f,w}\to\leftindex^hE_1^{s,f,w},\]
between the $E_1$-page of the genuine and Borel $C_2$-equivariant Adams spectral sequences, which is trivial when $w\ge s+2$. The differentials in the genuine and Borel Adams spectral sequences have the following correspondence.

\begin{prop}[{\cite[Theorem~6.4, 6.5]{ma2026borel}}]\label{d(NC)}
~\begin{enumerate}
        \item Let $z\in\leftindex^{C_2}{E}_1^{s,f,w}$ be an element with $j_*z=0$. Suppose $z$ supports a nontrivial differential $d_r^{C_2}([z])=[z']$, where $r\ge2$ and $z'\in\leftindex^{C_2}E_1^{s-1,f+r,w}$. Suppose further that $j_*z'=0$. Then there exists an element $x\in\leftindex^hE_1^{s,f+1,w}$ surviving to $\leftindex^hE_2$ with $\varphi([x])=[z]\in\leftindex^{C_2}E_2$, and an element $x'\in\leftindex^hE_1^{s-1,f+r,w}$ surviving to $\leftindex^hE_r$ with $\varphi([x'])=[z']\in\leftindex^{C_2}E_2$, such that $\leftindex^hd_r([x])=[x']$.
        \item Let $z\in\leftindex^{C_2}{E}_1^{s,f,w}$ be an element with $j_*z=0$. Suppose $z$ survives to $\leftindex^{C_2}E_{\infty}$ and detects $\hat{z}\in\pi^{C_2}_{s,w}$. Then there exists an element $x\in\leftindex^hE_1^{s,f+1,w}$ surviving to $\leftindex^hE_{\infty}$ and detecting $\hat{z}$, such that $\varphi([x])=[z]$. \tqed
\end{enumerate}
\end{prop}

On the other hand, the Borel Adams spectral sequence is related to the classical Adams spectral sequences for the stunted real projective spectra in the sense that the following two Mahowald squares are isomorphic, as shown in \cite[Section~3]{ma2026borel}:

\begin{equation}
    \begin{aligned}
\xymatrix@C=1cm{
 & \Ext_{\calA^h}\left({H\bbF_2}^h_{*,*}/\rho\right)[\rho] \ar@{=>}[rd]|{\textbf{algebraic }\mathbf{\rho-}\textbf{Bockstein SS}} \ar@{=>}[ld]|{\textbf{Borel Adams SS}} & \\
\pi_{*,*}^{C_2}\left(\bbS/\rho\right)^\wedge_2[\rho] \ar@{=>}[rd]|{\textbf{topological }\mathbf{\rho-}\textbf{Bockstein SS}} & & \Ext_{\calA^h}\left({H\bbF_2}^h_{*,*}\right) \ar@{=>}[ld]|{\textbf{Borel Adams SS}}\\
 & \left(\pi_{*,*}^{C_2}\bbS\right)^\wedge_2. & \\
}
\end{aligned}
\end{equation}

\begin{equation}\label{eq:mahowaldsq}
    \begin{aligned}
\xymatrix@!C@=1cm{
 & \bigoplus\Ext(\bbF_2) \ar@{=>}[rd]|{\textbf{algebraic Atiyah-Hirzebruch SS}} \ar@{=>}[ld]|{\textbf{Adams SS}} & \\
\bigoplus\pi_*\mathbb{S}^\wedge_2 \ar@{=>}[rd]|{\textbf{topological Atiyah-Hirzebruch SS}} & & \Ext(\varprojlim\limits_kH_*\bbR P_{-k}^{-w-1}) \ar@{=>}[ld]|{\textbf{Adams SS}}\\
 & \left(\pi_*\bbR P_{-\infty}^{-w-1}\right)^\wedge_2. & \\
}\end{aligned}
\end{equation}

Furthermore, for the lower right spectral sequences in the above squares, we have the following description.

\begin{prop}[{\cite[Theorem~3.3,~4.1,~4.2]{ma2026borel}}]\label{BorelandRPAdams}
The Borel Adams spectral sequence of the $C_2$-equivariant sphere
\[
    \Ext_{h}^{s,f,w}\implies\left(\pi_{s,w}^{C_2}\bbS\right)^\wedge_2
\]
is either isomorphic to the classical Adams spectral sequence of the stunted real projective spectra
\[\Ext^{s-w,f-1}\left(\Sigma\bbR P_{-\infty}^{-w-1}\right)\implies \left(\pi_{s-w}\Sigma\bbR P_{-\infty}^{-w-1}\right)^\wedge_2\]
when $w>0$, or isomorphic to
\[\Ext^{s-w,f-1}\left(\bbR P_{-w}^{\infty}\right)\oplus\Ext^{s-w,f}\implies\left(\pi_{s-w}\bbR P_{-w}^{\infty}\right)^\wedge_2\oplus\pi_{s-w}\bbS^\wedge_2\]
when $w\le0$.
\tqed
\end{prop}

The fiber sequence
\[\bbR P_{-\infty}^\infty\to\bbR P_{-w}^\infty\to\Sigma\bbR P_{-\infty}^{-w-1},\]
and Lin's theorem \cite{lin1980on} imply that we have the 2-adic equivalence $\tau_{\le-2}(\Sigma\bbR P_{-\infty}^{-w-1})\simeq\tau_{\le-2}(\bbR P_{-w}^\infty)$. In particular, when $w\ge\max\{s+2,0\}$, we have
\[\Ext^{s-w,f-1}\left(\Sigma\bbR P_{-\infty}^{-w-1}\right)\cong\Ext^{s-w,f-1}\left(\bbR P_{-w}^\infty\right).\]

\section{The Hurewicz image of \texorpdfstring{$H\underline{\bbF_2}$}{HF2}}\label{sec:hurewiczim}
\subsection{The algebraic Hurewicz image}

We first compute the $0$-line of $\Ext_{{C_2}}$. 
\begin{theorem}\label{algHurewiczim}
Recall that $v_2(\_)$ denotes the \textit{2-adic valuation}. Then
\[\Ext_{{C_2}}^{s, 0, w} =\begin{cases}
        \bbF_2\{\rho^{-s}\} \quad &\mathrm{if}\ s=w\leq 0,\\
        \bbF_2\{\frac{\theta}{\rho^s \tau^{w-s-2} }\} \quad &\mathrm{if}\ 0\leq s < 2^{v_2(w-s-1)}\ \mathrm{and}\ s-w\leq -2, \ \\
        0\quad &\mathrm{otherwise.}
    \end{cases}\]
\end{theorem}

\begin{proof} $\ $
    \begin{enumerate}
        \item If $s-w\geq 0$, for degree reasons, $\Ext_{NC}=0$. In the $\rho$-Bockstein spectral sequence for $\Ext_{\bbR}$, by \cite{duggerisaksen2017low, belmontisaksen2020rmotivic}, there are differentials
        \begin{align*}
            d_1(\tau)&=\rho h_0\\
            d_{2^n}(\tau^{2^n})&=\rho^{2^n}\tau^{2^{n-1}}h_n, \quad n\geq 1
        \end{align*}
        which are $\rho$-linear. Therefore, the only elements of filtration 0 that survive to $\Ext_\bbR$ are the powers of $\rho$.
        
        \item If $s-w<0$, $\Ext_{\bbR}=0$. In the $\rho$-Bockstein spectral sequence for $\Ext_{NC}$, the only elements of filtration 0 are of the form $\frac{\theta}{\rho^a\tau^b}$ for some $a, b\geq 0$ by \cref{rhobockstein}.

        By \cite[Proposition 4.10]{guillouisaksen2024c2}, for each $n\geq 0$, there is a $\rho$-Bockstein differential
        \[d_{2^n}\left(\frac{\theta}{\rho^{2^n}\tau^{2^n-1}}\right)=\frac{\theta}{\tau^{2^n+\lceil2^{n-1}\rceil-1}}\,h_n,\]
        which is both $\rho$-divisible and $\tau^{2^{n+1}}$-divisible. Thus, for each $i, k \geq 0$, we have
        \[d_{2^n}\left(\frac{\theta}{\rho^{2^n+i}\tau^{k\cdot 2^{n+1}+2^n-1}}\right)=\frac{\theta}{\rho^i\tau^{k\cdot 2^{n+1}+2^n+\lceil2^{n-1}\rceil-1}}h_n.\]
        Therefore, the only permanent cycles are $\frac{\theta}{\rho^j\tau^{k\cdot 2^{n+1}+2^n-1}}$ for $0\leq j\leq 2^n-1$. The claim follows by fitting in the correct bidegree.
    \end{enumerate}
\end{proof}

\begin{prop}\label{prop:1[s-w]}
    For $s-w\leq -2$ and $s<2^{v_2(w-s-1)}$, the class $\frac{\theta}{\rho^s\tau^{w-s-2}}\in \Ext_{C_2}^{s, 0, w}$ corresponds to the class $1[s-w]$ in the Adams spectral sequence of $\Ext^{s-w, 0}(\bbR P^\infty_{-w})$. Moreover, $\frac{\theta}{\rho^s\tau^{w-s-2}}$ supports a genuine Adams $d_r$-differential if and only if $1[s-w]$ supports a classical Adams $d_r$-differential.
\end{prop}
\begin{proof}
    By \cref{d(NC)}, the class $\frac{\theta}{\rho^s\tau^{w-s-2}}$ in $\Ext_{C_2}^{s, 0, w}$ corresponds to some class $x$ in $\Ext_h^{s, 1, w}$. By \cref{BorelandRPAdams}, since $w>0$, $x$ in $\Ext_h^{s, 1, w}$ further corresponds to a class in $\Ext^{s-w, 0}(\bbR P_{-\infty}^{-w-1}) =\Ext^{s-w, 0}(\bbR P^\infty_{-w})$. 

    Consider the algebraic Atiyah --Hirzebruch spectral sequence
    \[E_1^{s, f}=\bigoplus_{i=-w}^\infty \Ext^{s, f}(S^i)\implies \Ext^{s, f}(\bbR P^{\infty}_{-w}).\]
    At stem $s-w$ and filtration $0$, the only generator in the $E_1$-page is of the form $1[s-w]$. This class is the only possibility that corresponds to $x$.

    The correspondence of differentials follows from \cref{d(NC)} and the isomorphism of spectral sequences in \cref{BorelandRPAdams}.
\end{proof}

\begin{rmk}
If $s-w\leq -2$ and $s\geq 2^{v_2(w-s-1)}$, the class $\frac{\theta}{\rho^{s}\tau^{w-s-2}}$ supports a $\rho$-Bockstein differential according to the proof of \cref{algHurewiczim}. Correspondingly, $1[s-w]$ supports a nontrivial algebraic Atiyah--Hirzebruch differential in $\bbR P^\infty_{-w}$.
\end{rmk}

\subsection{The homotopy Hurewicz image}
It remains to determine the fate of the $0$-line elements in the genuine Adams spectral sequence. For simplicity, we denote the image-of-$J$ spectrum by $J$.

Recall that for $n=(2a+1)2^{c+4d}$ with $0\leq c\leq 3$,  $\psi(n)=2^c+8d$. A corollary of Adams' solution to the vector field problem \cite{adams1962vector} is detailed information about the primary attaching maps of stunted projective spectra. 

\begin{prop}\label{prop:RPattachingmap}
~\begin{enumerate}
        \item For $n>0$, the top cell of $\bbR P^{n-1}_{n-\psi(n)}$ splits off, while it does not split in $\bbR P^{n-1}_{n-\psi(n)-1}$. Such an attaching map from the $(n-1)$-cell to $(n-\psi(n)-1)$-cell can be chosen as a generator of $\pi_{\psi(n)-1}J$. 
        \item For $n< 0$, the top cell of $\bbR P^{n-1}_{n-\psi(-n)}$ splits off, while it does not split in $\bbR P^{n-1}_{n-\psi(-n)-1}$. Such an attaching map from the $(n-1)$-cell to $(n-\psi(-n)-1)$-cell can be chosen as a generator of $\pi_{\psi(-n)-1}J$.
    \end{enumerate}
\end{prop}
\begin{proof}
~\begin{enumerate}
    \item See \cite[Section~V.2]{brunermaymccluresteinberger1986hinfinity}.  
    \item Note that for $N > v_2(n)$, $\psi(n)=\psi(2^N+n)$ and if $n<0$, then $\psi(-n)=\psi(2^N+n)$. In particular, for $n<0$, by James periodicity,
    \[\bbR P^{n-1}_{n-\psi(-n)}\simeq \Sigma^{-2^N} \bbR P^{2^N+n-1}_{2^N+n-\psi(-n)}\simeq \Sigma^{-2^N} \bbR P^{2^N+n-1}_{2^N+n-\psi(2^N+n)}.\]
    The claim follows from part (1).
\end{enumerate}
\end{proof}

\begin{theorem}[\cref{mainthm}]\label{homotopyHurewiczim}
The $RO(C_2)$-graded Hurewicz image of $H\underline{\bbF_2}$ consists of
\[\im (\iota_{H\underline{\bbF_2}})_{s, w}=\begin{cases}
        \bbF_2\{\rho^{-s}\} \quad &\mathrm{if}\ s=w\leq 0,\\
        \bbF_2\{\frac{\theta}{\rho^s\tau^{w-s-2}}\} \quad &\mathrm{if}\ 0\leq s< \psi(w-s-1) \ \mathrm{and}\ s-w\leq -2,\\
        0\quad &\mathrm{otherwise.} 
        \end{cases}\]
In other words, the class $\frac{\theta}{\rho^s\tau^{w-s-2}}$ supports a nontrivial Adams differential if and only if $\psi(w-s-1) \leq s < 2^{v_2(w-s-1)}$.
\end{theorem}
\begin{proof}
~\begin{enumerate}
        \item For $s=w$, the powers of $\rho$ are surviving permanent cycles for degree reasons. In particular, $\rho$ detects $a_\sigma$.
        \item 
        By \cref{prop:1[s-w]}, it is enough to study the behavior of $1[s-w]$ in the Adams spectral sequence of $\bbR P^\infty_{-w}$.

        If $0\leq s < \psi(w-s-1)$, then $-w \geq s-w+1-\psi(w-s-1)$ so that the $(s-w)$-cell splits off in $\bbR P_{-w}^{s-w}$ by \cref{prop:RPattachingmap}. Therefore,
        \[\Ext(\bbR P_{-w}^{s-w}) \cong \Ext(\bbR P_{-w}^{s-w-1}) \oplus \Ext(S^{s-w}),\]
        and $1[s-w]$ is a permanent cycle. By naturality and degree reasons, $1[s-w]$ is a permanent cycle in the Adams spectral sequence of $\bbR P_{-w}^\infty$. By \cref{d(NC)}, $\frac{\theta}{\rho^s \tau^{w-s-2}}$ is a permanent cycle in the genuine $C_2$-equivariant Adams spectral sequence.

        On the other hand, $\psi(w-s-1)\leq s < 2^{v_2(w-s-1)}$ occurs only when $v_2(w-s-1) \geq 4$. By \cref{prop:RPattachingmap}, we have a differential in the Atiyah--Hirzebruch spectral sequence of $\bbR P_{-w}^\infty$,
        \[d_{\psi(w-s-1)}({1[s-w]})=j_{\psi(w-s-1)-1}[s-w-\psi(w-s-1)],\]
        where $j_{\psi(w-s-1)-1}$ is a chosen generator of $\pi_{\psi(w-s-1)-1}J$. Therefore, in the right part of the Mahowald square \cref{eq:mahowaldsq}, $1[s-w]$ must support a differential in either the algebraic Atiyah--Hirzebruch spectral sequence or the Adams spectral sequence of $\bbR P_{-w}^\infty$. Since $s<2^{v_2(w-s-1)}$, $\frac{\theta}{\rho^s \tau^{w-s-2}}$ survives to $E_2$-page of the genuine $C_2$-equivariant Adams spectral sequence; so does $1[s-w]$ to the $E_2$-page of the Adams spectral sequence of $\bbR P_{-w}^\infty$. Therefore, $1[s-w]$ must support an Adams differential in the Adams spectral sequence of $\bbR P_{-w}^\infty$. By \cref{prop:1[s-w]}, $\frac{\theta}{\rho^s \tau^{w-s-2}}$ supports a nontrivial differential in the genuine $C_2$-equivariant Adams spectral sequence.    
    \end{enumerate}
\end{proof}

\begin{rmk}
    In \cite[Lemma~7.4.5]{balderramaculverquigley2025motivic}, Balderrama--Culver--Quigley showed that there exists a surviving permanent cycle in $\Ext_{h}^{s,1,w}$ of the Borel $C_2$-equivariant Adams spectral sequence if and only if the bottom cell of $\bbR P_{w-s-1}^{w-1}$ splits. With the relationship between the genuine and Borel $C_2$-equivariant Adams spectral sequence revealed by \cref{genuinevsborel}--\cref{BorelandRPAdams}, their result would lead to an alternative proof of \cref{homotopyHurewiczim}.
\end{rmk}

\begin{proof}[Proof of \cref{thmA}]
    By \cite{adams1962vector}, $S^n$ admits $k$ linearly independent vector fields if and only if 
    \begin{equation}\label{eq:proofthmA}
        k\leq \psi(n+1)-1.
    \end{equation}
    Let $w-s-2=n$. In particular, $s-w\leq -2$. Then \cref{homotopyHurewiczim} implies \cref{eq:proofthmA} holds if and only if \[\frac{\theta}{\rho^k\tau^n}\in \im (\iota_{H\underline{\bbF_2}})_{k, k+n+2}.\]
\end{proof}

\section{Some genuine \texorpdfstring{$C_2$}{C2}-equivariant Adams differentials}\label{sec:genc2diff}
In this section, we first give concrete examples of Adams $d_2$-differentials supported by $0$-line elements in the genuine $C_2$-equivariant Adams spectral sequence. We then prove \cref{arbitrarylongdiff}, which identifies the differential supported by the last non-permanent cycle of a fixed coweight. Together with \cref{ex:d2} and \cref{ex:d3}, this implies that there are Adams differentials of arbitrary length on the $0$-line.

\begin{ex}\label{ex:d2}
    The first nontrivial Adams differentials on the $0$-line occur when $w-s=17$. In this case, there are $d_2$-differentials 
\[d_2(1[-17])=h_1h_3[-26]\]
in the Adams spectral sequence of $\bbR P_{-w}^\infty$ for $26\leq w\leq 32$ by Lin's data \cite{linwangxu2025machineproofs}. By \cref{prop:1[s-w]}, these correspond to genuine $C_2$-equivariant Adams $d_2$-differentials on $\frac{\theta}{\rho^j\tau^{15}}$ for $9\leq j\leq 15$. Note that since $\frac{\theta}{\rho^{15}\tau^{15}} \in \langle \rho, \frac{\theta}{\tau^{23}}, h_4 \rangle$, 
\begin{align*}
    d_2\left(\frac{\theta}{\rho^{15}\tau^{15}}\right) &= d_2\left(\langle \rho, \frac{\theta}{\tau^{23}}, h_4 \rangle\right)\\
    &=\langle \rho, \frac{\theta}{\tau^{23}}, d_2(h_4) \rangle +\Indet\\
    &=\langle \rho, \frac{\theta}{\tau^{23}}, h_0h_3^2 \rangle +\Indet\\
    &=\frac{\theta}{\tau^{22}}h_3^2 + \Indet.
\end{align*}
The indeterminacy consists of $\{0, \frac{\theta}{\rho^6\tau^{19}}h_1h_3\}$, and the differential in the Adams spectral sequence of $\bbR P_{-w}^\infty$ implies that $\frac{\theta}{\rho^6\tau^{19}}h_1h_3$ must be a summand. Thus,
\[d_2\left(\frac{\theta}{\rho^{15}\tau^{15}}\right)=\frac{\theta}{\tau^{22}}h_3^2+\frac{\theta}{\rho^6\tau^{19}}h_1h_3. \]
Multiplying by $\rho$, we obtain
\[d_2\left(\frac{\theta}{\rho^{9+j}\tau^{15}}\right) = \frac{\theta}{\rho^j\tau^{19}}h_1h_3, \quad 0\leq j\leq 5. \]

Consequently, when $v_2(w-s-1)=4$, the elements $\frac{\theta}{\rho^{9+j}\tau^{w-s-2}}$ all support $d_2$-differentials for $0\leq j\leq 6$.
\tqed 
\end{ex}

In general, for fixed $s-w$ with $v_2(w-s-1)\ge4$, the differentials supported by $\frac{\theta}{\rho^s\tau^{w-s-2}}$ for $s=2^{v_2(w-s-1)}-1$ are always $d_2$-differentials, by a similar argument of the higher Leibniz rule as above. However, as $s$ decreases, $\frac{\theta}{\rho^s\tau^{w-s-2}}$ tends to support longer differentials. Although we do not know each differential supported by those $0$-line elements explicitly, we are able to show the following patterns of the longest differentials.

\begin{theorem}[\cref{arbitrarylongdiff}]\label{longestdiff}
    If $v_2(w-s-1)\geq 5$, $\frac{\theta}{\rho^{\psi(w-s-1)}\tau^{w-s-2}}$ survives to the $E_r$-page for $r=v_2(w-s-1)-1$, and it supports a nontrivial $d_r$-differential.
\end{theorem}

\begin{proof}
    To simplify the notation, let $b=s-w$ and $c=\psi(w-s-1)$. 
    
    By \cref{prop:1[s-w]}, the differential supported by $\frac{\theta}{\rho^{\psi(w-s-1)}\tau^{w-s-2}}=\frac{\theta}{\rho^c\tau^{-b-2}}$ in the genuine $C_2$-equivariant Adams spectral sequence corresponds to the differential supported by $1[b]$ in the Adams spectral sequence of $\bbR P^\infty_{b-c}$. By \cref{prop:RPattachingmap}, the top cell of $\bbR P^{b}_{b-c+1}$ splits off, while it does not split off in $\bbR P^{b}_{b-c}$, and there is a subcomplex inclusion
    \[\Sigma^{b-c}\Cof(j_{c-1}) \hookrightarrow \bbR P^{b}_{b-c}, \]
    where $\Cof(\alpha)$ denotes the cofiber of the map $S^{|\alpha|} \overset{\alpha}{\to} S^0$ and $j_{c-1}$ is a generator of $\pi_{c-1}J$. Let $a_{c-1}\in \Ext^{c-1, m}(S^0)$ denote the element that detects $j_{c-1}$ in the Adams spectral sequence. Here, $m$ is the Adams filtration of $a_{c-1}$. 
    
    In particular, there is an Adams differential
    \[d_m(1[b])=a_{c-1}[b-c]\]
    in the Adams spectral sequence of $\Sigma^{b-c}\Cof(j_{c-1})$. The claim follows by pushing forward this differential along the inclusions
    \[\Sigma^{b-c}\Cof(j_{c-1}) \hookrightarrow \bbR P^b_{b-c} \hookrightarrow \bbR P^\infty_{b-c},\]
    as verified in \cref{lem:v2=4k}--\cref{lem:v2=4k+3} below.
\end{proof}

The last step in the proof of \cref{longestdiff} is decomposed into four lemmas depending on the value of $v_2(w-s-1) \pmod 4$, plus a special case when $v_2(w-s-1)=5$.

For reference, the following table records the values of $c$, $m$, $j_{c-1}$, and $a_{c-1}$, together with the lemma treating each case. Note that $m=v_2(w-s-1)-1$ if $v_2(w-s-1)\equiv 0, 1,$ or $2 \pmod 4$. Indeed, we show that in these cases the Adams differential on $1[b]$ in $\Sigma^{b-c}\Cof(j_{c-1})$ pushes forward to a differential of the same length in $\bbR P^\infty_{b-c}$. On the other hand, the cases $v_2(w-s-1)=5$ and $v_2(w-s-1)\equiv 3\pmod 4$ require more delicate arguments.

\begin{table}[H]
    \centering
    \begin{tabular}{@{}c|c|c|c|c|c@{}}
\hline 
$v_2(w-s-1)$ & 5 &  $4k$ & $4k+1$ & $4k+2$ & $4k+3$                                                   \\\hline 
$c$ & 10 & $8k+1$ & $8k+2$ & $8k+4$ & $8k+8$ 
\\\hline 
$j_{c-1}$ & $\eta^2\sigma$ & $\{P^{k-1}c_0\}$ & $\eta\cdot \{P^{k-1}c_0\}$ & $\{P^kh_2\}$ & gen of $\pi_{8k+7}J$ 
\\\hline 
$a_{c-1}$ & $h_1^2h_3$ & $P^{k-1}c_0$ & $h_1P^{k-1}c_0$ & $P^kh_2$ & $a_{c-1}$ 
\\\hline 
$m$ & $3$ & $4k-1$ & $4k$ & $4k+1$ & $4k+1-v_2(k+1)$ 
\\\hline 
Proof & \cref{lem:v2=5} & \cref{lem:v2=4k} & \cref{lem:v2=4k+1} & \cref{lem:v2=4k+2} & \cref{lem:v2=4k+3} 
\\\hline 
\end{tabular}
    \label{table:w-s-1}
    \caption{Indices for \cref{longestdiff}}
\end{table}

\begin{lemma}\label{lem:v2=4k}
    If $v_2(w-s-1)=4k$ for $k\geq 2$, there is a differential
    \[d_{4k-1}(1[b])=P^{k-1}c_0[b-c]\]
    in the Adams spectral sequence of $\bbR P^\infty_{b-c}$.
\end{lemma}
\begin{proof}
    If $v_2(w-s-1)=4k$, $j_{c-1}$ is detected by $a_{c-1}=P^{k-1}c_0$. In the Adams spectral sequence of $\Sigma^{b-c}\Cof(j_{c-1})$, there is a differential 
    \[d_{4k-1}(1[b])=a_{c-1}[b-c]=P^{k-1}c_0[b-c].\]
    Pushing forward this differential along the inclusion 
    \[\Sigma^{b-c}\Cof(j_{c-1}) \hookrightarrow \bbR P^b_{b-c}\hookrightarrow \bbR P^{\infty}_{b-c},\] 
    we obtain that $1[b]$ must support a $d_{r_1}$-differential for $r_1\geq 4k-1$. We claim that \[\Ext^{b-1, f}(\bbR P)=0, \quad f>4k-1;\] 
    hence, $r_1=4k-1$.

    Consider the algebraic Atiyah--Hirzebruch spectral sequence
    \begin{equation}\label{eq:algAHSS}
        \bigoplus_{i=b-c}^\infty \Ext(S^i)\implies \Ext(\bbR P^{\infty}_{b-c}).
    \end{equation}
    Since $v_2(w-s-1) = 4k$, the only possible element in stem $b-1$ and filtration higher than $4k-1$ is the order 2 element of the $h_0$-tower of the image-of-$J$, coming from the $(b-c+1)$-cell. Since $b=s-w$ is odd and $c=8k+1$ is odd, we have $b-c+1$ is odd. Then, the nontrivial $Sq^1$-action on $H^{b-c+1}(\bbR P^\infty_{b-c})$ implies that the order 2 element of $(b-c+1)$-cell is killed by a $d_1$-differential due to $h_0$-multiplication. Therefore, $\Ext^{b-1, >r}(\bbR P^{\infty}_{b-c})$ vanishes.
\end{proof}

\begin{rmk}
    In \cref{lem:v2=4k}, we identified the explicit differential
    \[d_{4k-1}(1[b])=P^{k-1}c_0[b-c]\]
    in the Adams spectral sequence of $\bbR P^\infty_{b-c}$. The corresponding differential in the genuine $C_2$-equivariant Adams spectral sequence will have the form 
    \[d_{4k-1}\left(\frac{\theta}{\rho^{\psi(w-s-1)}\tau^{w-s-2}}\right)=\frac{\theta}{\tau^{w-4k-3}}P^{k-1}c_0+\Indet.\]
    The close similarity between the names of the elements in these two spectral sequences, as well as the appearance of indeterminacy, is discussed in \cite[Section~7]{ma2026borel}. The same remarks apply to the differentials in the other lemmas of this section as well.
\end{rmk}

\begin{lemma}\label{lem:v2=4k+1}
    If $v_2(w-s-1)=4k+1$ for $k\geq 2$, there is a differential
    \[d_{4k}(1[b])=h_1P^{k-1}c_0[b-c]\]
    in the Adams spectral sequence of $\bbR P^\infty_{b-c}$.
\end{lemma}
\begin{proof}
    If $v_2(w-s-1)=4k+1$, $j_{c-1}$ is detected by $a_{c-1}=h_1P^{k-1}c_0$. In the Adams spectral sequence of $\Sigma^{b-c}\Cof(j_{c-1})$, there is a differential 
    \[d_{4k}(1[b])=a_{c-1}[b-c]=h_1P^{k-1}c_0[b-c].\]
    Pushing forward this differential along the inclusion 
    \[\Sigma^{b-c}\Cof(j_{c-1}) \hookrightarrow \bbR P^b_{b-c}\hookrightarrow \bbR P^{\infty}_{b-c},\] 
    we obtain that $1[b]$ must support a $d_{r_2}$-differential for $r_2\geq 4k$. We claim that $h_1P^{k-1}c_0[b-c]$ survives to $\Ext(\bbR P^\infty_{b-c})$ through the algebraic Atiyah--Hirzebruch spectral sequence \cref{eq:algAHSS}; hence, naturality guarantees that $r_2=4k$. 

    For degree reasons, $h_1P^{k-1}c_0[b-c]$ does not support algebraic Atiyah--Hirzebruch differentials. The only possible sources that could hit $h_1P^{k-1}c_0[b-c]$ are
    \begin{enumerate}
        \item $h_0^3P^{k-2}e_0[b-c+1]$,
        \item $P^{k-1}c_0[b-c+2]$,
        \item the order 4 element of the $h_0$-tower of the image-of-$J$ from the $(b-c+3)$-cell,
        \item and $P^{k-1}h_1^3[b-c+7]$.
    \end{enumerate}
    
    The first candidate (1) can be eliminated as $h_0^3P^{k-2}e_0 \cdot h_0 \neq h_1P^{k-1}c_0$. Since $b=s-w\equiv 3\pmod 4$ and $c=8k+2\equiv 2 \pmod 4$, we have $b-c\equiv 1 \pmod 4$, $b-c+2\equiv 3\pmod 4$. Therefore, $Sq^2$ acts trivially on $H^{b-c}(\bbR P^\infty_{b-c})$ and (2) can be ruled out. Since $b$ is odd, $c$ is even, $b-c+3$ is even. Similar arguments in \cref{lem:v2=4k} implies that the element (3) supports a $d_1$-differential due to $h_0$-multiplication.
    
    Finally, we use the \emph{Lambda complex} to show that the last candidate (4) is itself a nonzero cycle in $\Ext(\bbR P^\infty_{b-c})$, so it does not kill $h_1P^{k-1}c_0[b-c]$ in the algebraic Atiyah--Hirzebruch spectral sequence. 
    
    Recall from \cite{bousfieldcurtiskanquillenrectorschlesinger1966modp, cohenlinmahowald1988adams} that for any spectrum $Y$, there is a differential graded module $H_*(Y)\otimes \Lambda^{*, *}$ over the Lambda algebra $\Lambda^{*, *}$. Differentials in this complex are generated by
    \[d(x)=\sum_{i\geq 1} Sq^{i}_*(x)\otimes \lambda_{i-1}\]
    for $x\in H_*(Y)$, where $Sq^i_*$ is the transpose of $Sq^i$.
    
    For $Y=\bbR P^\infty_{b-c}$, we denote the generator of $H_i(\bbR P^\infty_{b-c})$ by $e_i$. Then the relevant transposed Steenrod operations are
    \[
\begin{array}{c|ccccccc}
 & e_{b-c+7} & e_{b-c+6} & e_{b-c+5} & e_{b-c+4}
 & e_{b-c+3} & e_{b-c+2} & e_{b-c+1} \\
\hline
Sq_*^1
& e_{b-c+6} &  & e_{b-c+4} &
& e_{b-c+2} &  & e_{b-c}  \\
Sq_*^2
& e_{b-c+5} &  &  & e_{b-c+2}
& e_{b-c+1} &  &   \\
Sq_*^3
&  &  & e_{b-c+2} &
&  &  &    \\
Sq_*^4
&  & e_{b-c+2} & e_{b-c+1} & e_{b-c}
&  &  &    \\
Sq_*^5
& e_{b-c+2} &  & e_{b-c} &
&  &  &    \\
Sq_*^6
& e_{b-c+1} &  &  &
&  &  &  \\
Sq_*^7
&  &  &  &
&  &  &
\end{array}
\]
with blank entries indicating trivial actions. 

    Let $\alpha \in \Lambda^{8k-5, 4k-3}$ be a cycle representing $P^{k-1}h_2$ in the Adams spectral sequence, and thus, $\lambda_0\lambda_0\alpha$ is a cycle representing $P^{k-1}h_1^3$. In particular, since $h_0\cdot P^{k-1}h_1^3=0$, there is some $\beta\in \Lambda^{8k-4, 4k-1}$ such that 
    \[d(\beta)=\lambda_0 \lambda_0 \lambda_0 \alpha.\]
    Consider the following differentials in $H_*(\bbR P^\infty_{b-c})\otimes \Lambda^{*, *}$,
    \[
\begin{aligned}
d(e_{b-c+7}\otimes \lambda_0\lambda_0\alpha)
={}&
e_{b-c+6}\otimes \lambda_0\lambda_0\lambda_0\alpha
+e_{b-c+5}\otimes \lambda_1\lambda_0\lambda_0\alpha\\
&+
e_{b-c+2}\otimes \lambda_4\lambda_0\lambda_0\alpha
+e_{b-c+1}\otimes \lambda_5\lambda_0\lambda_0\alpha,
\\[4pt]
d(e_{b-c+5}\otimes \lambda_2\lambda_0\alpha)
={}&
e_{b-c+5}\otimes \lambda_1\lambda_0\lambda_0\alpha
+e_{b-c+4}\otimes \lambda_0\lambda_2\lambda_0\alpha\\
&+
e_{b-c+2}\otimes \lambda_2\lambda_2\lambda_0\alpha
+e_{b-c+1}\otimes \lambda_3\lambda_2\lambda_0\alpha\\
&+
e_{b-c}\otimes \lambda_4\lambda_2\lambda_0\alpha,
\\[4pt]
d(e_{b-c+4}\otimes \lambda_1\lambda_2\alpha)
={}&
e_{b-c+4}\otimes \lambda_0\lambda_2\lambda_0\alpha
+e_{b-c+2}\otimes \lambda_1\lambda_1\lambda_2\alpha\\
&+
e_{b-c}\otimes \lambda_3\lambda_1\lambda_2\alpha,
\\[4pt]
d(e_{b-c+1}\otimes \lambda_6\lambda_0\alpha)
={}&
e_{b-c+1}\otimes
(\lambda_5\lambda_0\lambda_0+\lambda_3\lambda_2\lambda_0)\alpha\\
&+
e_{b-c}\otimes \lambda_0\lambda_6\lambda_0\alpha,
\\[4pt]
d\bigl(
e_{b-c}\otimes
(\lambda_6\lambda_1+\lambda_5\lambda_2+\lambda_3\lambda_4)\alpha
\bigr)
={}&
e_{b-c}\otimes
\bigl(
\lambda_5\lambda_0\lambda_1
+\lambda_3\lambda_1\lambda_2\\
&\qquad\qquad
+\lambda_5\lambda_1\lambda_0
+\lambda_3\lambda_3\lambda_0
\bigr)\alpha,
\\[4pt]
d(e_{b-c+6}\otimes \beta)
={}&
e_{b-c+6}\otimes \lambda_0\lambda_0\lambda_0\alpha
+e_{b-c+2}\otimes \lambda_3\beta .
\end{aligned}
\]
    Adding them together, we have
    \begin{equation}\label{eq:v2=4k+1}
        \begin{aligned}
d\Bigl(&
e_{b-c+7}\otimes \lambda_0\lambda_0\alpha
+ e_{b-c+5}\otimes \lambda_2\lambda_0\alpha
+ e_{b-c+4}\otimes \lambda_1\lambda_2\alpha \\
&+ e_{b-c+1}\otimes \lambda_6\lambda_0\alpha
+ e_{b-c}\otimes
  (\lambda_6\lambda_1+\lambda_5\lambda_2+\lambda_3\lambda_4)\alpha
+ e_{b-c+6}\otimes \beta
\Bigr) \\
&=
e_{b-c+2}\otimes
\Bigl(
\lambda_4\lambda_0\lambda_0\alpha
+\lambda_3\beta
+\lambda_2\lambda_2\lambda_0\alpha
+\lambda_1\lambda_1\lambda_2\alpha
\Bigr).
\end{aligned}
    \end{equation}

Let $\gamma:=\lambda_4\lambda_0\lambda_0\alpha
+\lambda_3\beta
+\lambda_2\lambda_2\lambda_0\alpha
+\lambda_1\lambda_1\lambda_2\alpha\in\Lambda^{8k-1,4k}$ be the cycle in the right-hand-side. The only nonzero homology class in $\Ext^{8k-1, 4k}$ is the element of order 2 in the $h_0$-tower of image-of-$J$, which is at least $h_0^3$-divisible. Therefore, this order-2 element is represented by an element of the form $\lambda_0\lambda_0\lambda_0 \tilde{\alpha}$. 

Suppose $\gamma$ is homologous to 0 in $\Ext^{8k-1, 4k}$, so that there exists $\tilde{\beta} \in \Lambda^{8k, 4k-1}$ with 
\[d(\tilde{\beta})=\gamma.\]
Then 
\[d(e_{b-c+2}\otimes \tilde{\beta})=e_{b-c+2}\otimes \gamma.\]
Otherwise, $\gamma$ is homologous to $\lambda_0\lambda_0\lambda_0 \tilde{\alpha}$, so that there exists $\tilde{\beta} \in \Lambda^{8k, 4k-1}$ with 
\[d(\tilde{\beta})=\gamma+\lambda_0\lambda_0\lambda_0 \tilde{\alpha}.\]
Then
\[d\Bigl(e_{b-c+2}\otimes \tilde{\beta}+(e_{b-c+3}\otimes \lambda_0\lambda_0+e_{b-c+1}\otimes\lambda_2\lambda_0+e_{b-c}\otimes\lambda_1\lambda_2)\tilde{\alpha}\Bigr)=e_{b-c+2}\otimes \gamma.\]
In either case, $e_{b-c+2}\otimes \gamma$ is a boundary. By \cref{eq:v2=4k+1}, we have a nonzero cycle whose leading term $e_{b-c+7}\otimes \lambda_0\lambda_0\alpha$ represents $P^{k-1}h_1^3[b-c+7]$. Hence this class survives in $\Ext(\mathbb RP^\infty_{b-c})$.
\end{proof}

\cref{lem:v2=4k+1} does not cover the case $k=1$, namely, $v_2(w-s-1)=5$. We give a separate proof here.

\begin{lemma}\label{lem:v2=5}
    If $v_2(w-s-1)=5$, there is a differential
    \[d_{4}(1[b])=h_1c_0[b-10]\]
    in the Adams spectral sequence of $\bbR P^\infty_{b-10}$.
\end{lemma}
\begin{proof}
    For $v_2(w-s-1)=5$, $c=10$ and $j_{c-1}=\eta^2\sigma$. Because of the nontrivial $Sq^4$-action on $H^{b-10}(\bbR P^b_{b-10})$, there is an \emph{Atiyah--Hirzebruch} differential
    \[d_4(\nu^2[b-6])=\nu^3[b-10]=(\eta^2\sigma+\eta\epsilon)[b-10]\]
    in $\bbR P^{b}_{b-10}$, which implies the following square commutes
    \[\begin{tikzcd}
    S^{b-1} \arrow[d, "\eta\epsilon"'] \arrow[r, "\eta^2\sigma"] & S^{b-10} \arrow[d, hook] \\
    S^{b-10} \arrow[r, hook]                                     & \bbR P^{b-1}_{b-10}.    
    \end{tikzcd}\]
    Therefore, the composite $S^{b-1}\overset{\eta^2\sigma}{\to}S^{b-10}\hookrightarrow \bbR P^{b-1}_{b-10}$, whose cofiber is $\bbR P^{b}_{b-10}$, is of Adams filtration at least 4. By a similar Lambda-complex computation in \cref{lem:v2=4k+1}, $h_1c_0[b-10]$ survives to $\Ext(\bbR P^\infty_{b-10})$ through the algebraic Atiyah--Hirzebruch spectral sequence. Thus, there is a differential 
    \[d_{4}(1[b])=h_1c_0[b-10]\]
    in the Adams spectral sequence of $\bbR P^{b}_{b-10}$.
    
\end{proof}

\begin{lemma}\label{lem:v2=4k+2}
    If $v_2(w-s-1)=4k+2$ for $k\geq 1$, there is a differential
    \[d_{4k+1}(1[b])=P^kh_2[b-c]\]
    in the Adams spectral sequence of $\bbR P^\infty_{b-c}$.
\end{lemma}
\begin{proof}
    If $v_2(w-s-1)=4k+2$, $j_{c-1}$ is detected by $a_{c-1}=P^{k}h_2$. In the Adams spectral sequence of $\Sigma^{b-c}\Cof(j_{c-1})$, there is a differential 
    \[d_{4k+1}(1[b])=a_{c-1}[b-c]=P^{k}h_2[b-c].\]
    Pushing forward this differential along the inclusion 
    \[\Sigma^{b-c}\Cof(j_{c-1}) \hookrightarrow \bbR P^b_{b-c}\hookrightarrow \bbR P^{\infty}_{b-c},\] 
    we obtain that $1[b]$ must support a $d_{r_3}$-differential for $r_3\geq 4k+1$. We claim that $P^kh_2[b-c]$ survives to $\Ext(\bbR P^\infty_{b-c})$ through the algebraic Atiyah--Hirzebruch spectral sequence \cref{eq:algAHSS} so that naturality guarantees that $r_3=4k+1$. 

    For degree reasons, the only possible sources that could hit $P^kh_2[b-c]$ are
    \begin{enumerate}
        \item $h_1P^{k-1}c_0[b-c+3]$,
        \item and the order 2 element in the $h_0$-tower of the image-of-$J$ from the $(b-c+5)$-cell.
    \end{enumerate}
      
    Since $c = 8k+4$ and $b\equiv 3 \pmod 4$, we have $b-c+3 \equiv  2\pmod 4$. Therefore, there is a nontrivial $Sq^2$-action on $H^{b-c+3}(\bbR P^\infty_{b-c})$. This implies the algebraic Atiyah--Hirzebruch differential 
    \[d_2(P^{k-1}c_0[b-c+5])=h_1P^{k-1}c_0[b-c+3],\] 
    so $h_1P^{k-1}c_0[b-c+3]$ can be eliminated.

    We again use the \emph{Lambda complex} to show that the element (2) is a nonzero cycle in $\Ext(\bbR P^\infty_{b-c})$. Since the order 2 element in the $h_0$-tower of the image-of-$J$ from the $(b-c+5)$-cell is $h_0^3$-divisible, it can be written of the form
    $\lambda_0^3\tilde{\alpha}$. Moreover, there exists a $\tilde{\beta}\in \Lambda^{8k, 4k}$ such that
    \[d(\tilde{\beta})=\lambda_0^4\tilde{\alpha}.\]
    
    For $c=8k+4$, the relevant transposed Steenrod operations are
    \[
\begin{array}{c|ccccc}
 & e_{b-c+5} & e_{b-c+4} & e_{b-c+3}
 & e_{b-c+2} & e_{b-c+1}  \\
\hline
Sq_*^1
& e_{b-c+4} &  & e_{b-c+2} &  & e_{b-c}   \\
Sq_*^2
& e_{b-c+3} &  &  & e_{b-c} &    \\
Sq_*^3
&  &  & e_{b-c} &  &    \\
Sq_*^4
& e_{b-c+1} &  &  &  &   \\
Sq_*^5
&  &  &  &  &
\end{array}.
\]

    Consider the differentials
    \[
\begin{aligned}
d\bigl(
  e_{b-c+5}\otimes \lambda_0^3\tilde{\alpha}
  +e_{b-c+4}\otimes \tilde{\beta}
\bigr)
={}&
e_{b-c+4}\otimes \lambda_0^4\tilde{\alpha}
+e_{b-c+3}\otimes \lambda_1\lambda_0^3\tilde{\alpha} \\
&+
e_{b-c+1}\otimes \lambda_3\lambda_0^3\tilde{\alpha}
+e_{b-c+4}\otimes \lambda_0^4\tilde{\alpha} \\
={}&
e_{b-c+3}\otimes \lambda_1\lambda_0^3\tilde{\alpha}
+e_{b-c+1}\otimes \lambda_3\lambda_0^3\tilde{\alpha}, \\[4pt]
d(e_{b-c+3}\otimes \lambda_2\lambda_0^2\tilde{\alpha})
={}&
e_{b-c+3}\otimes \lambda_1\lambda_0^3\tilde{\alpha}
+e_{b-c+2}\otimes \lambda_0\lambda_2\lambda_0^2\tilde{\alpha} \\
&+
e_{b-c}\otimes \lambda_2^2\lambda_0^2\tilde{\alpha}, \\[4pt]
d(e_{b-c+2}\otimes \lambda_1\lambda_2\lambda_0\tilde{\alpha})
={}&
e_{b-c+2}\otimes \lambda_0\lambda_2\lambda_0^2\tilde{\alpha}
+e_{b-c}\otimes \lambda_1^2\lambda_2\lambda_0\tilde{\alpha}, \\[4pt]
d\bigl(
  e_{b-c+1}\otimes
  (\lambda_4\lambda_0^2+\lambda_2^2\lambda_0+\lambda_1^2\lambda_2)
  \tilde{\alpha}
\bigr)
={}&
e_{b-c+1}\otimes \lambda_3\lambda_0^3\tilde{\alpha} \\
&+
e_{b-c}\otimes
\lambda_0
(\lambda_4\lambda_0^2+\lambda_2^2\lambda_0+\lambda_1^2\lambda_2)
\tilde{\alpha}, \\[4pt]
d(e_{b-c}\otimes \lambda_3\lambda_2\lambda_0\tilde{\alpha})
={}&
e_{b-c}\otimes \lambda_3\lambda_1\lambda_0^2\tilde{\alpha}.
\end{aligned}
\]

Adding them together and modulo relations in $\Lambda^{*, *}$, we find that 
\[
\begin{aligned}
d\Bigl(
& e_{b-c+5}\otimes \lambda_0^3\tilde{\alpha}
+ e_{b-c+4}\otimes \tilde{\beta}
+ e_{b-c+3}\otimes \lambda_2\lambda_0^2\tilde{\alpha} + e_{b-c+2}\otimes \lambda_1\lambda_2\lambda_0\tilde{\alpha}\\
&+ e_{b-c+1}\otimes
  (\lambda_4\lambda_0^2+\lambda_2^2\lambda_0+\lambda_1^2\lambda_2)
  \tilde{\alpha} 
+ e_{b-c}\otimes \lambda_3\lambda_2\lambda_0\tilde{\alpha}
\Bigr)
&=0.
\end{aligned}
\]
    
    Therefore, there is a nonzero cycle whose leading term is $e_{b-c+5}\otimes \lambda_0^3\tilde{\alpha}$. As a result, $P^kh_2[b-c]$ survives in $\Ext(\bbR P^\infty_{b-c})$. 
\end{proof}

We have the last case $v_2(w-s-1)\equiv 3\pmod 4$ left, whose proof needs the techniques from \cite{linwangxu2025lastkervaire}.

\begin{lemma}\label{lem:v2=4k+3}
    If $v_2(w-s-1)=4k+3$ for $k\geq 1$, there is a differential
    \[d_{4k+2}(1[b])=h_2P^kh_2 [b-c]\]
    in the Adams spectral sequence of $\bbR P^\infty_{b-c}$.
\end{lemma}
\begin{proof}
    If $v_2(w-s-1) = 4k+3$, then $c=8k+8$ and $j_{c-1}$ is the generator of $\pi_{8k+7}J$ whose order is $2^{v_2(k+1)+4}$ \cite{davismahowaldimage1989}. Set $r=4k+2$. 
    
    Let $\bar{j}_{c-1}$ denote the composite
     \[\bar{j}_{c-1}:S^{b-1} \overset{j_{c-1}}{\longrightarrow} S^{b-c} \hookrightarrow \Sigma^{b-c}\Cof(2).\]
    Consider the 3-cell complex $T:=\Cof(\bar{j}_{c-1})$. In particular, there is a natural cofiber sequence 
    \begin{equation}\label{cofiberofT}
        S^{b-c}\overset{\partial}{\to }\Sigma^{b-c}\Cof(j_{c-1}) \overset{i}{\hookrightarrow} T \overset{p}{\twoheadrightarrow} S^{b-c+1}.
    \end{equation}
    
    In the Adams spectral sequence of $\Sigma^{b-c}\Cof(j_{c-1})$, there is an Adams differential
    \begin{equation}\label{eq1}
        d_m(1[b])=a_{c-1}[b-c].
    \end{equation}
    By \cref{diffofimj} below, a local chart of the Adams spectral sequence of $S^0$ around $a_{c-1}$ is
    \[\begin{tikzpicture}[scale=0.5, d1/.style={draw={red}}]
        \node at (1,0){$\bullet$};
        \node at (1,1){$\bullet$};
        \draw (1,0) -- (1,1);
        \draw (1,1) -- (1,1.6);
        \node at (1, 2.2){$\cdot$};
        \node at (1, 2.0){$\cdot$};
        \node at (1, 1.8){$\cdot$};
        \node at (1.8, 1)[scale=0.8]{$a_{c-1}$};
        \node at (1.8, 0)[scale=0.8]{$a'_{c-1}$};
        \node at (-1,4){$\bullet$};
        \node at (-1,3){$\bullet$};
        \draw (-1,3) -- (-1,4);
        \draw (-1,3) -- (-1, 2.4);
        \node at (-1, 2.2){$\cdot$};
        \node at (-1, 2.0){$\cdot$};
        \node at (-1, 1.8){$\cdot$};
        \node at (-1.5,4.5)[scale=0.8]{$h_2P^k h_2$};
        \draw[d1, ->] (1,0) -- (-0.9, 3.9);
        \node at (0.4, 2.5)[scale=0.7]{$d_{r-m}$};
    \end{tikzpicture}\]
    where there is some generator $a'_{c-1}$ satisfying $a'_{c-1}\cdot h_0=a_{c-1}$ and supports a differential
    \begin{equation}\label{eq2}
        d_{r-m+1}(a'_{c-1})=h_2P^kh_2. 
    \end{equation}

    The cofiber sequence in \cref{cofiberofT} induces a long exact sequence of the $\Ext$-groups, and because of the subcomplex inclusion $\Sigma^{b-c}\Cof(2)\hookrightarrow T$, we obtain
    \begin{align*}
        \partial_*:\Ext(S^{b-c}) &\to \Ext(\Sigma^{b-c}\Cof(j_{c-1}))\\
        a'_{c-1} &\mapsto a_{c-1}[b-c].
    \end{align*}
    Combined with \cref{eq2} and the Generalized Mahowald Trick \cite[Theorem~6.12]{linwangxu2025lastkervaire},
    \[\begin{tikzpicture}[scale=0.5, d1/.style={draw={red}}]
        \node at (1,0){$\bullet$};
        \node at (1.5, -0.5)[scale=0.85]{$a'_{c-1}$};
        \node at (0,4){$\bullet$};
        \node at (-0.5,4.5)[scale=0.85]{$h_2P^k h_2$};
        \draw[->] (1,0) -- (0.1, 3.9);
        \node at (0.5, 2.5)[scale=0.8]{$d_{r-m+1}$};
        \node at (0.5, -2){$\Ext(S^{b-c})$};

        \node at (6, 1){$\bullet$};
        \node at (6.5, 0.4)[scale=0.85]{$a_{c-1}[b-c]$};
        \node at (6.5, -2){$\Ext(\Sigma^{b-c}\Cof(j_{c-1}))$};

        \node at (-5, 4){$\bullet$};
        \node at (-5.5, 4.5)[scale=0.85]{$\Sigma^{-1}(h_2P^k h_2[b-c+1])$};
        \node at (-5.5, -2){$\Ext(\Sigma^{-1}T)$};

        \node at (11, 4){$\bullet$};
        \node at (11.5, 4.5)[scale=0.85]{$h_2P^k h_2[b-c+1]$};
        \node at (11.5, -2){$\Ext(T)$};

        \draw[->] (-5,4) -- (-0.2,4);
        \node at (-2.5, 3.5)[scale=0.8]{$d_0^p$};
        \draw[->] (1,0) -- (5.8,1);
        \node at (3.3, 1.1)[scale=0.8]{$d_1^\partial$};
        \draw[->, dashed] (6,1) -- (10.8,4);
        \node at (8.2, 3.3)[scale=0.8]{$d_{r-m}^i$};
    \end{tikzpicture}\]
    since $d^\partial_1$ has no crossing for degree reasons, there is an $(i, E_{r-m+2})$-extension
    \[d_{r-m}^{i, E_{r-m+2}}(a_{c-1}[b-c])=h_2P^kh_2[b-c+1].\]

    By \cite[Corollary~6.23]{linwangxu2025lastkervaire}, this $(i, E_{r-m+2})$-extension can be stretched to a $(i, E_\infty)$-extension with no crossing, as \cref{diffofimj} implies that elements of higher filtration than $a_{c-1}[b-c]$ are permanent cycles. By the Generalized Leibniz Rule \cite[Theorem~6.1]{linwangxu2025lastkervaire} and \cref{eq1}, 
    \[\begin{tikzpicture}[scale=0.5, d1/.style={draw={red}}]
        \node at (1,0){$\bullet$};
        \node at (1.5, -0.5)[scale=0.85]{$1[b]$};
        \node at (0,4){$\bullet$};
        \node at (-0.8,4.5)[scale=0.85]{$a_{c-1}[b-c]$};
        \draw[->] (1,0) -- (0.1, 3.9);
        \node at (0, 2)[scale=0.8]{$d_{m}$};
        \node at (0.5, -2){$\Ext(\Sigma^{b-c}\Cof(j_{c-1}))$};

        \node at (6, 0){$\bullet$};
        \node at (6.5, -0.6)[scale=0.85]{$1[b]$}; 
        \node at (5, 7){$\bullet$};
        \node at (5, 7.5)[scale=0.85]{$h_2P^kh_2[b-c+1]$};
        \node at (6.5, -2){$\Ext(T)$};

        \draw[->] (1,0) -- (5.8,0);
        \node at (3.5, 0.4)[scale=0.8]{$d_0^i$};
        \draw[->] (0,4) -- (4.8,7);
        \node at (2.4, 6.3)[scale=0.8]{$d_{r-m}^i$};

        \draw[->, dashed] (6,0) -- (5,6.8);
        \node at (6, 3)[scale=0.8]{$d_r$};
    \end{tikzpicture}\]
    we have the differential
    \begin{equation}\label{eq3}
        d_r(1[b])=h_2P^kh_2[b-c+1]
    \end{equation}
    in the Adams spectral sequence of $T$. Pushing forward \cref{eq3} along the natural inclusions \[T \hookrightarrow \bbR P^{b}_{b-c} \hookrightarrow \bbR P^{\infty}_{b-c},\]
    we obtain that $1[b]$ must support a $d_{r_4}$-differential for $r_4\geq r$. We claim that \[\Ext^{b-1, >r}(\bbR P^{\infty}_{b-c})=0;\] 
    hence, $r_4=r=4k+2$.

    Consider the algebraic Atiyah--Hirzebruch spectral sequence in \cref{eq:algAHSS}. If $v_2(w-s-1) = 4k+3$, the only possible element of higher filtration is $h_0^2P^kh_2[b-c+4]$. Since $b$ is odd, $c=8k+8$, we have $b-c+4$ is odd, and thus, the nontrivial $Sq^1$-action on $H^{b-c+5}(\bbR P^\infty_{b-c})$ implies the algebraic Atiyah--Hirzebruch differential
    \[d_1(h_0P^kh_2[b-c+5])=h_0^2P^kh_2[b-c+4].\]
    Therefore, $\Ext^{b-1, >r}(\bbR P^{\infty}_{b-c})$ vanishes.
\end{proof}

We conclude this section by giving an example of a $d_3$-differential on the $0$-line, which provides a standard low-length instance of the same general pattern illustrated in \cref{fig:fil0diagram}.

\begin{prop}\label{ex:d3}
    The element $\frac{\theta}{\rho^{32}\tau^{63}}$ supports a nontrivial $d_3$-differential.
\end{prop}
\begin{proof}
    Consider the case $w-s=65$. By similar arguments in \cref{ex:d2}, we have
\[d_2(\frac{\theta}{\rho^{63}\tau^{63}})= \frac{\theta}{\tau^{94}}h_5^2+\frac{\theta}{\rho^{30} \tau^{79}} h_1h_5,\]
and
\[d_2(\frac{\theta}{\rho^{33+i}\tau^{63}}) =\frac{\theta}{\rho^{i} \tau^{79}} h_1h_5, \quad (0\leq i \leq 29).\]
Correspondingly, there are differentials
\[d_2(1[-65])=h_1h_5[-98]\]
in the Adams spectral sequence of $\bbR P^\infty_{-w}$ for $98\leq w\leq 128$.

Consider the cofiber sequence
\begin{equation}\label{eq:d3}
    \Sigma^{-1}\bbR P^\infty_{-97} \overset{\delta}{\longrightarrow} \bbR P^{-98}_{-99} \overset{i}{\longrightarrow} \bbR P^{\infty}_{-99} \overset{p}{\longrightarrow} \bbR P^{\infty}_{-97}.
\end{equation}
Since $\bbR P^{-98}_{-99}\simeq \Sigma^{-99}\Cof(2)$, there is an $h_0$-extension 
\[h_0\cdot h_1[-98]=h_1^2[-99]\]
in $\Ext(\bbR P^{-98}_{-99})$. Therefore, the Leibniz rule implies there is an Adams differential
\begin{align*}
    d_2(h_1h_5[-98])&= d_2(h_5\cdot h_1[-98])\\
    &=d_2(h_5)\cdot h_1[-98]\\
    &=h_4^2\cdot (h_0\cdot h_1[-98])\\
    &=h_1^2h_4^2[-99]
\end{align*}
in $\bbR P^{-97}_{-99}$.

Since there is a nontrivial action
\[Sq^{16}: H^{-99}(\bbR P^{\infty}_{-99}) \to H^{-83}(\bbR P^{\infty}_{-99}),\]
we have a $\delta$-extension
\[d_{1}^{\delta}(h_1^2h_4[-83])=h_1^2h_4^2[-99]\]
in the long exact sequence of $\Ext$-groups induced by \cref{eq:d3}. Then by the Generalized Mahowald Trick \cite[Theorem~6.12]{linwangxu2025lastkervaire}, 
\[\begin{tikzpicture}[scale=0.5, d1/.style={draw={red}}]
        \node at (1,0){$\bullet$};
        \node at (1.5, -0.5)[scale=0.85]{$h_1h_5[-98]$};
        \node at (0,3){$\bullet$};
        \node at (-0.5,3.5)[scale=0.85]{$h_1^2h_4^2[-99]$};
        \draw[->] (1,0) -- (0.1, 2.9);
        \node at (0.3, 1.4)[scale=0.8]{$d_2$};
        \node at (0.5, -2){$\Ext(\bbR P^{-98}_{-99})$};

        \node at (6, 0){$\bullet$};
        \node at (6.5, -0.5)[scale=0.85]{$h_1h_5[-98]$};
        \node at (6.5, -2){$\Ext(\bbR P^{\infty}_{-99})$};

        \node at (-5, 1.5){$\bullet$};
        \node at (-5.5, 2)[scale=0.85]{$h_1^2h_4[-83]$};
        \node at (-5.5, -2){$\Ext(\Sigma^{-1}\bbR P^{\infty}_{-97})$};

        \node at (11, 1.5){$\bullet$};
        \node at (11.5, 2)[scale=0.85]{$h_1^2h_4[-83]$};
        \node at (11.5, -2){$\Ext(\bbR P^{\infty}_{-97})$};

        \draw[->] (-5,1.5) -- (-0.2, 3);
        \node at (-2.5, 2.8)[scale=0.8]{$d_1^{\delta}$};
        \draw[->] (1,0) -- (5.8,0);
        \node at (3.3, 0.4)[scale=0.8]{$d_0^i$};
        \draw[->, dashed] (6,0) -- (10.8,1.3);
        \node at (8.2, 1)[scale=0.8]{$d_{1}^p$};
    \end{tikzpicture}\]
since $d^i_0$ has no crossing for filtration reasons, there is a $(p, E_3)$-extension \[d_1^{p, E_3}(h_1h_5[-98])=h_1^2h_4[-83].\]

By \cite[Corollary~6.23]{linwangxu2025lastkervaire}, this $(p, E_3)$-extension can be stretched to a $(p, E_\infty)$-extension with no crossing for degree reasons. By the Generalized Leibniz Rule \cite[Theorem~6.1]{linwangxu2025lastkervaire}, 
\[\begin{tikzpicture}[scale=0.5, d1/.style={draw={red}}]
        \node at (1,0){$\bullet$};
        \node at (1.5, -0.5)[scale=0.85]{$1[-65]$};
        \node at (0,3){$\bullet$};
        \node at (-0.8,3.5)[scale=0.85]{$h_1h_5[-98]$};
        \draw[->] (1,0) -- (0.1, 2.9);
        \node at (0, 1.5)[scale=0.8]{$d_2$};
        \node at (0.5, -2){$\Ext(\bbR P^{\infty}_{-99})$};

        \node at (6, 0){$\bullet$};
        \node at (6.5, -0.6)[scale=0.85]{$1[-65]$}; 
        \node at (5, 5){$\bullet$};
        \node at (5, 5.5)[scale=0.85]{$h_1^2h_4[-83]$};
        \node at (6.5, -2){$\Ext(\bbR P^{\infty}_{-97})$};

        \draw[->] (1,0) -- (5.8,0);
        \node at (3.5, 0.4)[scale=0.8]{$d_0^p$};
        \draw[->] (0,3) -- (4.8,5);
        \node at (2.4, 4.5)[scale=0.8]{$d_{1}^p$};

        \draw[->, dashed] (6,0) -- (5,4.9);
        \node at (6, 2.5)[scale=0.8]{$d_3$};
    \end{tikzpicture}\]
we have $d_3(1[-65])=h_1^2h_4[-83]$ in the Adams spectral sequence of $\bbR P^\infty_{-97}$. Correspondingly, there is a genuine $C_2$-equivariant Adams differential
\[d_3(\frac{\theta}{\rho^{32}\tau^{63}})=\frac{\theta}{\rho^{14}\tau^{71}}h_1^2h_4 +\Indet.\]
\end{proof}

\begin{proof}[Proof of \cref{thmB}]
    \cref{ex:d2} and \cref{ex:d3} give the existence of $d_2$- and $d_3$-differentials, while \cref{longestdiff} proves the existence of $d_r$-differentials for $r\geq 4$.
\end{proof}

\section{The image-of-\texorpdfstring{$J$}{J} differential}\label{sec:imjdiff}
Recall that $a_{8k+7}$ denotes the element in $\Ext(S^0)$ that detects $j_{8k+7}\in\pi_{8k+7}J$. In this section, we prove the classical Adams differential needed in the proof of \cref{arbitrarylongdiff}. More precisely, we identify the differential supported by a class one $h_0$-step below $a_{8k+7}$. 

For $k\ge1$, let $r_k:= v_2(k+1)+2$. 

\begin{theorem}[\cref{imjdiffintro}]\label{diffofimj}
    There exists an element $a_{8k+7}'$ satisfying $a_{8k+7}'\cdot h_0=a_{8k+7}$ such that it survives to the $E_{r_k}$-page of the Adams spectral sequence and supports a nontrivial $d_{r_k}$-differential
    \[d_{r_k}(a_{8k+7}')=h_2\cdot P^kh_2.\]
\end{theorem}
\begin{proof}
    The cases $k=1$, $2$, $3$ follow from the differentials $d_3(h_0^2h_4)=h_2Ph_2$, $d_2(h_0i)=h_2P^2h_2$, and $d_4(h_0^9h_5)=h_2P^3h_2$ respectively. We assume $k\ge4$ from now on.

    Consider the fiber sequence
    \begin{equation}\label{eq:sec5}
        S^0\xrightarrow{f}S^0\xrightarrow{g}\Cof(2),
    \end{equation}
    where $f$ is the multiplication by 2. Since $j_{8k+7}\in\pi_{8k+7}S^0$ is not 2-divisible and has Adams filtration $4k+3-r_k$, $g_*(j_{8k+7})$ is a nontrivial element in $\pi_{8k+7}\Cof(2)$ with Adams filtration at least $4k+3-r_k$. By the $1/5$-vanishing line in the Adams spectral sequence of $\Cof(2)$ (\cite[Theorem~5]{mahowald1970orderimj}, \cite[Proposition~15.8]{burklundhahnsenger2023boundaries}, \cite[Theorem~1.1]{chang2025vanishingline}), since $4k+3-r_k\ge(8k+7)/5+5$ for $k\ge4$, the Adams filtration of $g_*(j_{8k+7})$ is at least $(8k+7)/2-1.5=4k+2$.

    We now proceed by reversing the argument of \cite[Theorem~5]{ma2024geometricboundary}. Consider the synthetic lift of the fiber sequence in \cref{eq:sec5}
    \[S^{0,1}\xrightarrow{\hat f}S^{0,0}\xrightarrow{\hat g}\nu\Cof(2),\]
    where $\hat f$ is detected by $h_0$. Let $\hat\jmath_{8k+7}\in\pi_{8k+7,12k+10-r_k}S^{0,0}$ be a synthetic lift of $j_{8k+7}$, which is detected by $a_{8k+7}\in\pi_{8k+7,12k+10-r_k}S^{0,0}/\lambda$. Since the Adams filtration of $\hat g(\hat\jmath_{8k+7})$ is at least $4k+2$, $\hat g(\hat\jmath_{8k+7})$ is divisible by $\lambda^{r_k-1}$.

    Let $Y_{r_k-1}$ denote the fiber of the composite
    \[Y_{r_k-1}=\Fib (S^{0,0}\to \nu\Cof(2)\to\nu\Cof(2)/\lambda^{r_k-1}).\]
    Then $\hat\jmath_{8k+7}$ can be lifted to an element $\tilde\jmath_{8k+7}\in\pi_{*,*}Y_{r_k-1}$. Consider the maps of fiber sequences
    \[\xymatrix{
    Y_{r_k-1} \ar[r] \ar[d] & S^{0,0} \ar[r] \ar[d] & \nu\Cof(2)/\lambda^{r_k-1} \ar[d]^{\txt{id}}\\
    S^{0,1}/\lambda^{r_k-1} \ar[r]^{\hat f} \ar[d]^{\pi} & S^{0,0}/\lambda^{r_k-1} \ar[r]^(0.45){\hat g} \ar[d] & \nu\Cof(2)/\lambda^{r_k-1} \ar[d] \\
    S^{0,1}/\lambda \ar[r]^{\hat f} & S^{0,0}/\lambda \ar[r]^(0.45){\hat g} & \nu\Cof(2)/\lambda.
    \ar@{} [uull];[uull]+(0,10)^{\tilde\jmath_{8k+7}}
    \ar@{} [uul];[uul]+(0,10)^{\hat\jmath_{8k+7}}
    \ar@{} [ll];[ll]+(0,-10)_{a_{8k+7}'}
    \ar@{} [l];[l]+(0,-10)_{a_{8k+7}}
    }\]
    Let $a_{8k+7}'$ be the image of $\tilde\jmath_{8k+7}$ in $\pi_{*,*}S^{0,1}/\lambda$. On the one hand, $\hat f(a_{8k+7}')=a_{8k+7}$ implies that $a_{8k+7}'\cdot h_0=a_{8k+7}$. On the other hand, $a_{8k+7}'$ is a $d_{r_k-1}$-cycle since it can be lifted along $\pi$. However, $a_{8k+7}'$ cannot be a permanent cycle; otherwise, $j_{8k+7}$ detected by $a_{8k+7}$ is not the generator of $\pi_{8k+7}J$. This forces the nontrivial differential 
    \[d_{r_k}(a_{8k+7}')=h_2P^kh_2\] 
    for degree reasons.
\end{proof}

\section{Geometric construction}\label{sec:geoconstr}
In this section, we give a direct geometric proof of \cref{thmA}. Assuming the existence of linearly independent vector fields on $S^n$, we explicitly construct unstable $C_2$-equivariant maps between representation spheres whose $H\underline{\bbF_2}$-Hurewicz images are the classes $\theta/(\rho^k\tau^n)$. Conversely, we show that the existence of such a class in the Hurewicz image forces the corresponding vector fields to exist.

\begin{theorem}\label{thm:unstablegeoconstr}
    If $S^n$ admits $k$ linearly independent vector fields, then there exists an unstable $C_2$-equivariant map $F:S^{(n+k+2)\sigma}\to S^{n+2}$ representing $\frac{\theta}{\rho^k\tau^n}\in\pi_{k,n+k+2}^{C_2}S_{C_2}$.
\end{theorem}
\begin{proof}
    \cite{adams1962vector} ensures that $k\le\psi(n+1)-1$. Within this range, by \cite{hurwitz1922komposition} and \cite{radon1922lineare}, there exist Hurwitz-Radon matrices $A_1,\dots,A_k\in O(n+1)$ satisfying
    \[A_i^2=-I_{n+1};\ A_iA_j=-A_jA_i.\ (i\ne j,\ i,j=1,\dots,k)\]\
    Let $A_0=I_{n+1}$. Then for any $\bfv\in\bbR^{n+1}$, $A_0\bfv,\dots,A_k\bfv$ are orthogonal with $\|A_i\bfv\|=\|\bfv\|$ for all $0\le i\le k$. 

    Consider the $C_2$-equivariant quadratic map
    \[f:\bbR^{(n+k+2)\sigma}\to \bbR^{n+2}; \quad (\bfv,x_0,\dots,x_k)\mapsto \left(2\sum_{i=0}^kx_iA_i\bfv, \sum_{i=0}^kx_i^2-\|\bfv\|^2\right),\]
    where $\bfv\in\bbR^{n+1}$ and $x_0,\dots,x_k\in\bbR$, and the $C_2$-action is multiplication by $-1$. Note that
    \begin{align*} \|f(\bfv,x_0,\dots,x_k)\|^2 &= 4\sum_{i=0}^kx_i^2\|\bfv\|^2+\left(\sum_{i=0}^kx_i^2-\|\bfv\|^2\right)^2 \\ &= \left(\sum_{i=0}^kx_i^2+\|\bfv\|^2\right)^2 \\ &= \|(\bfv,x_0,\dots,x_k)\|^4, \end{align*}
    and so we can define the $C_2$-equivariant map $F$ as the one-point compactification of $f$.

    Next, we compute the induced map
    \[H^{n+2}(F;\underline{\bbF_2}):H^{n+2}(S^{n+2};\underline{\bbF_2})\to H^{n+2}(S^{(n+k+2)\sigma};\underline{\bbF_2})\]
    and show that it is nontrivial. By definition, this corresponds to the induced map on classical cohomology between the respective orbit spaces:
    \[H^{n+2}(S^{n+2};\bbF_2)\to H^{n+2}(S^{(n+k+2)\sigma}/C_2;\bbF_2).\]

    Letting $SX$ denote the unreduced suspension of a space $X$, we have 
    \[S^{(n+k+2)\sigma}/C_2\cong S(\bbR P^{n+k+1}).\]
    Since $f$ is a homogeneous quadratic map, the map $F/C_2$ is the unreduced suspension of the restricted orbit map 
    \[\tilde f=f|_{S((n+k+2)\sigma)}/C_2:\bbR P^{n+k+1}\to S^{n+1}.\]

    The $n$-skeleton of $\bbR P^{n+k+1}$, represented by the points $\{(\bfv,0,\dots,0) \mid \|\bfv\|=1\}$, is mapped entirely to the point $(\mathbf 0,-1)$ under $\tilde f$. Furthermore, $\tilde f$ maps the $(n+1)$-cell of $\bbR P^{n+k+1}$, given by
    \[\{(\bfv,x_0,0,\dots,0) \mid \|\bfv\|^2+x_0^2=1, x_0>0\},\]
    bijectively onto $S^{n+1}\setminus\{(\mathbf0,-1)\}$. Its inverse sends a point $(\bfu,y)$ to
    \[\left(\frac{\bfu}{\sqrt{2+2y}},\sqrt{\frac{1+y}2},0,\dots,0\right).\]
    Therefore, $\tilde f$ induces an isomorphism on $H^{n+1}$, which yields the desired result.
\end{proof}

\begin{rmk}
    While $F$ is the unreduced suspension of
    \[f|_{S((n+k+2)\sigma)}: S((n+k+2)\sigma) \to S^{n+1},\]
    $F$ admits no further desuspension as a based $C_2$-equivariant map.
\end{rmk}

\begin{theorem}\label{thm:geoproof}
    If $\frac\theta{\rho^k\tau^n}$ lies in the Hurewicz image of $H\underline{\bbF_2}$, then $S^n$ admits $k$ linearly independent vector fields.
\end{theorem}

\begin{proof}
    First, observe that $\frac{\theta}{\rho^k\tau^n}\in\pi_{k,n+k+2}S_{C_2}$. By Hauschild's equivariant suspension theorem \cite[Theorem~II.4]{hauschild1977aquivariante}, $\Sigma^\infty_+$ induces an isomorphism
    \[[S^{(n+k+2)\sigma+\delta},S^{n+2+\delta}]^{C_2}\xrightarrow{\cong}\pi_{k,n+k+2}S_{C_2},\]
    where $\delta=\max\{0,k-n\}$. Consequently, there exists an unstable equivariant map
    \[F:S^{(n+k+2)\sigma+\delta}\to S^{n+2+\delta}\]
    that induces a nontrivial homomorphism on $H^{n+2+\delta}(-;\underline{\bbF_2})$. Passing to orbit spaces, $F$ induces a map
    \[\Sigma^{1+\delta}\bbR P^{n+k+1}\to S^{n+2+\delta}\]
    which remains nontrivial on $H^{n+2+\delta}(-;\bbF_2)$. Furthermore, since any map from $\Sigma^{1+\delta}\bbR P^n$ to $S^{n+2+\delta}$ is nullhomotopic, this map factors through the quotient, yielding
    \[\Sigma^{1+\delta}\bbR P_{n+1}^{n+k+1}\to S^{n+2+\delta}.\]
    This implies that the bottom cell of the stunted projective spectrum $\bbR P_{n+1}^{n+k+1}$ splits. Finally, by \cite{james1959stiefel} and \cite{atiyah1961thom}, the splitting of the bottom cell of $\bbR P_{n+1}^{n+k+1}$ implies that $S^n$ admits $k$ linearly independent vector fields.
\end{proof}

\section{The Hurewicz image of \texorpdfstring{$H\underline{\bbZ}$}{HZ} and \texorpdfstring{$H\underline{A}$}{HA}} \label{sec:HZHA} 
We determine the Hurewicz image of $H\underline{\bbZ}$ and $H\underline{A}$ in this section. Recall the following from \cite{dugger2005atiyah, lewis1987roggraded}.
\begin{prop}
~\begin{enumerate}
    \item The $RO(C_2)$-graded abelian group structure of $H\underline{\bbZ}_{*, *}$ is
    \[H\underline{\bbZ}_{s,w}=\begin{cases}
        \bbZ \quad &\mathrm{if}\ s=0\ \mathrm{and}\ w\ \mathrm{is\ even},\\
        \bbF_2 \quad &\mathrm{if}\ s<0 \ \mathrm{and}\ s-w=2k\ \mathrm{for\ } k\geq 0,\\
        \bbF_2 \quad &\mathrm{if}\ s\geq0 \ \mathrm{and}\ s-w=2k-1\ \mathrm{for\ } k\leq -1,\\
        0\quad &\mathrm{otherwise.}
    \end{cases}\]
    \item The $RO(C_2)$-graded abelian group structure of $H\underline{A}_{*, *}$ is
    \[H\underline{A}_{s,w}=\begin{cases}
        A(C_2)\quad &\mathrm{if}\ s=w=0,\\
        \bbZ \quad &\mathrm{if}\ s=0\ \mathrm{and}\ w=2k\ \mathrm{for\ } k\neq0,\\
        \bbZ \quad &\mathrm{if}\ s=w\neq 0,\\
        \bbF_2 \quad &\mathrm{if}\ s<0 \ \mathrm{and}\ s-w=2k\ \mathrm{for\ } k\geq 1,\\
        \bbF_2 \quad &\mathrm{if}\ s\geq0 \ \mathrm{and}\ s-w=2k-1\ \mathrm{for\ } k\leq -1,\\
        0\quad &\mathrm{otherwise. \rlap{\hspace{17.7em} \tqed}}
    \end{cases}\]
\end{enumerate}
\end{prop}

\begin{figure}
        \centering
        \begin{tikzpicture}[scale=0.5]
\draw[thick,dashed, ->] (-6,0) -- (6,0);
\draw[thick,dashed, ->] (0,-6) -- (0, 7.5);

\node at (-5, 7){$\bullet= \bbF_2$};
\node at (-5, 6){$\square= \bbZ$};
\node at (-5, 5){$\blacksquare= A(C_2)$};

\node at (6.5,0){s};
\node at (0, 8){w};

\node at (-0.6, 0.3)[scale=0.75]{0};
\node at (-0.6, 2)[scale=0.75]{2};
\node at (-0.6, 4)[scale=0.75]{4};
\node at (-0.6, 6)[scale=0.75]{6};
\node at (0.6, -2)[scale=0.75]{-2};
\node at (0.6, -4)[scale=0.75]{-4};

\node at (2, -0.4)[scale=0.75]{2};
\node at (4, -0.4)[scale=0.75]{4};
\node at (-2, 0.4)[scale=0.75]{-2};
\node at (-4, 0.4)[scale=0.75]{-4};

\node at (0,2){$\square$};
\node at (0.4, 1.7)[scale=0.75]{$\theta$};

\node at (0,3){$\bullet$};
\node at (1,4){$\bullet$};
\node at (2,5){$\bullet$};
\node at (3,6){$\bullet$};
\node at (4,7){$\bullet$};

\node at (0,4){$\square$};

\node at (0,5){$\bullet$};
\node at (1,6){$\bullet$};
\node at (2,7){$\bullet$};

\node at (0,6){$\square$};

\node at (0,7){$\bullet$};

\node at (0.4, 2.7)[scale=0.75]{$\frac{\theta}{\tau}$};
\node at (0.4, 3.7)[scale=0.75]{$\frac{\theta}{\tau^2}$};
\node at (0.4, 4.7)[scale=0.75]{$\frac{\theta}{\tau^3}$};
\node at (0.4, 5.7)[scale=0.75]{$\frac{\theta}{\tau^4}$};
\node at (0.4, 6.7)[scale=0.75]{$\frac{\theta}{\tau^5}$};

\node at (0,0){$\square$};
\node at (-1,-1){$\bullet$};
\node at (-2,-2){$\bullet$};
\node at (-3,-3){$\bullet$};
\node at (-4,-4){$\bullet$};
\node at (-5,-5){$\bullet$};

\node at (-1.3, -0.7)[scale=0.75]{$\rho$};
\node at (-2.3, -1.7)[scale=0.75]{$\rho^2$};
\node at (-3.3, -2.7)[scale=0.75]{$\rho^3$};
\node at (-4.3, -3.7)[scale=0.75]{$\rho^4$};
\node at (-5.3, -4.7)[scale=0.75]{$\rho^5$};

\node at (0,-2){$\square$};
\node at (-1,-3){$\bullet$};
\node at (-2,-4){$\bullet$};
\node at (-3,-5){$\bullet$};

\node at (0,-4){$\square$};
\node at (-1,-5){$\bullet$};

\node at (-0.4, -1.7)[scale=0.75]{$\tau^2$};
\node at (-0.4, -3.7)[scale=0.75]{$\tau^4$};

\draw[thick,dashed, ->] (9,0) -- (21,0);
\draw[thick,dashed, ->] (15,-6) -- (15, 7.5);

\node at (21.5,0){s};
\node at (15, 8){w};

\node at (14.4, 0.3)[scale=0.75]{0};
\node at (14.4, 2)[scale=0.75]{2};
\node at (14.4, 4)[scale=0.75]{4};
\node at (14.4, 6)[scale=0.75]{6};
\node at (15.6, -2)[scale=0.75]{-2};
\node at (15.6, -4)[scale=0.75]{-4};

\node at (17, -0.4)[scale=0.75]{2};
\node at (19, -0.4)[scale=0.75]{4};
\node at (13, 0.4)[scale=0.75]{-2};
\node at (11, 0.4)[scale=0.75]{-4};

\node at (15,2){$\square$};
\node at (15.4, 1.7)[scale=0.75]{$\theta$};

\node at (15,3){$\bullet$};
\node at (16,4){$\bullet$};
\node at (17,5){$\bullet$};
\node at (18,6){$\bullet$};
\node at (19,7){$\bullet$};

\node at (15,4){$\square$};

\node at (15,5){$\bullet$};
\node at (16,6){$\bullet$};
\node at (17,7){$\bullet$};

\node at (15,6){$\square$};

\node at (15,7){$\bullet$};

\node at (15.4, 2.7)[scale=0.75]{$\frac{\theta}{\tau}$};
\node at (15.4, 3.7)[scale=0.75]{$\frac{\theta}{\tau^2}$};
\node at (15.4, 4.7)[scale=0.75]{$\frac{\theta}{\tau^3}$};
\node at (15.4, 5.7)[scale=0.75]{$\frac{\theta}{\tau^4}$};
\node at (15.4, 6.7)[scale=0.75]{$\frac{\theta}{\tau^5}$};

\node at (15,0){$\blacksquare$};
\node at (14,-1){$\square$};
\node at (13,-2){$\square$};
\node at (12,-3){$\square$};
\node at (11,-4){$\square$};
\node at (10,-5){$\square$};

\node at (16,1){$\square$};
\node at (17,2){$\square$};
\node at (18,3){$\square$};
\node at (19,4){$\square$};
\node at (20,5){$\square$};

\node at (13.6, -0.7)[scale=0.75]{$\rho$};
\node at (12.6, -1.7)[scale=0.75]{$\rho^2$};
\node at (11.6, -2.7)[scale=0.75]{$\rho^3$};
\node at (10.6, -3.7)[scale=0.75]{$\rho^4$};
\node at (9.6, -4.7)[scale=0.75]{$\rho^5$};

\node at (16.6, 0.7)[scale=0.75]{$\eta$};
\node at (17.6, 1.7)[scale=0.75]{$\frac{\eta}{\rho}$};
\node at (18.6, 2.7)[scale=0.75]{$\frac{\eta}{\rho^2}$};
\node at (19.6, 3.7)[scale=0.75]{$\frac{\eta}{\rho^3}$};
\node at (20.6, 4.7)[scale=0.75]{$\frac{\eta}{\rho^4}$};

\node at (15,-2){$\square$};
\node at (14,-3){$\bullet$};
\node at (13,-4){$\bullet$};
\node at (12,-5){$\bullet$};

\node at (15,-4){$\square$};
\node at (14,-5){$\bullet$};

\node at (14.6, -1.7)[scale=0.75]{$\tau^2$};

\node at (14.6, -3.7)[scale=0.75]{$\tau^4$};

\end{tikzpicture}
        \caption{$H\underline{\mathbb Z}_{s, w}$ and $H\underline{A}_{s, w}$}
        \label{HZHAfigure}
    \end{figure}
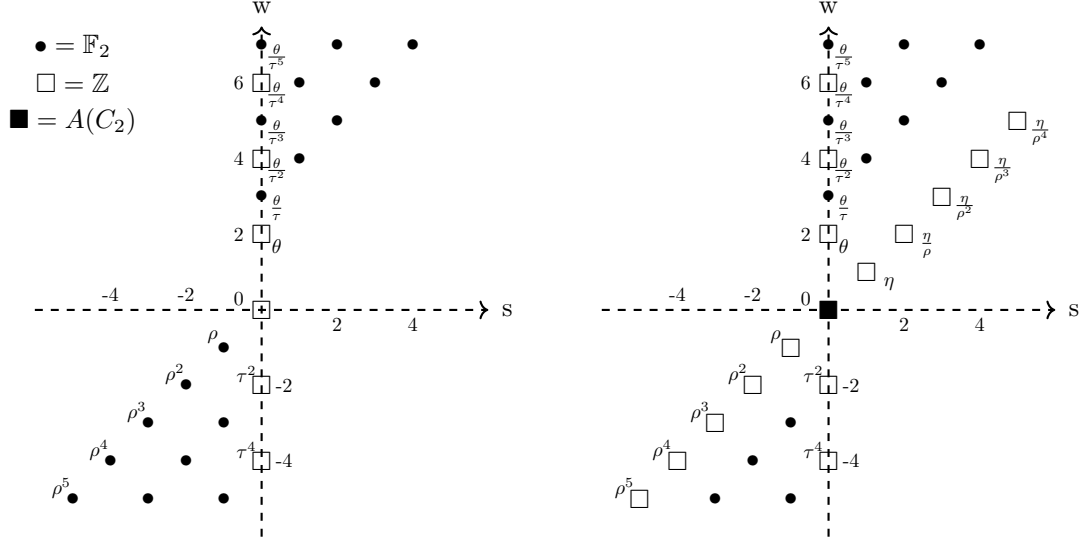

The additive structures of $H\underline{\bbZ}_{*, *}$ and $H\underline{A}_{*, *}$ are depicted in \cref{HZHAfigure}, where each black dot represents $\bbF_2$, each hollow square represents $\bbZ$, and the black square represents $A(C_2)$. References for multiplicative structures are \cite[Theorem 2.8]{dugger2005atiyah} and \cite[Theorem 7.7]{sikora2022on}.

Since the Hurewicz map $\iota_{H\underline{\bbF_2}}: S^0 \to H\underline{\bbF_2}$ factors through
\[S^0\longrightarrow H\underline{A} \to H\underline{\bbZ}\to H\underline{\bbF_2},\]
by abuse of notation, we denote the generators of $H\underline{A}_{*, *}$ and $H\underline{\bbZ}_{*, *}$ by the same names as their images in $H\underline{\bbF_2}_{*, *}$ and the $\bbF_2[\rho]$-module structure is inherited similarly. For $H\underline{A}$, $H\underline{A}_{0,0}=A(C_2)=\bbZ\{1, [C_2]\}$. There is an additional generator $\eta\in H\underline{A}_{1,1}$, which is the Hurewicz image of the non-nilpotent element $\eta_{C_2}\in \pi_{1,1}^{C_2}$ \cite{belmontxuzhang2024reduced}, since both satisfy the multiplicative relation $\eta \cdot \rho = 2-[C_2]$.  The notation $\frac{\eta}{\rho^j}$ suggests that multiplication by $\rho$ induces isomorphisms  \[\cdot \rho: \pi_{n+1,n+1}H\underline{A} \to \pi_{n,n}H\underline{A}, \qquad n\geq 1.\]

\vspace{0.2cm}

\begin{theorem}[\cref{mainHurewiczofHZHA}]\label{HurewiczofHZHA}
~\begin{enumerate}
    \item The $RO(C_2)$-graded Hurewicz image of $H\underline{\bbZ}$ consists of
    \[\im (\iota_{H\underline{\bbZ}} )_{s, w}=\begin{cases}
        \bbZ\{1\} \quad &\mathrm{if}\ s=w=0,\\
        \bbZ\{\frac{\theta}{\tau^{2k-2}}\} \quad &\mathrm{if}\ s=0\ \mathrm{and}\ w=2k\ \mathrm{for}\ k > 0,\\
        \bbZ\{2\tau^{2k}\} \quad &\mathrm{if}\ s=0\ \mathrm{and}\ w=-2k\ \mathrm{for}\ k>0,\\
        \bbF_2\{\rho^{-s}\} \quad &\mathrm{if}\ s=w<0,\\
        \bbF_2\{\frac{\theta}{\rho^s\tau^{w-s-2}}\} \quad &\mathrm{if}\ 0\leq s< \psi(w-s-1) \ \mathrm{and}\ s-w=2k-1\ \mathrm{for\ } k\leq -1,\\
        0\quad &\mathrm{otherwise.}
    \end{cases}\]
    \item The $RO(C_2)$-graded Hurewicz image of $H\underline{A}$ consists of
    \[\im (\iota_{H\underline{A}} )_{s, w}=\begin{cases}
        A(C_2)\quad &\mathrm{if}\ s=w=0,\\
        \bbZ\{\frac{\theta}{\tau^{2k-2}}\} \quad &\mathrm{if}\ s=0\ \mathrm{and}\ w=2k\ \mathrm{for}\ k > 0,\\
        \bbZ\{2\tau^{2k}\} \quad &\mathrm{if}\ s=0\ \mathrm{and}\ w=-2k\ \mathrm{for}\ k>0,\\
        \bbZ\{\rho^{-s}\} \quad &\mathrm{if}\ s=w<0,\\
        \bbZ\{2^{s-n(s)-1}\frac{\eta}{\rho^{s-1}}\} \quad &\mathrm{if}\ s=w>0,\\
        \bbF_2\{\frac{\theta}{\rho^s\tau^{w-s-2}}\} \quad &\mathrm{if}\ 0\leq s< \psi(w-s-1) \ \mathrm{and}\ s-w=2k-1\ \mathrm{for\ } k\leq -1,\\
        0\quad &\mathrm{otherwise.}
    \end{cases}\]
    where 
    \[n(s) = \begin{cases} 
    4t-1 & s = 8t,
    \\4t & s = 8t+j \text{ for } j = 1,2,3,4,
    \\4t+1 & s = 8t+5,
    \\4t+2 & s = 8t+6,
    \\4t+3 & s = 8t+7.
    \end{cases}\]
\end{enumerate}
\end{theorem}

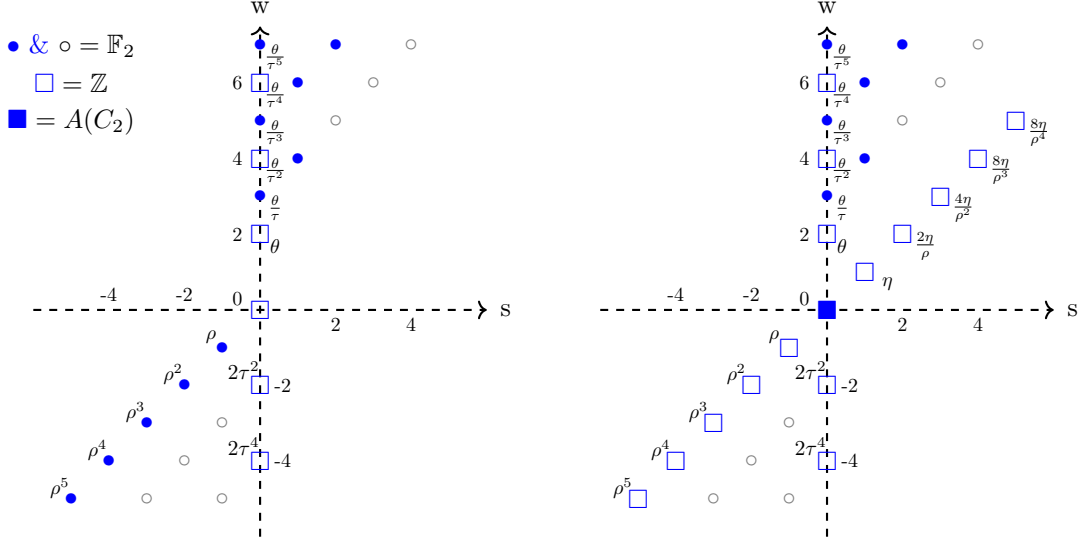
\begin{figure}
        \centering
        \begin{tikzpicture}[scale=0.5]
\draw[thick,dashed, ->] (-6,0) -- (6,0);
\draw[thick,dashed, ->] (0,-6) -- (0, 7.5);

\node at (-5, 7){$\color{blue}{\bullet}\ \&\ \color{black}\circ = \bbF_2$};
\node at (-5, 6){$\color{blue}\square\color{black}= \mathbb Z$};
\node at (-5, 5){$\color{blue}\blacksquare \color{black}= A(C_2)$};

\node at (6.5,0){s};
\node at (0, 8){w};

\node at (-0.6, 0.3)[scale=0.75]{0};
\node at (-0.6, 2)[scale=0.75]{2};
\node at (-0.6, 4)[scale=0.75]{4};
\node at (-0.6, 6)[scale=0.75]{6};
\node at (0.6, -2)[scale=0.75]{-2};
\node at (0.6, -4)[scale=0.75]{-4};

\node at (2, -0.4)[scale=0.75]{2};
\node at (4, -0.4)[scale=0.75]{4};
\node at (-2, 0.4)[scale=0.75]{-2};
\node at (-4, 0.4)[scale=0.75]{-4};

\node at (0,2){$\color{blue}\square$};
\node at (0.4, 1.7)[scale=0.75]{$\theta$};

\node at (0,3){$\color{blue}\bullet$};
\node at (1,4){$\color{blue}\bullet$};
\node at (2,5){$\color{gray}\circ$};
\node at (3,6){$\color{gray}\circ$};
\node at (4,7){$\color{gray}\circ$};

\node at (0,4){$\color{blue}\square$};

\node at (0,5){$\color{blue}\bullet$};
\node at (1,6){$\color{blue}\bullet$};
\node at (2,7){$\color{blue}\bullet$};

\node at (0,6){$\color{blue}\square$};

\node at (0,7){$\color{blue}\bullet$};

\node at (0.4, 2.7)[scale=0.75]{$\frac{\theta}{\tau}$};
\node at (0.4, 3.7)[scale=0.75]{$\frac{\theta}{\tau^2}$};
\node at (0.4, 4.7)[scale=0.75]{$\frac{\theta}{\tau^3}$};
\node at (0.4, 5.7)[scale=0.75]{$\frac{\theta}{\tau^4}$};
\node at (0.4, 6.7)[scale=0.75]{$\frac{\theta}{\tau^5}$};

\node at (0,0){$\color{blue}\square$};
\node at (-1,-1){$\color{blue}\bullet$};
\node at (-2,-2){$\color{blue}\bullet$};
\node at (-3,-3){$\color{blue}\bullet$};
\node at (-4,-4){$\color{blue}\bullet$};
\node at (-5,-5){$\color{blue}\bullet$};

\node at (-1.3, -0.7)[scale=0.75]{$\rho$};
\node at (-2.3, -1.7)[scale=0.75]{$\rho^2$};
\node at (-3.3, -2.7)[scale=0.75]{$\rho^3$};
\node at (-4.3, -3.7)[scale=0.75]{$\rho^4$};
\node at (-5.3, -4.7)[scale=0.75]{$\rho^5$};

\node at (0,-2){$\color{blue}\square$};
\node at (-1,-3){$\color{gray}\circ$};
\node at (-2,-4){$\color{gray}\circ$};
\node at (-3,-5){$\color{gray}\circ$};

\node at (0,-4){$\color{blue}\square$};
\node at (-1,-5){$\color{gray}\circ$};

\node at (-0.43, -1.6)[scale=0.8]{$2\tau^2$};
\node at (-0.43, -3.6)[scale=0.8]{$2\tau^4$};

\draw[thick,dashed, ->] (9,0) -- (21,0);
\draw[thick,dashed, ->] (15,-6) -- (15, 7.5);

\node at (21.5,0){s};
\node at (15, 8){w};

\node at (14.4, 0.3)[scale=0.75]{0};
\node at (14.4, 2)[scale=0.75]{2};
\node at (14.4, 4)[scale=0.75]{4};
\node at (14.4, 6)[scale=0.75]{6};
\node at (15.6, -2)[scale=0.75]{-2};
\node at (15.6, -4)[scale=0.75]{-4};

\node at (17, -0.4)[scale=0.75]{2};
\node at (19, -0.4)[scale=0.75]{4};
\node at (13, 0.4)[scale=0.75]{-2};
\node at (11, 0.4)[scale=0.75]{-4};

\node at (15,2){$\color{blue}\square$};
\node at (15.4, 1.7)[scale=0.75]{$\theta$};

\node at (15,3){$\color{blue}\bullet$};
\node at (16,4){$\color{blue}\bullet$};
\node at (17,5){$\color{gray}\circ$};
\node at (18,6){$\color{gray}\circ$};
\node at (19,7){$\color{gray}\circ$};

\node at (15,4){$\color{blue}\square$};

\node at (15,5){$\color{blue}\bullet$};
\node at (16,6){$\color{blue}\bullet$};
\node at (17,7){$\color{blue}\bullet$};

\node at (15,6){$\color{blue}\square$};

\node at (15,7){$\color{blue}\bullet$};

\node at (15.4, 2.7)[scale=0.75]{$\frac{\theta}{\tau}$};
\node at (15.4, 3.7)[scale=0.75]{$\frac{\theta}{\tau^2}$};
\node at (15.4, 4.7)[scale=0.75]{$\frac{\theta}{\tau^3}$};
\node at (15.4, 5.7)[scale=0.75]{$\frac{\theta}{\tau^4}$};
\node at (15.4, 6.7)[scale=0.75]{$\frac{\theta}{\tau^5}$};

\node at (15,0){$\color{blue}\blacksquare$};
\node at (14,-1){$\color{blue}\square$};
\node at (13,-2){$\color{blue}\square$};
\node at (12,-3){$\color{blue}\square$};
\node at (11,-4){$\color{blue}\square$};
\node at (10,-5){$\color{blue}\square$};

\node at (16,1){$\color{blue}\square$};
\node at (17,2){$\color{blue}\square$};
\node at (18,3){$\color{blue}\square$};
\node at (19,4){$\color{blue}\square$};
\node at (20,5){$\color{blue}\square$};

\node at (13.6, -0.7)[scale=0.75]{$\rho$};
\node at (12.6, -1.7)[scale=0.75]{$\rho^2$};
\node at (11.6, -2.7)[scale=0.75]{$\rho^3$};
\node at (10.6, -3.7)[scale=0.75]{$\rho^4$};
\node at (9.6, -4.7)[scale=0.75]{$\rho^5$};

\node at (16.6, 0.7)[scale=0.75]{$\eta$};
\node at (17.6, 1.7)[scale=0.75]{$\frac{2\eta}{\rho}$};
\node at (18.6, 2.7)[scale=0.75]{$\frac{4\eta}{\rho^2}$};
\node at (19.6, 3.7)[scale=0.75]{$\frac{8\eta}{\rho^3}$};
\node at (20.6, 4.7)[scale=0.75]{$\frac{8\eta}{\rho^4}$};

\node at (15,-2){$\color{blue}\square$};
\node at (14,-3){$\color{gray}\circ$};
\node at (13,-4){$\color{gray}\circ$};
\node at (12,-5){$\color{gray}\circ$};

\node at (15,-4){$\color{blue}\square$};
\node at (14,-5){$\color{gray}\circ$};

\node at (14.57, -1.6)[scale=0.8]{$2\tau^2$};

\node at (14.57, -3.6)[scale=0.8]{$2\tau^4$};

\end{tikzpicture}
        \caption{The Hurewicz image of $H\underline{\mathbb Z}_{s, w}$ and $H\underline{A}_{s, w}$}
        \label{hurewiczHZHAfigure}
    \end{figure}

\begin{proof}
~\begin{enumerate}
    \item Since the Hurewicz map $\iota_{H\underline{\bbF_2}}$ factors through $\iota_{H\underline{\bbZ}}$, \cref{homotopyHurewiczim} forces the result as claimed, with exceptional cases only in degrees $(s, w)$ for $s=0$ and $w$ even.

    If $w>0$, the generators $\omega_{-k}\in \pi_{0, 2k}^{C_2}$ that detects $\frac{\theta}{\tau^{2k-2}}$ are torsion-free \cite{belmontxuzhang2024reduced}, and since its image in $H\underline{\bbF_2}_{*, *}$ is nontrivial, it must have nontrivial image in $H\underline{\bbZ}_{*, *}$.

    If $w<0$, $\tau^{2k}\in H\underline{\bbZ}_{0, -2k}$ satisfies the multiplicative relation
    \[\tau^{2k} \cdot \frac{\theta}{\tau^{2k-2}} = 2.\]
    On the other hand, according to \cite{belmontxuzhang2024reduced}, there are torsion-free generators $\omega_{k} \in \pi_{0, -2k}^{C_2}$ satisfying
    \[\omega_{k}\cdot \omega_{-k} =2[C_2].\]
    Since $[C_2]$ has image $2$ in $H\underline{\bbZ}_{*, *}$, the image of $\omega_k$ in $H\underline{\bbZ}_{*, *}$ must be $2\tau^{2k}$ in $H\underline{\bbZ}_{*, *}$.
    
    \item Since the Hurewicz map $\iota_{H\underline{\bbZ}}$ factors through $\iota_{H\underline{A}}$, part (1) implies the result as claimed, except the cases for $s=w> 0$.

    Recall that $a_\sigma\in\pi_{-1,-1}^{C_2}S^0$ is the Euler class of $\sigma$, whose $H\underline{\bbF_2}$-Hurewicz image is $\rho$. According to \cite{belmontxuzhang2024reduced}, the torsion-free generator of $\pi_{s, s}^{C_2}$ is $\frac{\eta^s_{C_2}}{2^{n(s)}}$ subject to the relation
    \[\frac{\eta^N_{C_2}}{a_\sigma}= \frac{\eta^{N+1}_{C_2}}{2}\]
    for large enough $N$. In particular,
    \[\frac{\eta^s_{C_2}}{2^{n(s)}} \cdot a_\sigma^{s-1} =2^{s-n(s)-1}\cdot \eta_{C_2}.\]
    On the other hand, from the definition,
    \[\frac{\eta}{\rho^{s-1}} \cdot \rho^{s-1}=\eta\]
    in $H\underline{A}_{*, *}$.
    Therefore, the generators $\frac{\eta^s_{C_2}}{2^{n(s)}}$ all map injectively into $H\underline{A}_{s,s}$:
    \[\iota_{H\underline{A}}\left(\frac{\eta^s_{C_2}}{2^{n(s)}}\right)=2^{s-n(s)-1}\frac{\eta}{\rho^{s-1}}.\]
\end{enumerate}
\end{proof}

\cref{hurewiczHZHAfigure} represents the Hurewicz image of $H\underline{\bbZ}$ and $H\underline{A}$. The generators with names that are in the Hurewicz image are indicated in blue.

\begingroup
\raggedright
\bibliography{refs}

@article {adams1962vector,
    AUTHOR = {Adams, J. F.},
     TITLE = {Vector fields on spheres},
   JOURNAL = {Ann. of Math. (2)},
  FJOURNAL = {Annals of Mathematics. Second Series},
    VOLUME = {75},
      YEAR = {1962},
     PAGES = {603--632},
      ISSN = {0003-486X},
   MRCLASS = {57.30},
  MRNUMBER = {139178},
MRREVIEWER = {M. F. Atiyah},
       DOI = {10.2307/1970213},
       URL = {https://doi.org/10.2307/1970213},
}

@article {adams1966on,
    AUTHOR = {Adams, J. F.},
     TITLE = {On the groups {$J(X)$}. {IV}},
   JOURNAL = {Topology},
  FJOURNAL = {Topology. An International Journal of Mathematics},
    VOLUME = {5},
      YEAR = {1966},
     PAGES = {21--71},
      ISSN = {0040-9383},
   MRCLASS = {55.40},
  MRNUMBER = {198470},
MRREVIEWER = {E. Dyer},
       DOI = {10.1016/0040-9383(66)90004-8},
       URL = {https://doi.org/10.1016/0040-9383(66)90004-8},
}

@article {adams1960nonexistence,
    AUTHOR = {Adams, J. F.},
     TITLE = {On the non-existence of elements of {H}opf invariant one},
   JOURNAL = {Ann. of Math. (2)},
  FJOURNAL = {Annals of Mathematics. Second Series},
    VOLUME = {72},
      YEAR = {1960},
     PAGES = {20--104},
      ISSN = {0003-486X},
   MRCLASS = {55.40},
  MRNUMBER = {141119},
MRREVIEWER = {M.\ A.\ Kervaire},
       DOI = {10.2307/1970147},
       URL = {https://doi.org/10.2307/1970147},
}

@book {andrews2015v1,
    AUTHOR = {Andrews, Michael Joseph},
     TITLE = {The v1-periodic part of the {A}dams spectral sequence at an
              odd prime},
      NOTE = {Thesis (Ph.D.)--Massachusetts Institute of Technology},
 PUBLISHER = {ProQuest LLC, Ann Arbor, MI},
      YEAR = {2015},
     PAGES = {(no paging)},
   MRCLASS = {Thesis},
  MRNUMBER = {3427191},
       URL =
              {http://gateway.proquest.com/openurl?url_ver=Z39.88-2004&rft_val_fmt=info:ofi/fmt:kev:mtx:dissertation&res_dat=xri:pqm&rft_dat=xri:pqdiss:0831008},
}

@article {atiyah1961thom,
    AUTHOR = {Atiyah, M. F.},
     TITLE = {Thom complexes},
   JOURNAL = {Proc. London Math. Soc. (3)},
  FJOURNAL = {Proceedings of the London Mathematical Society. Third Series},
    VOLUME = {11},
      YEAR = {1961},
     PAGES = {291--310},
      ISSN = {0024-6115,1460-244X},
   MRCLASS = {57.30},
  MRNUMBER = {131880},
MRREVIEWER = {R.\ Bott},
       DOI = {10.1112/plms/s3-11.1.291},
       URL = {https://doi.org/10.1112/plms/s3-11.1.291},
}

@article {balderramaculverquigley2025motivic,
    AUTHOR = {Balderrama, William and Culver, Dominic Leon and Quigley, J.
              D.},
     TITLE = {The motivic lambda algebra and motivic {H}opf invariant one
              problem},
   JOURNAL = {Geom. Topol.},
  FJOURNAL = {Geometry \& Topology},
    VOLUME = {29},
      YEAR = {2025},
    NUMBER = {3},
     PAGES = {1489--1570},
      ISSN = {1465-3060,1364-0380},
   MRCLASS = {14F42 (55Q25 55Q45 55S10 55T15)},
  MRNUMBER = {4918111},
       DOI = {10.2140/gt.2025.29.1489},
       URL = {https://doi.org/10.2140/gt.2025.29.1489},
}

@Article{behrenshillhopkinsmahowald2020detecting,
 Author = {Behrens, M. and Hill, M. and Hopkins, M. J. and Mahowald, M.},
 Title = {Detecting exotic spheres in low dimensions using coker {{\(J\)}}},
 FJournal = {Journal of the London Mathematical Society. Second Series},
 Journal = {J. Lond. Math. Soc., II. Ser.},
 ISSN = {0024-6107},
 Volume = {101},
 Number = {3},
 Pages = {1173--1218},
 Year = {2020},
 Language = {English},
 DOI = {10.1112/jlms.12301},
 zbMATH = {7226678},
 Zbl = {1460.55017}
}

@article {behrensshah2020c2,
    AUTHOR = {Behrens, Mark and Shah, Jay},
     TITLE = {{$C_2$}-equivariant stable homotopy from real motivic stable
              homotopy},
   JOURNAL = {Ann. K-Theory},
  FJOURNAL = {Annals of K-Theory},
    VOLUME = {5},
      YEAR = {2020},
    NUMBER = {3},
     PAGES = {411--464},
      ISSN = {2379-1683},
   MRCLASS = {14F42 (55N91 55P91 55Q91)},
  MRNUMBER = {4132743},
       DOI = {10.2140/akt.2020.5.411},
       URL = {https://doi.org/10.2140/akt.2020.5.411},
}

@article {behrensmahowaldquigley2023hurewicz,
    AUTHOR = {Behrens, Mark and Mahowald, Mark and Quigley, J. D.},
     TITLE = {The 2-primary {H}urewicz image of tmf},
   JOURNAL = {Geom. Topol.},
  FJOURNAL = {Geometry \& Topology},
    VOLUME = {27},
      YEAR = {2023},
    NUMBER = {7},
     PAGES = {2763--2831},
      ISSN = {1465-3060},
   MRCLASS = {55Q45 (55Q51 55T15)},
  MRNUMBER = {4645486},
MRREVIEWER = {Andrew J. Baker},
       DOI = {10.2140/gt.2023.27.2763},
       URL = {https://doi.org/10.2140/gt.2023.27.2763},
}

@article {belmontxuzhang2024reduced,
    AUTHOR = {Belmont, Eva and Xu, Zhouli and Zhang, Shangjie},
     TITLE = {The reduced ring of the {$RO(C_2)$}-graded {$C_2$}-equivariant
              stable stems},
   JOURNAL = {Proc. Amer. Math. Soc. Ser. B},
  FJOURNAL = {Proceedings of the American Mathematical Society. Series B},
    VOLUME = {11},
      YEAR = {2024},
     PAGES = {1--14},
      ISSN = {2330-1511},
   MRCLASS = {55Q91 (55Q45)},
  MRNUMBER = {4685813},
       DOI = {10.1090/bproc/203},
       URL = {https://doi.org/10.1090/bproc/203},
}

@ARTICLE{belmontisaksen2020rmotivic,
    AUTHOR = {Belmont, Eva and Isaksen, Daniel C.},
     TITLE = {{$\mathbb R$}-motivic stable stems},
   JOURNAL = {J. Topol.},
  FJOURNAL = {Journal of Topology},
    VOLUME = {15},
      YEAR = {2022},
    NUMBER = {4},
     PAGES = {1755--1793},
      ISSN = {1753-8416,1753-8424},
   MRCLASS = {14F42 (55Q45 55S10 55T15)},
  MRNUMBER = {4461846},
MRREVIEWER = {Anand\ Sawant},
       DOI = {10.1112/topo.12256},
       URL = {https://doi.org/10.1112/topo.12256},
}

@ARTICLE{belmontguillouisaksen2021c2,
    AUTHOR = {Belmont, Eva and Guillou, Bertrand J. and Isaksen, Daniel C.},
     TITLE = {{$C_2$}-equivariant and {$\mathbb{R}$}-motivic stable stems {II}},
   JOURNAL = {Proc. Amer. Math. Soc.},
  FJOURNAL = {Proceedings of the American Mathematical Society},
    VOLUME = {149},
      YEAR = {2021},
    NUMBER = {1},
     PAGES = {53--61},
      ISSN = {0002-9939,1088-6826},
   MRCLASS = {55Q45 (14F42 55Q91 55T15)},
  MRNUMBER = {4172585},
MRREVIEWER = {Niall\ Taggart},
       DOI = {10.1090/proc/15167},
       URL = {https://doi.org/10.1090/proc/15167},
}

@article{bhattacharya2024new,
      title={New infinite families in the stable homotopy groups of spheres}, 
      author={Prasit Bhattacharya and Irina Bobkova and J. D. Quigley},
      journal={To appear in Geom. Topol.},
      year={2026},
      note={arXiv:2404.10062},
      eprint={2404.10062},
      archivePrefix={arXiv},
      primaryClass={math.AT},
      url={https://arxiv.org/abs/2404.10062}, 
}

@Article{bousfieldcurtiskanquillenrectorschlesinger1966modp,
 Author = {Bousfield, A. K. and Curtis, E. B. and Kan, D. M. and Quillen, D. G. and Rector, D. L. and Schlesinger, J. W.},
 Title = {The mod-{{\(p\)}} lower central series and the {Adams} spectral sequence},
 FJournal = {Topology},
 Journal = {Topology},
 ISSN = {0040-9383},
 Volume = {5},
 Pages = {331--342},
 Year = {1966},
 Language = {English},
 DOI = {10.1016/0040-9383(66)90024-3},
 zbMATH = {3254672},
 Zbl = {0158.20502}
}

@book{brunermaymccluresteinberger1986hinfinity,
  title={{$H_\infty$} Ring Spectra and Their Applications},
  author={Bruner, R.R. and May, J.P. and McClure, J.E. and Steinberger, M.},
  isbn={9780387164342},
  lccn={86003845},
  series={Lecture Notes in Mathematics},
  year={1986},
  publisher={Springer Berlin Heidelberg}
}

@article {bruner2022adams,
    AUTHOR = {Bruner, Robert R. and Rognes, John},
     TITLE = {The {A}dams spectral sequence for the image-of-{$J$} spectrum},
   JOURNAL = {Trans. Amer. Math. Soc.},
  FJOURNAL = {Transactions of the American Mathematical Society},
    VOLUME = {375},
      YEAR = {2022},
    NUMBER = {8},
     PAGES = {5803--5827},
      ISSN = {0002-9947},
   MRCLASS = {55Q50 (55T15)},
  MRNUMBER = {4469237},
MRREVIEWER = {Steven R. Costenoble},
       DOI = {10.1090/tran/8680},
       URL = {https://doi.org/10.1090/tran/8680},
}

@book {brunerrognes2021adams,
    AUTHOR = {Bruner, Robert R. and Rognes, John},
     TITLE = {The {A}dams spectral sequence for topological modular forms},
    SERIES = {Mathematical Surveys and Monographs},
    VOLUME = {253},
 PUBLISHER = {American Mathematical Society, Providence, RI},
      YEAR = {[2021] \copyright 2021},
     PAGES = {xix+690},
      ISBN = {978-1-4704-5674-0},
   MRCLASS = {55N34 (18G40 55P43 55Q45 55T15)},
  MRNUMBER = {4284897},
MRREVIEWER = {Tilman\ Bauer},
       DOI = {10.1090/surv/253},
       URL = {https://doi.org/10.1090/surv/253},
}

@article {burklundhahnsenger2023boundaries,
    AUTHOR = {Burklund, Robert and Hahn, Jeremy and Senger, Andrew},
     TITLE = {On the boundaries of highly connected, almost closed
              manifolds},
   JOURNAL = {Acta Math.},
  FJOURNAL = {Acta Mathematica},
    VOLUME = {231},
      YEAR = {2023},
    NUMBER = {2},
     PAGES = {205--344},
      ISSN = {0001-5962,1871-2509},
   MRCLASS = {55P42 (55P43 57N65)},
  MRNUMBER = {4683371},
MRREVIEWER = {Geoffrey\ M. L. Powell},
       DOI = {10.4310/acta.2023.v231.n2.a1},
       URL = {https://doi.org/10.4310/acta.2023.v231.n2.a1},
}

@article {chang2025vanishingline,
    AUTHOR = {Chang, Kevin},
     TITLE = {A {$v_1$}-banded vanishing line for the {${\rm mod}\, 2$}
              {M}oore spectrum},
   JOURNAL = {Homology Homotopy Appl.},
  FJOURNAL = {Homology, Homotopy and Applications},
    VOLUME = {27},
      YEAR = {2025},
    NUMBER = {1},
     PAGES = {29--49},
      ISSN = {1532-0073,1532-0081},
   MRCLASS = {55P42 (55Q10)},
  MRNUMBER = {4870762},
MRREVIEWER = {Tilman\ Bauer},
       DOI = {10.4310/hha.2025.v27.n1.a3},
       URL = {https://doi.org/10.4310/hha.2025.v27.n1.a3},
}

@article {cohenlinmahowald1988adams,
    AUTHOR = {Cohen, Ralph L. and Lin, W\^{e}n Hsiung and Mahowald, Mark E.},
     TITLE = {The {A}dams spectral sequence of the real projective spaces},
   JOURNAL = {Pacific J. Math.},
  FJOURNAL = {Pacific Journal of Mathematics},
    VOLUME = {134},
      YEAR = {1988},
    NUMBER = {1},
     PAGES = {27--55},
      ISSN = {0030-8730,1945-5844},
   MRCLASS = {55T15 (55P42 55Q10)},
  MRNUMBER = {953499},
MRREVIEWER = {Vincent\ Giambalvo},
       URL = {http://projecteuclid.org/euclid.pjm/1102689365},
}

@article {davismahowaldimage1989,
    AUTHOR = {Davis, Donald M. and Mahowald, Mark},
     TITLE = {The image of the stable {$J$}-homomorphism},
   JOURNAL = {Topology},
  FJOURNAL = {Topology. An International Journal of Mathematics},
    VOLUME = {28},
      YEAR = {1989},
    NUMBER = {1},
     PAGES = {39--58},
      ISSN = {0040-9383},
   MRCLASS = {55Q50 (55T15)},
  MRNUMBER = {991098},
MRREVIEWER = {Donald\ W.\ Kahn},
       DOI = {10.1016/0040-9383(89)90031-1},
       URL = {https://doi.org/10.1016/0040-9383(89)90031-1},
}

@book{douglasfrancishenriqueshill2014topological,
  title      = {Topological Modular Forms},
  editor     = {Douglas, Christopher L. and Francis, John and Henriques, Andr\'e G. and Hill, Michael A.},
  series     = {Mathematical Surveys and Monographs},
  volume     = {201},
  year       = {2014},
  publisher  = {American Mathematical Society},
  address    = {Providence, RI},
  pages      = {318},
  isbn       = {978-1-4704-1884-7},
  doi        = {10.1090/surv/201},
  mrnumber   = {3223024}
}

@article {duggerisaksen2010motivic,
    AUTHOR = {Dugger, Daniel and Isaksen, Daniel C.},
     TITLE = {The motivic {A}dams spectral sequence},
   JOURNAL = {Geom. Topol.},
  FJOURNAL = {Geometry \& Topology},
    VOLUME = {14},
      YEAR = {2010},
    NUMBER = {2},
     PAGES = {967--1014},
      ISSN = {1465-3060,1364-0380},
   MRCLASS = {55T15 (14F42 55P42)},
  MRNUMBER = {2629898},
MRREVIEWER = {Markus\ Szymik},
       DOI = {10.2140/gt.2010.14.967},
       URL = {https://doi.org/10.2140/gt.2010.14.967},
}

@article {dugger2005atiyah,
    AUTHOR = {Dugger, Daniel},
     TITLE = {An {A}tiyah-{H}irzebruch spectral sequence for {$KR$}-theory},
   JOURNAL = {$K$-Theory},
  FJOURNAL = {$K$-Theory. An Interdisciplinary Journal for the Development,
              Application, and Influence of $K$-Theory in the Mathematical
              Sciences},
    VOLUME = {35},
      YEAR = {2005},
    NUMBER = {3-4},
     PAGES = {213--256 (2006)},
      ISSN = {0920-3036,1573-0514},
   MRCLASS = {19L64 (55S45 55T25)},
  MRNUMBER = {2240234},
MRREVIEWER = {G\'{e}rald\ Gaudens},
       DOI = {10.1007/s10977-005-1552-9},
       URL = {https://doi.org/10.1007/s10977-005-1552-9},
}

@article {duggerisaksen2017low,
    AUTHOR = {Dugger, Daniel and Isaksen, Daniel C.},
     TITLE = {Low-dimensional {M}ilnor-{W}itt stems over {$\mathbb{R}$}},
   JOURNAL = {Ann. K-Theory},
  FJOURNAL = {Annals of K-Theory},
    VOLUME = {2},
      YEAR = {2017},
    NUMBER = {2},
     PAGES = {175--210},
      ISSN = {2379-1683},
   MRCLASS = {14F42 (55Q45 55S10 55T15)},
  MRNUMBER = {3590344},
MRREVIEWER = {Oliver R\"{o}ndigs},
       DOI = {10.2140/akt.2017.2.175},
       URL = {https://doi.org/10.2140/akt.2017.2.175},
}

@article {eckmann1942gruppen,
    AUTHOR = {Eckmann, Beno},
     TITLE = {Gruppentheoretischer {B}eweis des {S}atzes von
              {H}urwitz-{R}adon \"{u}ber die {K}omposition quadratischer
              {F}ormen},
   JOURNAL = {Comment. Math. Helv.},
  FJOURNAL = {Commentarii Mathematici Helvetici},
    VOLUME = {15},
      YEAR = {1943},
     PAGES = {358--366},
      ISSN = {0010-2571,1420-8946},
   MRCLASS = {09.0X},
  MRNUMBER = {9936},
MRREVIEWER = {R.\ P.\ Agnew},
       DOI = {10.1007/BF02565652},
       URL = {https://doi.org/10.1007/BF02565652},
}

@article {guillouhillisaksenravenel2020cohomology,
    AUTHOR = {Guillou, Bertrand J. and Hill, Michael A. and Isaksen, Daniel
              C. and Ravenel, Douglas Conner},
     TITLE = {The cohomology of {$C_2$}-equivariant {$\mathcal{A}(1)$} and the
              homotopy of {${\rm ko}_{C_2}$}},
   JOURNAL = {Tunis. J. Math.},
  FJOURNAL = {Tunisian Journal of Mathematics},
    VOLUME = {2},
      YEAR = {2020},
    NUMBER = {3},
     PAGES = {567--632},
      ISSN = {2576-7658},
   MRCLASS = {14F42 (55Q91 55T15)},
  MRNUMBER = {4041284},
MRREVIEWER = {Bj\o rn Ian Dundas},
       DOI = {10.2140/tunis.2020.2.567},
       URL = {https://doi.org/10.2140/tunis.2020.2.567},
}

@incollection {greenlees1988borelhomology,
    AUTHOR = {Greenlees, J. P. C.},
     TITLE = {The power of mod {$p$} {B}orel homology},
 BOOKTITLE = {Homotopy theory and related topics ({K}inosaki, 1988)},
    SERIES = {Lecture Notes in Math.},
    VOLUME = {1418},
     PAGES = {140--151},
 PUBLISHER = {Springer, Berlin},
      YEAR = {1990},
   MRCLASS = {55T15 (55N25)},
  MRNUMBER = {1048182},
MRREVIEWER = {Donald M. Davis},
       DOI = {10.1007/BFb0083699},
       URL = {https://doi.org/10.1007/BFb0083699},
}

@article {greenlees1988stable,
    AUTHOR = {Greenlees, J. P. C.},
     TITLE = {Stable maps into free {$G$}-spaces},
   JOURNAL = {Trans. Amer. Math. Soc.},
  FJOURNAL = {Transactions of the American Mathematical Society},
    VOLUME = {310},
      YEAR = {1988},
    NUMBER = {1},
     PAGES = {199--215},
      ISSN = {0002-9947,1088-6850},
   MRCLASS = {55P42 (55T15)},
  MRNUMBER = {938918},
MRREVIEWER = {Donald\ W.\ Kahn},
       DOI = {10.2307/2001117},
       URL = {https://doi.org/10.2307/2001117},
}

@article {guillouisaksen2020bredon,
    AUTHOR = {Guillou, Bertrand J. and Isaksen, Daniel C.},
     TITLE = {The {B}redon-{L}andweber region in {$C_2$}-equivariant stable
              homotopy groups},
   JOURNAL = {Doc. Math.},
  FJOURNAL = {Documenta Mathematica},
    VOLUME = {25},
      YEAR = {2020},
     PAGES = {1865--1880},
      ISSN = {1431-0635,1431-0643},
   MRCLASS = {55Q91 (14F42 55Q45 55T15)},
  MRNUMBER = {4184454},
}

@misc{guillouisaksen2024c2,
      title={{$C_2$}-Equivariant Stable Stems}, 
      author={Bertrand J. Guillou and Daniel C. Isaksen},
      year={2024},
      note={arxiv: 2404.14627},
      eprint={2404.14627},
      archivePrefix={arXiv},
      primaryClass={math.AT},
      url={https://arxiv.org/abs/2404.14627}, 
}

@article {hahnshi2020real,
    AUTHOR = {Hahn, Jeremy and Shi, XiaoLin Danny},
     TITLE = {Real orientations of {L}ubin-{T}ate spectra},
   JOURNAL = {Invent. Math.},
  FJOURNAL = {Inventiones Mathematicae},
    VOLUME = {221},
      YEAR = {2020},
    NUMBER = {3},
     PAGES = {731--776},
      ISSN = {0020-9910},
   MRCLASS = {55P43},
  MRNUMBER = {4132956},
MRREVIEWER = {Drew Heard},
       DOI = {10.1007/s00222-020-00960-z},
       URL = {https://doi.org/10.1007/s00222-020-00960-z},
}

@article {hauschild1977aquivariante,
    AUTHOR = {Hauschild, H.},
     TITLE = {\"Aquivariante {H}omotopie. {I}},
   JOURNAL = {Arch. Math. (Basel)},
  FJOURNAL = {Archiv der Mathematik},
    VOLUME = {29},
      YEAR = {1977},
    NUMBER = {2},
     PAGES = {158--165},
      ISSN = {0003-889X,1420-8938},
   MRCLASS = {57E15 (55D40)},
  MRNUMBER = {467774},
MRREVIEWER = {T.\ tom Dieck},
       DOI = {10.1007/BF01220390},
       URL = {https://doi.org/10.1007/BF01220390},
}

@article {hillhopkinsravenel2016nonexistence,
    AUTHOR = {Hill, M. A. and Hopkins, M. J. and Ravenel, D. C.},
     TITLE = {On the nonexistence of elements of {K}ervaire invariant one},
   JOURNAL = {Ann. of Math. (2)},
  FJOURNAL = {Annals of Mathematics. Second Series},
    VOLUME = {184},
      YEAR = {2016},
    NUMBER = {1},
     PAGES = {1--262},
      ISSN = {0003-486X},
   MRCLASS = {55P91 (55N22 55P42 55Q45 55T15 55U35 57R15)},
  MRNUMBER = {3505179},
MRREVIEWER = {Paul G. Goerss},
       DOI = {10.4007/annals.2016.184.1.1},
       URL = {https://doi.org/10.4007/annals.2016.184.1.1},
}

@article {hukriz2001real,
    AUTHOR = {Hu, Po and Kriz, Igor},
     TITLE = {Real-oriented homotopy theory and an analogue of the
              {A}dams-{N}ovikov spectral sequence},
   JOURNAL = {Topology},
  FJOURNAL = {Topology. An International Journal of Mathematics},
    VOLUME = {40},
      YEAR = {2001},
    NUMBER = {2},
     PAGES = {317--399},
      ISSN = {0040-9383},
   MRCLASS = {55T15 (19L47 55N22 55N91 55P42 55P91)},
  MRNUMBER = {1808224},
MRREVIEWER = {J. P. C. Greenlees},
       DOI = {10.1016/S0040-9383(99)00065-8},
       URL = {https://doi.org/10.1016/S0040-9383(99)00065-8},
}

@article {hurwitz1922komposition,
    AUTHOR = {Hurwitz, A.},
     TITLE = {\"Uber die {K}omposition der quadratischen {F}ormen},
   JOURNAL = {Math. Ann.},
  FJOURNAL = {Mathematische Annalen},
    VOLUME = {88},
      YEAR = {1922},
    NUMBER = {1-2},
     PAGES = {1--25},
      ISSN = {0025-5831,1432-1807},
   MRCLASS = {99-04},
  MRNUMBER = {1512117},
       DOI = {10.1007/BF01448439},
       URL = {https://doi.org/10.1007/BF01448439},
}

@Article{isaksenwangxu2023stable,
 Author = {Isaksen, Daniel C. and Wang, Guozhen and Xu, Zhouli},
 Title = {Stable homotopy groups of spheres: from dimension 0 to 90},
 FJournal = {Publications Math{\'e}matiques},
 Journal = {Publ. Math., Inst. Hautes {\'E}tud. Sci.},
 ISSN = {0073-8301},
 Volume = {137},
 Pages = {107--243},
 Year = {2023},
 Language = {English},
 DOI = {10.1007/s10240-023-00139-1},
 zbMATH = {7692199}
}

@article {james1959stiefel,
    AUTHOR = {James, I. M.},
     TITLE = {Spaces associated with {S}tiefel manifolds},
   JOURNAL = {Proc. London Math. Soc. (3)},
  FJOURNAL = {Proceedings of the London Mathematical Society. Third Series},
    VOLUME = {9},
      YEAR = {1959},
     PAGES = {115--140},
      ISSN = {0024-6115,1460-244X},
   MRCLASS = {55.00},
  MRNUMBER = {102810},
MRREVIEWER = {W.\ S.\ Massey},
       DOI = {10.1112/plms/s3-9.1.115},
       URL = {https://doi.org/10.1112/plms/s3-9.1.115},
}

@incollection {lewis1987roggraded,
    AUTHOR = {Lewis, Jr., L. Gaunce},
     TITLE = {The {$R{\rm O}(G)$}-graded equivariant ordinary cohomology of
              complex projective spaces with linear {${\bf Z}/p$} actions},
 BOOKTITLE = {Algebraic topology and transformation groups ({G}\"{o}ttingen,
              1987)},
    SERIES = {Lecture Notes in Math.},
    VOLUME = {1361},
     PAGES = {53--122},
 PUBLISHER = {Springer, Berlin},
      YEAR = {1988},
      ISBN = {3-540-50528-8},
   MRCLASS = {55N91 (55N22)},
  MRNUMBER = {979507},
       DOI = {10.1007/BFb0083034},
       URL = {https://doi.org/10.1007/BFb0083034},
}

@misc{linwangxu2025lastkervaire,
      title={On the Last {K}ervaire Invariant Problem}, 
      author={Weinan Lin and Guozhen Wang and Zhouli Xu},
      year={2025},
      note={arxiv: 2412.10879},
      eprint={2412.10879},
      archivePrefix={arXiv},
      primaryClass={math.AT},
      url={https://arxiv.org/abs/2412.10879}, 
}

@misc{linwangxu2025machineproofs,
      title={Machine Proofs for {Adams} Differentials and Extension Problems among {CW} Spectra}, 
      author={Weinan Lin and Guozhen Wang and Zhouli Xu},
      year={2025},
      note={arxiv: 2412.10876}, 
      eprint={2412.10876},
      archivePrefix={arXiv},
      primaryClass={math.AT},
      url={https://arxiv.org/abs/2412.10876}, 
}

@article {lin1980on,
    AUTHOR = {Lin, Wen Hsiung},
     TITLE = {On conjectures of {M}ahowald, {S}egal and {S}ullivan},
   JOURNAL = {Math. Proc. Cambridge Philos. Soc.},
  FJOURNAL = {Mathematical Proceedings of the Cambridge Philosophical
              Society},
    VOLUME = {87},
      YEAR = {1980},
    NUMBER = {3},
     PAGES = {449--458},
      ISSN = {0305-0041},
   MRCLASS = {55Q10},
  MRNUMBER = {556925},
MRREVIEWER = {Donald M. Davis},
       DOI = {10.1017/S0305004100056887},
       URL = {https://doi.org/10.1017/S0305004100056887},
}

@article {lishiwangxu2019hurewicz,
    AUTHOR = {Li, Guchuan and Shi, XiaoLin Danny and Wang, Guozhen and Xu,
              Zhouli},
     TITLE = {Hurewicz images of real bordism theory and real
              {J}ohnson-{W}ilson theories},
   JOURNAL = {Adv. Math.},
  FJOURNAL = {Advances in Mathematics},
    VOLUME = {342},
      YEAR = {2019},
     PAGES = {67--115},
      ISSN = {0001-8708},
   MRCLASS = {55N22 (55Q91 55T15)},
  MRNUMBER = {3877362},
MRREVIEWER = {Hans-Werner Henn},
       DOI = {10.1016/j.aim.2018.11.002},
       URL = {https://doi.org/10.1016/j.aim.2018.11.002},
}

@misc{lili2026adams,
      title={The Adams differentials on the $e$-family}, 
      author={Runji Li and Yuxuan Li},
      year={2026},
      eprint={2602.20184},
      note={arxiv: 2602.20184},
      archivePrefix={arXiv},
      primaryClass={math.AT},
      url={https://arxiv.org/abs/2602.20184}, 
}

@article {ma2024geometricboundary,
    AUTHOR = {Ma, Sihao},
     TITLE = {A proof of the generalized geometric boundary theorem using
              filtered spectra},
   JOURNAL = {Topology Appl.},
  FJOURNAL = {Topology and its Applications},
    VOLUME = {355},
      YEAR = {2024},
     PAGES = {Paper No. 109006, 12},
      ISSN = {0166-8641,1879-3207},
   MRCLASS = {55T25 (55P43 55T05 55T15)},
  MRNUMBER = {4772506},
MRREVIEWER = {Sarah\ Whitehouse},
       DOI = {10.1016/j.topol.2024.109006},
       URL = {https://doi.org/10.1016/j.topol.2024.109006},
}

@article{ma2026borel,
      title={The {B}orel and genuine {$C_2$}-equivariant Adams spectral sequences}, 
      journal={To appear in Algebr. Geom. Topol.},
      author={Sihao Ma},
      year={2026},
      note={arxiv: 2208.12883},
      eprint={2208.12883},
      archivePrefix={arXiv},
      primaryClass={math.AT},
      url={https://arxiv.org/abs/2208.12883}, 
}

@article {mahowald1970orderimj,
    AUTHOR = {Mahowald, Mark},
     TITLE = {The order of the image of the {$J$}-homomorphisms},
   JOURNAL = {Bull. Amer. Math. Soc.},
  FJOURNAL = {Bulletin of the American Mathematical Society},
    VOLUME = {76},
      YEAR = {1970},
     PAGES = {1310--1313},
      ISSN = {0002-9904},
   MRCLASS = {55.40 (57.00)},
  MRNUMBER = {270369},
MRREVIEWER = {J.\ F.\ Adams},
       DOI = {10.1090/S0002-9904-1970-12656-8},
       URL = {https://doi.org/10.1090/S0002-9904-1970-12656-8},
}

@article {mahowald1982imageofj,
    AUTHOR = {Mahowald, Mark},
     TITLE = {The image of {$J$} in the {$EHP$} sequence},
   JOURNAL = {Ann. of Math. (2)},
  FJOURNAL = {Annals of Mathematics. Second Series},
    VOLUME = {116},
      YEAR = {1982},
    NUMBER = {1},
     PAGES = {65--112},
      ISSN = {0003-486X},
   MRCLASS = {55Q40},
  MRNUMBER = {662118},
MRREVIEWER = {Donald M. Davis},
       DOI = {10.2307/2007048},
       URL = {https://doi.org/10.2307/2007048},
}

@article {mahowaldravenel1993root,
    AUTHOR = {Mahowald, Mark E. and Ravenel, Douglas C.},
     TITLE = {The root invariant in homotopy theory},
   JOURNAL = {Topology},
  FJOURNAL = {Topology. An International Journal of Mathematics},
    VOLUME = {32},
      YEAR = {1993},
    NUMBER = {4},
     PAGES = {865--898},
      ISSN = {0040-9383},
   MRCLASS = {55Q45},
  MRNUMBER = {1241877},
MRREVIEWER = {Haynes\ R.\ Miller},
       DOI = {10.1016/0040-9383(93)90055-Z},
       URL = {https://doi.org/10.1016/0040-9383(93)90055-Z},
}

@article {may2020structure,
    AUTHOR = {May, Clover},
     TITLE = {A structure theorem for {$RO(C_2)$}-graded {B}redon
              cohomology},
   JOURNAL = {Algebr. Geom. Topol.},
  FJOURNAL = {Algebraic \& Geometric Topology},
    VOLUME = {20},
      YEAR = {2020},
    NUMBER = {4},
     PAGES = {1691--1728},
      ISSN = {1472-2747,1472-2739},
   MRCLASS = {55N91},
  MRNUMBER = {4127082},
MRREVIEWER = {Nansen\ Petrosyan},
       DOI = {10.2140/agt.2020.20.1691},
       URL = {https://doi.org/10.2140/agt.2020.20.1691},
}

@article {miller1981relations,
    AUTHOR = {Miller, Haynes R.},
     TITLE = {On relations between {A}dams spectral sequences, with an
              application to the stable homotopy of a {M}oore space},
   JOURNAL = {J. Pure Appl. Algebra},
  FJOURNAL = {Journal of Pure and Applied Algebra},
    VOLUME = {20},
      YEAR = {1981},
    NUMBER = {3},
     PAGES = {287--312},
      ISSN = {0022-4049,1873-1376},
   MRCLASS = {55T15 (55P42)},
  MRNUMBER = {604321},
MRREVIEWER = {J.\ F.\ Adams},
       DOI = {10.1016/0022-4049(81)90064-5},
       URL = {https://doi.org/10.1016/0022-4049(81)90064-5},
}

@article {radon1922lineare,
    AUTHOR = {Radon, J.},
     TITLE = {Lineare {S}charen orthogonaler {M}atrizen},
   JOURNAL = {Abh. Math. Sem. Univ. Hamburg},
  FJOURNAL = {Abhandlungen aus dem Mathematischen Seminar der Universit\"at
              Hamburg},
    VOLUME = {1},
      YEAR = {1922},
    NUMBER = {1},
     PAGES = {1--14},
      ISSN = {0025-5858,1865-8784},
   MRCLASS = {99-04},
  MRNUMBER = {3069384},
       DOI = {10.1007/BF02940576},
       URL = {https://doi.org/10.1007/BF02940576},
}

@article {sikora2022on,
    AUTHOR = {Sikora, Igor},
     TITLE = {On the {$RO(Q)$}-graded coefficients of
              {E}ilenberg-{M}ac{L}ane spectra},
   JOURNAL = {J. Homotopy Relat. Struct.},
  FJOURNAL = {Journal of Homotopy and Related Structures},
    VOLUME = {17},
      YEAR = {2022},
    NUMBER = {4},
     PAGES = {525--568},
      ISSN = {2193-8407,1512-2891},
   MRCLASS = {55P91 (20J05)},
  MRNUMBER = {4514125},
MRREVIEWER = {David\ Barnes},
       DOI = {10.1007/s40062-022-00314-x},
       URL = {https://doi.org/10.1007/s40062-022-00314-x},
}

@article {szymik2007equivariant,
    AUTHOR = {Szymik, Markus},
     TITLE = {Equivariant stable stems for prime order groups},
   JOURNAL = {J. Homotopy Relat. Struct.},
  FJOURNAL = {Journal of Homotopy and Related Structures},
    VOLUME = {2},
      YEAR = {2007},
    NUMBER = {1},
     PAGES = {141--162},
      ISSN = {1512-2891},
   MRCLASS = {55Q91 (55T15)},
  MRNUMBER = {2369156},
MRREVIEWER = {J.\ P. C. Greenlees},
}

@article {Voevodsky2003motivic,
    AUTHOR = {Voevodsky, Vladimir},
     TITLE = {Motivic cohomology with {${\bf Z}/2$}-coefficients},
   JOURNAL = {Publ. Math. Inst. Hautes \'{E}tudes Sci.},
  FJOURNAL = {Publications Math\'{e}matiques. Institut de Hautes \'{E}tudes
              Scientifiques},
    VOLUME = {98},
      YEAR = {2003},
     PAGES = {59--104},
      ISSN = {0073-8301,1618-1913},
   MRCLASS = {14F42 (12G05 19D45 19E15)},
  MRNUMBER = {2031199},
MRREVIEWER = {Eric\ M.\ Friedlander},
       DOI = {10.1007/s10240-003-0010-6},
       URL = {https://doi.org/10.1007/s10240-003-0010-6},
}
\bibliographystyle{alpha}
\endgroup

\end{document}